\documentclass[preprint,nopreprintline]{elsarticle}

\usepackage{etex}               %

\usepackage[bookmarks=true,colorlinks=true,linkcolor=blue]{hyperref}

\usepackage{xcolor}
\definecolor{jr@red}{RGB}{228,26,28}
\definecolor{jr@blue}{RGB}{55,126,184}
\definecolor{jr@green}{RGB}{77,175,74}
\definecolor{jr@purple}{RGB}{152,78,163}
\definecolor{jr@orange}{RGB}{255,127,0}
\definecolor{jr@yellow}{RGB}{255,255,51}
\definecolor{jr@brown}{RGB}{166,86,40}
\definecolor{jr@pink}{RGB}{247,129,191}
\definecolor{jr@gray}{RGB}{153,153,153}

\usepackage{amsmath}          %
\usepackage{amssymb}          %
\usepackage{mathrsfs}         %
\usepackage{mathtools}        %

\usepackage{soul}             %
\usepackage[normalem]{ulem}
\usepackage[margin=1in]{geometry}

\usepackage{caption}          %
\usepackage{makecell,multirow,multicol,booktabs} %

\usepackage[group-separator={,}]{siunitx}

\usepackage{tikz}
  \usetikzlibrary{calc,positioning}
\usepackage{pgfplots}
  \pgfplotsset{compat=newest}

\usepackage{hyperref}

\usepackage{amsthm}
\theoremstyle{plain}
\newtheorem{lemma}{Lemma}[section]

\newtheorem{proposition}[lemma]{Proposition}
\theoremstyle{definition}
\newtheorem{definition}[lemma]{Definition}

\theoremstyle{remark}
\newtheorem{remark}[lemma]{Remark}

\usepackage[section]{algorithm} %
\usepackage{algpseudocode}  %

\makeatletter%
\newcounter{algorithmicH} %
\let\oldalgorithmic\algorithmic
\renewcommand{\algorithmic}{%
  \stepcounter{algorithmicH} %
  \oldalgorithmic} %
\renewcommand{\theHALG@line}{ALG@line.\thealgorithmicH.\arabic{ALG@line}}
\makeatother%

\newlength{\tfwidth}
\newlength{\tfheight}
\newlength{\tfxa}
\newlength{\tfxb}
\newlength{\tfya}
\newlength{\tfyb}
\newcommand{\trimFigNoBox}[6]{%
\setlength\fboxsep{1pt}
\setlength\fboxrule{0.0pt}
\fbox{\includegraphics[width=#2, clip, trim=#3 #4 #5 #6]{#1}}%
}
\newcommand{\trimFigHeightWithBox}[6]{%
\setlength\fboxsep{0pt}%
\setlength\fboxrule{1.0pt}
\fbox{\includegraphics[height=#2, clip, trim=#3 #4 #5 #6]{#1}}%
}

\newcommand{\trimFig}[6]{%
\setlength{\tfwidth}{(#2+#2*\real{#3})+#2*\real{#4}}
\setlength{\tfheight}{(#2+#2*\real{#5})+#2*\real{#6}}%
\setlength{\tfxa}{\tfwidth*\real{#3}}%
\setlength{\tfxb}{\tfwidth*\real{#4}}%
\setlength{\tfya}{\tfheight*\real{#5}}%
\setlength{\tfyb}{\tfheight*\real{#6}}%
\trimFigNoBox{#1}{#2}{\tfxa}{\tfya}{\tfxb}{\tfyb}%
}

\newcommand{\trimhb}[6]{%

\sbox\figBox{\includegraphics{#1}}
\setlength{\tfwidth}{\the\wd\figBox}
\setlength{\tfheight}{\the\ht\figBox}
\setlength{\tfxa}{\tfwidth*\real{#3}}%
\setlength{\tfxb}{\tfwidth*\real{#4}}%
\setlength{\tfya}{\tfheight*\real{#5}}%
\setlength{\tfyb}{\tfheight*\real{#6}}%
\trimFigHeightWithBox{#1}{#2}{\tfxa}{\tfya}{\tfxb}{\tfyb}%
}
\renewcommand{\num}[2]{#1e#2} %

\newcommand{\gradient}{\nabla}

\newcommand{\partd}[2]{\frac{\partial #1}{\partial #2}}

\newcommand{\abs}[1]{\left\lvert#1\right\rvert}

\newcommand{\norm}[1]{\left\lVert#1\right\rVert}

\DeclareMathOperator{\erf}{erf}

\newcommand{\set}[1]{\ensuremath{\mathbb{#1}}}
\newcommand{\N}{\set{N}}

\newcommand{\R}{\set{R}}

\newcommand{\boldvec}[1]{\ensuremath{\mathbf{#1}}}

\ifdefined\vec
  \renewcommand{\vec}[1]{\boldvec{#1}}
\else
  \newcommand{\vec}[1]{\boldvec{#1}}
\fi

\newcommand{\dv}{\boldsymbol{d}}

\newcommand{\pv}{\boldsymbol{p}}

\newcommand{\vv}{\boldsymbol{v}}

\newcommand{\xv}{\boldsymbol{x}}

\newcommand{\Bv}{\boldsymbol{B}}

\newcommand{\indfn}{{\chi\smash[t]{\mathstrut}}}

\newcommand{\indmin}{\indfn_\mathrm{min}}
\newcommand{\indmax}{\indfn_\mathrm{max}}

\makeatletter
\newcommand{\editHighlighting}[3]{%
  \if@display%
  \textcolor{red!40!white}{\text{\sout{#1}} \textcolor{#2}{#3}}%
  \else%
  \textcolor{red!40!white}{\sout{#1} \textcolor{#2}{#3}}%
  \fi%
}
\makeatother
\newcommand{\editOne}[2]{\editHighlighting{#1}{blue!80!black}{#2}}
\newcommand{\editTwo}[2]{\editHighlighting{#1}{green!60!black}{#2}}
\newcommand{\editAll}[2]{\editHighlighting{#1}{orange!80!black}{#2}}
  \renewcommand{\editOne}[2]{#2}
  \renewcommand{\editTwo}[2]{#2}
  \renewcommand{\editAll}[2]{#2}
\newcommand{\editBOne}[2]{\editHighlighting{#1}{blue!80!black}{#2}}
  \renewcommand{\editBOne}[2]{#2}
\newcommand{\editCOne}[2]{\editHighlighting{#1}{blue!80!black}{#2}}
  \renewcommand{\editCOne}[2]{#2}

\newif\ifshowstatus
\showstatustrue %

\begin{document}

\begin{frontmatter}

\title{%
  Scalable Implicit
  Solvers with Dynamic Mesh Adaptation for a
  Relativistic Drift-Kinetic Fokker--Planck--Boltzmann Model%
  \tnoteref{mytitlenote}}
\tnotetext[mytitlenote]{This work was jointly supported by the
  U.S. Department of Energy through the Fusion Theory Program of the
  Office of Fusion Energy Sciences and the SciDAC partnership on
  Tokamak Disruption Simulation between the Office of Fusion Energy
  Sciences and the Office of Advanced Scientific Computing; and
  through the FASTMath Institute. Los Alamos
  National Laboratory is operated by Triad National Security, LLC, for
  the National Nuclear Security Administration of U.S. Department of
  Energy (Contract No.~89233218CNA000001). Argonne National Laboratory
  is operated under contract DE-AC02-06CH11357.
}

\author[anl,vt]{Johann Rudi\corref{cor1}}
\ead{jrudi@vt.edu}

\author[anl,vt,bu]{Max Heldman}
\ead{maxh@vt.edu}

\author[anl]{Emil M. Constantinescu}
\ead{emconsta@mcs.anl.gov}

\author[lanl]{Qi Tang\corref{cor1}}
\ead{qtang@lanl.gov}

\author[lanl]{Xian-Zhu Tang}
\ead{xtang@lanl.gov}

\address[anl]{Mathematics and Computer Science Division, Argonne National Laboratory, Lemont, IL 60439.}
\address[vt]{Department of Mathematics, Virginia Tech, Blacksburg, VA 24061.}
\address[bu]{Department of Mathematics and Statistics, Boston University, Boston, MA 02215.}
\address[lanl]{Theoretical Division, Los Alamos National Laboratory, Los Alamos, NM 87545.}

\cortext[cor1]{Corresponding authors}

\begin{abstract}
In this work we consider a relativistic drift-kinetic model for
runaway electrons along with a Fokker--Planck operator for small-angle
Coulomb collisions, a radiation damping operator, and a secondary
knock-on (Boltzmann) collision source.
We develop a new
{scalable fully implicit solver utilizing finite volume and}
conservative finite
difference schemes and dynamic mesh adaptivity. A new data management
framework in the PETSc library based on the p4est library is developed to enable
simulations with dynamic adaptive mesh refinement (AMR),
\editAll{parallel computation, and load balancing.  This framework is tested
through the development of the runaway electron solver that is able to}
{distributed memory parallelization, and dynamic load balancing of computational
work.  This framework and the runaway electron solver building on the framework
are able to}
dynamically capture both bulk Maxwellian at the
low-energy region and a runaway tail at the high-energy region.
To effectively capture features via the AMR algorithm, a new AMR
indicator prediction strategy is proposed that is performed alongside the
implicit time evolution of the solution.  This strategy is complemented by the
introduction of computationally cheap feature-based AMR indicators that are
analyzed theoretically.
Numerical results quantify the advantages of the prediction strategy in better
capturing features compared with nonpredictive strategies; and we demonstrate
trade-offs regarding computational costs.
\editAll{The full solver is further verified}
{The robustness with respect to model parameters, algorithmic scalability,
and parallel scalability are demonstrated}
through several benchmark problems including manufactured solutions and
solutions of different physics models.  We focus on demonstrating the advantages
of using implicit time stepping and AMR for runaway electron simulations.
\end{abstract}

\end{frontmatter}

\def\secstatus{finished}

\section{Introduction}
\label{sec:intro}

Plasma disruptions and their mitigation strategies have been one of
the most active research areas in tokamak fusion in recent
years. Studying runaway electrons, which are a major cause for
intolerable machine damages by tokamak disruptions, is of great
interest.  The runaway electrons are typically considered in the
guiding center formulation, which has been averaged over the
gyromotion. The resulting model is a relativistic
Fokker--Planck--Boltzmann model when both small-angle and large-angle
Coulomb collisions are accounted for.  For more details on the
physics, see the recent review
papers~\cite{boozer-pop-2015,breizman2019physics} and references therein.

One of the main interests in studying the runaway electrons is to
understand the impact of different collisions (large-angle,
small-angle, etc.) and radiation damping under the dominant electric
field along the magnetic field direction.  The competition between
parallel electric field acceleration and collisional/radiative drag
sets the advective term, albeit in momentum space of the relativistic
drift-kinetic equation, while the collisional energy diffusion and
pitch-angle scattering appear in the form of diffusion operators in
energy and pitch of the momentum space. There is a natural time-scale
separation between the advective and diffusive time scales, with the
latter much longer than the former at suprathermal energy.  Therefore
it is critical to step over the advective scale due to parallel
electric field acceleration of the runaway electrons and study the
collision or diffusion scale. As a result, implicit time stepping is
typically used.

The existing work based on a continuous representation
of distribution functions includes conservative finite difference and
finite volume schemes~\cite{GuoMcDevittTang2017, stahl2017norse,
hesslow2018effect, daniel2020fully}. They typically solve the
distribution function in the 2D momentum space.  Several
numerical challenges are associated with this model. First, the advection
due to the parallel electric field is dominant in  most  of the
domain, but the diffusion coefficient is nonuniform and becomes larger
when getting close to the thermal velocity. Since we are generally
interested in the impact of collisions, a desired algorithm should
step over the fast advection time scale. Second, the distribution
function becomes very anisotropic in pitch and has a nontrivial fat (in energy) tail as
time evolves, which is commonly referred to a ``runaway tail.''  A
corresponding treatment in meshing is necessary to resolve the
highly anisotropic (in pitch) tail. Third, the distribution function can span over 10 or more orders of
magnitude. An accurate numerical algorithm is needed to resolve such a
large range of values.  This work  addresses all these
issues.  For the first issue, we will rely on implicit time stepping,
nonlinear solvers, and algebraic multigrid preconditioners. The second
issue will be resolved through developing an adaptive mesh refinement (AMR) capability, which is a
major focus of the current work. The third issue will be resolved
through the accurate numerical schemes in both space and time as well
as a modeling choice to chop off a large portion of a Maxwellian close
to the thermal velocity.
In both fluid and kinetic models, AMR plays a key role in fusion
device modeling because of localized interesting structures such as
thin current sheets, plasmoids, and very anisotropic
distributions. Many of the previous related studies focus on fluid
plasma models such as extended magnetohydrodynamics (MHD) or its
equilibrium (see, e.g.,~\cite{strauss1998adaptive, philip2008implicit,
baty2019finmhd, peng2020adaptive, tang2022adaptive} and the references
therein), while there are only a limited number of studies for
continuous kinetic simulations~\cite{hittinger2013block,
adams2017landau, wettervik2017relativistic}.  \editBOne{}{ In the
broader context of statistical physics for a variety of applications,
AMR has found usage in improving the computational efficiency since it
``allows resolving important regions of phase space where the
particles are present and reduce the number of cells in the regions
with no particles.''\cite{arslanbekov2013kinetic}.  In
Ref.~\cite{kolobov2019boltzmann}, a 1D1V electron kinetic solver with
AMR was explicitly demonstrated for the thermal electron runaway
acceleration, also known as Dreicer acceleration, in what amounts to
be an electrostatic glow discharge for low-temperature plasma
generation.  The prospective advantage of AMR in mesh-base kinetic
solvers, as pointed out in these previous work, holds even greater
promise and importance in the current problem of Ohmic-to-runaway
current conversion in a magnetized plasma.  This is because the strong
energy dependence of plasma collisionality introduces a distinct set
of numerical challenges that AMR must be tailored to overcome in this
paper.  First, the energy range of the electrons can easily span 6--7
orders of magnitude between thermal electrons (eVs) and the runaway tail ($>$ MeV),
which produces huge variations in plasma collisionality.  Second,
extreme anisotropy in the momentum space, in the electron pitch $\xi,$
is increasingly aggravated for higher electron energies.  Third, the
runaway tail distribution, which is the quantity of interest for this
type of problems, can be 12 orders of magnitude, or more, lower than
the bulk electron density, but the resolution of this tail is critical to
evaluate the primary runaway seeds for the follow-up avalanche
growth.  Much of our considerations in AMR adaptation strategy,
spatial and temporal discretization, and the necessary implicit
solvers for vastly varying plasma collisionality as a function of
electron energy, is geared toward overcoming these extreme numerical
challenges.}

One of the most important pieces in any of those AMR algorithms is an
indicator to determine where to refine and coarsen the mesh. There are
two types of AMR indicators in the literature: error-estimator-based
and feature-based. For instance, an indicator based on the features in
the current and vorticity is proposed in the adaptive MHD
solver~\cite{philip2008implicit}, while a weighted sum of several
finite element error estimators is used as the indicator for the same
resistive MHD system in~\cite{tang2022adaptive}. Other common
error-estimator-based approaches for finite difference/finite volume
include Richardson extrapolation~\cite{berger1989local}, solution
smoothness via derivatives~\cite{Constantinescu_A2008a}, and
goal-oriented strategies~\cite{offermans2022error}. These probe
certain aspects of the solution error but can become expensive if one
considers \editAll{less}{} frequent regridding.
Another common issue is that the AMR indicators are based on the
present solution, which results in adaptive meshes being lagged
compared with features of interest because regridding often
happens every ten or more time steps.
For another class of problems in wave propagation, explicit time stepping is
computationally attractive; and in this explicit case one time step with
explicit Euler can be taken \cite{GuoCheng17} to approximate mesh refinement
one time step in the future.
In this work, where we focus on implicit time stepping, we explore several
different types of AMR indicators and propose predicting these AMR indicators
for multiple time steps.

\editAll{}{
The computational challenges of solving for the runaway electron distribution
from relativistic Fokker--Planck and Fokker--Planck--Boltzmann
equations---despite being linear partial differential equations (PDEs)---are:
(i) the ability to resolve the dynamically evolving runaway distribution both
near the runaway vortex regime and at a parallel electric field that is
significantly beyond the onset of the runaway vortex;
(ii) the solution being a small but positive function that varies 10--20 orders of
magnitude in momentum space;
(iii) the competing effects of advection- and diffusion-dominated regimes that
are shifting throughout the momentum space.
}

\editAll{}{
Challenges (i) and (ii) are addressed with an AMR discretization to resolve
the dynamical structure of the electron distribution in momentum space; while
(iii) is tackled with an algorithmically scalable solver for implicit time
stepping on the AMR discretization.
}
\editAll{}{Specifically}, we develop adaptive, scalable, fully implicit solvers
for a time-dependent relativistic drift-kinetic Fokker--Planck--Boltzmann model in
the momentum space.  The octree-based AMR
package {\tt p4est}~\cite{BursteddeWilcoxGhattas11}
is used as the basic framework to provide a parallel dynamic AMR
algorithm.
The numerical schemes are based on conservative finite difference or
finite volume algorithms, such as, the QUICK and MUSCL schemes
\cite{LeVeque02, Toro09, Leonard79}.  The solvers are
developed under the {\tt PETSc} framework \cite{petsc-user-ref}, and
our focus is fully implicit time stepping, nonlinear solvers, and
preconditioning strategies.  Diagonally implicit
Runge--Kutta (DIRK)~\cite{alexander1977diagonally} or explicit first-stage singly diagonally implicit Runge--Kutta (ESDIRK)
integrators \cite{giraldo2013implicit} %
are chosen. The nonlinear solver is based on a Jacobian-free
Newton--Krylov (JFNK) method \cite{knoll2004jacobian}, which is
preconditioned by a finite difference coloring Jacobian and
hypre's algebraic multigrid preconditioner \cite{hypre,
FalgoutYang02}.  While developing the adapative solvers, we present and analyze
several feature-based AMR indicators for the distribution function.
We further propose predicting these AMR indicators for multiple time steps, which
can be done at a fraction of the computational cost when reducing the original
advection--diffusion PDE to a pure-advection PDE.  We quantify the benefits of
predictive AMR on the resolution of distribution functions and we describe the
limitations when trading off computational costs for accuracy; this will be
done with extensive numerical experiments.
\editOne{}{
We demonstrate the robustness of our numerical solvers and preconditioners for
the computationally challenging runaway electron models within a simulation with
dynamic AMR.  We achieve problem-robust and algorithmically scalable methods.
Additionally, we document the possibilities and limitations for parallel
scalability, where the existence of limitations is inherent to any implicit scheme.
}

The rest of the paper is organized as follows.
Section~\ref{sec:governing-eqns} gives an overview of the relativistic
Fokker--Planck--Boltzmann model. The discussions focus on the forms of the
collision operator, synchrotron radiation, and the full equation in
the momentum space.
Section~\ref{sec:amr} describes \editTwo{AMR algorithms and several AMR
indicators}{the key contributions related to AMR algorithms}.
In particular, a prediction approach for AMR indicators is proposed.
Section~\ref{sec:solver} presents the details of \editTwo{the}{further key
contributions: the robust and algorithmically scalable} numerical schemes and
their implementation in our code framework.
Section~\ref{sec:setup} focuses on the experimental setup and reproducibility.
We describe several manufactured solutions for accuracy tests and some details
on computational environments.  We also discuss the impact of the field and
knock-on collision operators.
Section~\ref{sec:results} presents several numerical examples, which include
\editTwo{both algorithmic and}{algorithmic scalability, parallel scalability,
and a} physics study.
The conclusions and discussion of future work are given in
Section~\ref{sec:conclusion}.

\def\secstatus{finished}

\section{Governing equations}
\label{sec:governing-eqns}

We study the phase-space dynamics of a runaway electron distribution
and consider a simplified version of the relativistic drift-kinetic
Fokker--Planck--Boltzmann equation~\cite{brizard1999nonlinear}. The
governing equations we are considering are effectively described in
the spherical coordinates of $(p, \xi, \theta)$, with $p$ the
normalized momentum magnitude, $\xi = p_\parallel/p$ the pitch,
and $\theta$ the azimuthal angle.  Here $p$ is dimensionless
with normalization by $m_e c$, with $m_e$ the rest mass
of an electron and $c$ the speed of light, and $\xi$ is the cosine
of the polar angle in the standard spherical coordinate.  The parallel
direction,  denoted by subscript $\parallel,$ aligns with the
magnetic field.  We further follow the guiding center model and assume
the solution has azimuthal symmetry, which reduces the momentum space
to two dimensions.  This corresponds to a momentum-space domain of
$[0, +\infty]\times[-1, 1]$.  The normalized relativisitic
Fokker--Planck--Boltzmann (RFP) equation to describe runaway electron
distribution, $f(t, p, \xi),$ in a slab geometry, is given by
\begin{align}
\label{eqn:rfp}
\frac{\partial f}{\partial t} - E \left( \xi \frac{\partial f}{\partial p} + \frac{1-\xi^2}{p} \frac{\partial f}{\partial \xi} \right) = C(f) + \alpha R(f) + S(f)\,,
\end{align}
where $E$ stands for the normalized electric field parallel to the
magnetic field in a tokamak and its normalization scale is
$E_c \coloneqq m_e c/ e \tau_c$ with $e$ the elementary charge
and $\tau_c$ the time scale.  Here the time scale of the
equation, $\tau_c$, is the relativistic electron collision time given
by
\begin{align}
\tau_c \coloneqq \frac{ 4\pi \epsilon_0^2 m_e^2 c^3}{e^4 n_e \ln \Lambda},
\end{align}
where $n_e$ is the (background) thermal electron density, $\epsilon_0$
is the electrical permittivity, and $\ln \Lambda$ is the Coulomb
logarithm.  The so-called Connor--Hastie $E_c$ sets the critical value for
runaway electron generation~\cite{connor1975relativistic}.  In the
current study all those values are given as constant based on
practical devices such as ITER. On the right-hand side are
several secondary effects due to different types of collisions and
synchrotron radiation (defined in detail later), and $\alpha$ is
introduced to describe the intensity of radiation damping. All the physical
quantities appearing in the current work have been summarized
in~\ref{sec:quantities}.

This problem has been well studied by plasma
physicists~\cite{GuoMcDevittTang2017, stahl2017norse,
hesslow2018effect, Guo-etal-PoP-2019, daniel2020fully}, commonly using conservative
finite difference or finite volume schemes on a stretched grid. The
dynamics of the equation is primarily dominated by the advection term
due to the strong electric field. The distribution function has a
sharp boundary layer at $p = p_{\min}$ due to the boundary condition
(defined in detail later); and as time evolves, the distribution
accumulates around $\xi = -1$ if $E>0$ (or $\xi = 1$ if $E<0$) due to
the advection. The distribution is further impacted by the collisions
and the source term in the large $p$ region and eventually forms a
nontrivial tail structure that is of the greatest interest. A majority
of the previous work uses the stretched grid to resolve the localized
distribution and tails. Adaptive mesh refinement has not been
explored on this problem to the best of our knowledge.

We note that in~\cite{daniel2020fully} a different coordinate of
$(p_\parallel, p_\perp)$ is used.
Choosing the coordinate of $(p, \xi)$ for the RFP equation has several advantages. The
solution is highly localized, as found in~\cite{GuoMcDevittTang2017,
stahl2017norse, hesslow2018effect,Guo-etal-PoP-2019, daniel2020fully}. It is more
efficient to resolve the same structure in $(p, \xi)$ while the
coordinate of $(p_\parallel, p_\perp)$ leads to half of the domain
resolving a trivial distribution tail.  Note that the physics of
interest in this study is the small population of runaway electrons.
It is therefore much easier to resolve this small population tail in
$(p, \xi)$ by chopping off the large Maxwellian bulk close to $p = 0$.
\editBOne{One disadvantage of using $(p, \xi)$ is the coordinate singularity
when $p_{\min}\rightarrow 0,$ and the fine grid spacing in $\xi$ near
the $\xi=-1$ boundary where runaway electrons of small $p_\perp=\xi p$
are needed to be resolved for large $p.$ Here}{We often} set $p_{\min}$ to
correspond to a few times the thermal energy of the background
electrons, and we impose an approximate Dirichlet boundary condition given
by the Maxwellian bulk, which is commonly adopted in the physics
codes~\cite{GuoMcDevittTang2017, stahl2017norse, hesslow2018effect,
Guo-etal-PoP-2019, daniel2020fully} and will be detailed later.
\editBOne{}{We also provide a test that extends the boundary condition to $p=0$
through a compatible boundary condition.}

\subsection{Coulomb collision}
In this study the bulk distribution is considered to be close to a
Maxwellian with a thermal speed
much less than the light speed, namely, $v_t/c \ll 1$, where $v_t$
stands for the thermal velocity of the background electrons.  The
(small-angle) Coulomb collision of the runaway electrons, which is of
much lower density, with the
background electrons, $C(f),$ can be approximated by the so-called
test-particle collision operator~\cite{GuoMcDevittTang2017,
stahl2017norse}. We use the test-particle collision operator proposed
in~\cite{papp2011runaway} that is applicable to both thermal
($p\sim v_t/c \ll 1$) and relativistic ($p > 1$) electrons.  This
Coulomb collision operator has the form
\begin{align}
\label{eq:collision}
  C(f) =
    \frac{1}{p^2}\partd{}{p}\left[p^2\left(C_F f + C_A\partd{f}{p}\right)\right] +
    \frac{C_B}{p^2}\partd{}{\xi}\left[\left(1-\xi^2\right)\partd{f}{\xi}\right],
\end{align}
where all the collision coefficients are $p$-dependent and given by
\begin{align*}
  C_F &= \frac{2 c^2}{v_t^2}\Psi(x),\\
  C_A &= \frac{\sqrt{1+p^2}}{p}\Psi(x),\\
  C_B &= \frac{\sqrt{1+p^2}}{2p}
         \left(Z + \erf(x) - \Psi(x) + \frac{v_t^2} {2 c^2} \, \frac{p^2}{1+p^2}\right).
\end{align*}
Here $x$ is $p$-dependent,
\[
  x = \frac{c}{v_t} \, \frac{p}{\sqrt{1+p^2}},
\]
$Z$ stands for a charge number, and $\Psi(x) = [\erf(x) -
x \erf'(x)]/(2x^2)$ is the Chandrasekhar function.  A study of the
collision coefficients in a transition from nonrelativistic to
relativistic regimes is presented in~\cite{GuoMcDevittTang2017}, where the authors show   that the collision coefficients at small $p$ are several
orders of magnitude larger than those at larger $p$ and thus the
diffusion becomes dominant compared with the advection field.
Therefore, a fully implicit treatment of the collision operator is
necessary.  When needed, the collision coefficients are also modified
to include the partial screening effect~\cite{hesslow2018effect}, which
addresses the role of bound electrons in collisions between fast
electrons and partially ionized impurities.  Mathematically, it
corresponds to modifying the $C_B$ and the Coulomb logarithm
$\ln \Lambda$, but it does not involve additional numerical
challenges.  For more details, see~\cite{hesslow2017effect,
hesslow2018effect}.

\subsection{Synchrotron radiation}
When there is an acceleration of a charged particle,
there is an energy loss through the emitted
electromagnetic radiation.
In a dominating magnetic field where a charged particle follows cyclotron motion,
relativistic electrons can emit a significant amount of synchrotron radiation.  The
resulting radiation damping operator satisfies
\begin{align}
\label{eq:radiation}
  R(f) = \frac{1}{p^2}\partd{}{p}\left[p^3\gamma(1-\xi^2)f\right] -
               \partd{}{\xi}\left[\frac{\xi(1-\xi^2)}{\gamma}f\right],
\end{align}
where $\gamma=\sqrt{1+p^2}$ is the Lorentz factor.  The dimensionless
parameter for the intensity of damping is formally defined as
$\alpha\coloneqq\tau_c/\tau_s$, a ratio between the collision time
scale and the synchrotron radiation time scale, where
$\tau_s \coloneqq 6\pi \epsilon_0 m_e^3 c^3/e^4 B^2$ and $B$ is the
magnitude of the magnetic field.  Typically it is about 0.1 to 0.3 for
ITER-like tokamaks and 0.001 to 0.05 for a smaller device, which is
the range we use in our simulations.

\subsection{Knock-on collision}
\label{sec:knock_on}
Another important source in runaway electron generations is the
avalanche mechanism due to knock-on collisions between runaway
electrons and the background thermal
electrons~\cite{Rosenbluth_1997,chiu1998fokker}.  The knock-on source
considers the secondary large-angle collisions in addition to the
primary small-angle collisions of the Fokker--Planck
operator~\cite{boozer-pop-2015,breizman2019physics}.  The knock-on
source is
\begin{align}
S(p, \xi) = S_1(p, \xi) + S_2(p,\xi),
\label{eqn:chiu_full}
\end{align}
which consists of the secondary electron source term
\begin{align}
S_1(p, \xi) = \int_{-1}^1 \int_0^\infty \Pi(\gamma', \xi'; \gamma, \xi) f(p', \xi') 2 \pi p'^2 dp' \, d\xi',
\end{align}
and the annihilation term
\begin{align}
S_2(p, \xi) = -f(p, \xi)  \int_{-1}^1 \int_0^\infty \Pi(\gamma, \xi; \gamma', \xi') 2 \pi p'^2 dp' \, d\xi',
\end{align}
with $\gamma = \sqrt{1+p^2}$ and $\gamma' = \sqrt{1+p'^2}$.

In this work we use the Chiu model in~\cite{chiu1998fokker} and approximate the scattering rate function in $S_1$ by
\begin{align}
\Pi(\gamma', \xi'; \gamma, \xi) \approx \frac 1 {\ln \Lambda} \frac{p'^2}{2\pi p^2 \, |\xi|} \frac{d \sigma (\gamma', \gamma)}{dp} \, \delta(p' - p^*),
\end{align}
where $\delta(\cdot)$ is a Dirac delta function that indicates only electrons with particular energy at $p^*$ can generate the required secondary electron at a given point $(p, \xi)$
with $p^* = \sqrt{(\gamma^*)^2-1}$ and $\gamma^* = \frac{(\gamma+1)/(\gamma-1) \, \xi^2 + 1}{(\gamma+1)/(\gamma-1) \, \xi^2 - 1}$.
The M{\o}ller cross section for large-angle collisions is used,
\begin{align}
 \frac{d \sigma (\gamma', \gamma) }{dp} = \frac{p}{\gamma}\, \frac{2\pi\gamma'^2 }{(\gamma'-1)^3(\gamma'+1)}\left[
 x^2 -3 x + \left( \frac{\gamma'-1}{\gamma'}\right)^2 (1+x) \right],
\end{align}
where
\begin{align*}
x = \frac 1{\nu(1-\nu)},\quad \text{and} \quad \nu = \frac{\gamma-1}{\gamma'-1}.
\end{align*}
Note that the M{\o}ller cross section is turned on only if $\gamma'\ge 2\gamma-1$;  otherwise $S_1$ is set to  0, which is commonly imposed
(see~\cite{mcdevitt2019avalanche}, for instance).
One can easily show that the source term $S_1$ is  positive only in a thin region of $\xi \in [-\sqrt{\gamma/(\gamma+1)}, -p/(\gamma+1)]$.
An adaptive mesh is beneficial to resolve this thin region.
The annihilation term is integrated in the range of $\gamma' \in [\gamma_0, \frac{\gamma+1}2]$,
where $\gamma_0$ is computed from the computation domain, namely, $\gamma_0 = \sqrt{1+p_{\min}^2}$.
The M{\o}ller cross section is well defined in such a range.
This leads to
\begin{align}
S_2(p,\xi) = -\frac 1{\ln\Lambda} f(p, \xi) \, \sigma(\gamma, \gamma_{0}),
\end{align}
where $\sigma(\gamma, \gamma_{0})$ is the integrated form of the M{\o}ller cross section, given by
\begin{align*}
\sigma(\gamma, \gamma_{0}) = \frac{2\pi}{\gamma^2-1}\left[ \frac{\gamma+1}2 -\gamma_{0} - \gamma^2\left(\frac{1}{\gamma-\gamma_{0}} - \frac 1{\gamma_{0}-1} \right) + \frac{2\gamma-1}{\gamma-1}\ln \frac{\gamma_{0}-1}{\gamma-\gamma_{0}} \right]
\end{align*}
for $\gamma\ge2\gamma_0-1$ or otherwise $\sigma(\gamma, \gamma_0) = 0$.

\subsection{Computational domain and boundary conditions}
Since the focus of this study is on relativistic runaway electrons where
the bulk of the distribution is still assumed to be a Maxwellian, it
is reasonable to use a computational domain: $[p_{\min},
p_{\max}] \times [-1, 1]$, where $p_{\min}$ is chosen as being thermal
(e.g., let $p_{\min} = v_t/c$ or a few times of $v_t/c$).  One can
further assume that the left boundary condition at $p = p_{\min}$ is
Dirichlet and given by the Maxwell--J\"uttner distribution, which is
uniform in $\xi$. At the right boundary of $p = p_{\max}$, a Neumann
boundary condition $\partial f/\partial n = 0$ is prescribed, so one
must check the solution to make sure that $p_{\max}$ is high enough that
$f$ is negligibly small there.  At the top and bottom boundaries, because of the involvement of the coefficient $1-\xi^2$, all the fluxes along
$\xi$ become 0, which is in fact the natural outcome of the coordinate
chosen. Therefore, Neumann boundary conditions are used at the top and
bottom boundaries.

\subsection{Conservation and an alternative RFP form}
The determinant of the Jacobian  corresponding to the transformation
between $(p_\parallel, \vec{p}_\perp)$ to $(p, \xi)$, under azimuthal
symmetry,  is given by $J=p^2$ (see \ref{sec:coordinate}).
Thus, the definition of the divergence of a vector $\vv = (v_1, v_2)$ in $(p,\xi)$ becomes
\[
\nabla \cdot \vv \coloneqq
  \frac 1 {p^2} \Big[ \frac{\partial}{\partial p}(p^2 v_1) +
  \frac{\partial}{\partial\xi} (p^2 v_2) \Big].
\]
One can
\editAll{easily}{}
show that the electric field ``advection'' term
in~\eqref{eqn:rfp}, $\big(-E \xi, - E ({1-\xi^2})/{p} \big)$, is
divergence-free. This is in fact more obvious in the coordinate of
$(p_\parallel, \mathbf{p}_\perp)$, as the RFP equation in
$(p_\parallel, \mathbf{p}_\perp)$ becomes
\[
\frac{\partial f}{\partial t} - E \frac{\partial f}{\partial p_\parallel}= C(f) + \alpha R(f) + S(f)
\]
and $E$ is a constant.
\editAll{It is also easy to see that}{Moreover,}
$f p^2$ is a conserved value in the coordinate of $(p, \xi)$, which is consistent with the fact that the distribution function is conservative in $(p_\parallel, \mathbf{p}_\perp)$. As a result, a
\editOne{finite volume}{conservative finite difference}
method is chosen in this study to satisfy the conservation of the distribution.

\editAll{In our implementation}{Following the previous work~\cite{GuoMcDevittTang2017}}, we choose to solve $\tilde f \coloneqq fp$. Then we can rewrite the RFP equation~\eqref{eqn:rfp} as
\begin{align}
\label{eqn:conser2}
  \partd{\tilde f}{t} +
  \frac{1}{p}\partd{\Gamma_p(\tilde f)}{p} +
  \partd{\Gamma_\xi(\tilde f)}{\xi}
  =   S(\tilde f),
\end{align}
where the fluxes are
\begin{align}
\label{eqn:conser2-flux-p}
  \Gamma_p(\tilde f) &=
    -p \left[ E\xi +\alpha p \gamma (1-\xi^2) \right] \tilde f -
    (p C_F - C_A) \tilde f - p C_A \partd{\tilde f}{p},\\
\label{eq:conser2-flux-xi}
  \Gamma_\xi(\tilde f) &=
    -(1-\xi^2) \left[ \left(\frac E p - \frac{\alpha  \xi}{\gamma} \right) \tilde f +
                      \frac{C_B}{p^2}\partd{\tilde f }{\xi} \right].
\end{align}
This conservative form is a choice of our implementation, \editAll{}{which eases the head-to-head comparison between the current implementation with our previous one~\cite{GuoMcDevittTang2017}.
This choice is not critical for the success of the proposed AMR algorithm.}
\editAll{In  previous work [3], it is shown that solving for $\tilde f$ helps diminish the negativity of the distribution function in the tail region.}{}

\def\secstatus{finished}

\section{Adaptive mesh refinement: Indicators and their prediction in time}
\label{sec:amr}

\editAll{Indicators are flags that correspond to each grid point and determine}
{Indicators are spatial fields that determine at each grid point}
whether to refine or coarsen the mesh.
This section presents the indicators for adaptivity and shows how indicators are
propagated forward in time ahead of the simulation in order to predict the
 regions of the mesh that need to be refined and derefined.
We consider an abstract setting; hence the discussion is general and applicable
beyond the RFP equation, which is otherwise the focus of this work.
First, we present how the indicators are computed, and then we propose an approach for
the prediction of indicators.

\subsection{AMR indicators for coarsening and refinement}
\label{sec:amr-indicators}

We define an indicator as a functional, $\indfn$, that maps a function, say
$f$, to a single positive scalar value for each mesh cell $\Omega_h$
within the entire domain $\Omega$.  Hence we write
\[
  \indfn(f;\Omega_h)\in\R_+ \quad\text{for each }\Omega_h\subset\Omega.
\]
In the following, the dependence on $\Omega_h$ is assumed implicitly such
that we can use the brief notation $\indfn(f) = \indfn(f;\Omega_h)$.
The result is then used to evaluate the criteria for AMR.  A mesh cell is coarsened if the indicator is below a threshold
\[
  \indfn(f) < \indmin,
\]
and it is refined if the indicator is above a threshold
\[
  \indmax < \indfn(f),
\]
where $0\le\indmin<\indmax$ are (prescribed) constants that stay fixed during
the simulation.  These criteria for coarsening and refinement are carried
out uniformly for every cell of the mesh.  This process yields the requirement
that $\indfn(f)$ has to produce values that are independent of any scale,
because the function $f$ is likely to vary significantly.  Indeed, large
variations of $f$ are the reason to utilize AMR in the first place.

To proceed to defining AMR indicators, we introduce the following
preliminary definitions and observations.
Let $f:\Omega\rightarrow\R_+$ be a positive\footnote{%
  To simplify the
  notation, we assume the function $f$ to have positive values. One could, however, remove the assumption on positivity and, in place
  of $f$, use $\abs{f}+\varepsilon$ with a chosen $0 < \varepsilon \ll 1$.
} and continuously differentiable
function, $f\in C^1(\Omega)$, over an open domain $\Omega\in\R^d$.  We define
the gradient scale of a function $f>0$ as the nondimensionalized quantity
$(\gradient f)/f$, and as a result the scale of $f$ is neutralized.
It satisfies the \emph{gradient-scale identity}
\begin{equation}
\label{eq:gs-identity}
  \frac{\gradient f(x)}{f(x)} = \gradient\left(\log f(x)\right),
  \quad\text{for }x\in\Omega.
\end{equation}
Since we are concerned with meshes of a finite resolution, we denote
$\Omega_h\subset\Omega$ to be a mesh cell with a \emph{characteristic
size} $h>0$, and we let $\overline{\Omega}_h$ be its closure.  Then we define a
\emph{discrete local gradient magnitude}, which is motivated by finite
difference gradients,
\begin{equation}
\label{eq:discrete-grad}
  G_h[f](x) \coloneqq
    \max_{y\in\overline{\Omega}_h} \frac{\abs{f(x)-f(y)}}{h},
  \quad\text{for }x\in\Omega_h,
\end{equation}
where $\abs{\cdot}$ denotes the absolute value.  Note that for
\eqref{eq:discrete-grad} to be well defined, it is only required
for $f$ to be Lipschitz continuous.

We define two AMR indicators based on the gradient scale (GS) in two different
ways, therefore making use of both sides of the identity \eqref{eq:gs-identity}.

\begin{definition}[Gradient-scale indicator]
\label{def:gs-ind}
Let $f>0$ be Lipschitz continuous, and let $\Omega_h\subset\Omega$ be a
cell of the mesh.
The first version of the GS indicator is
\begin{equation}
\label{eq:gs}
  \indfn_{GS}(f) \coloneqq
    h \max_{x\in\overline{\Omega}_h} \left( \frac{G_h[f](x)}{f(x)} \right).
\end{equation}
The second version is
\begin{equation}
\label{eq:lgs}
  \indfn_{LGS}(f) \coloneqq
    h \max_{x\in\overline{\Omega}_h} \Bigl( G_h[\log f](x) \Bigr).
\end{equation}
\end{definition}

The indicators in Definition~\ref{def:gs-ind} are derived from a discrete
version of the gradient-scale identity with an additional multiplication by the
characteristic size $h$ of the cell.  The multiplication by $h$ reduces
the value of the indicator as the mesh is refined, and it also neutralizes
division by $h$ appearing in the discrete gradient \eqref{eq:discrete-grad}.
Hence, the indicator becomes nondimensionalized regarding the cell
size.  Overall, because we eliminated the scale of the function $f$ and the
scale of the discrete gradient $h$, the resulting indicators $\indfn_{GS}$ and
$\indfn_{LGS}$ are nondimensional.

An alternative indicator based on the dynamic ratio (i.e., maximum value
divided by minimum value of $f$ in $\overline{\Omega}_h$) is defined next.
Subsequently, a relation between the two definitions of indicators is
established in Proposition~\ref{prop:lgs-vs-ldr}.

\begin{definition}[Log-DR indicator]
\label{def:ldr-ind}
We define an AMR indicator based on the logarithm of the dynamic ratio (DR)
within a cell $\Omega_h$
\begin{equation}
\label{eq:ldr}
  \indfn_{LDR}(f) \coloneqq
    \log\left( \frac{\max_{x\in\overline{\Omega}_h} f(x)}
                    {\min_{x\in\overline{\Omega}_h} f(x)} \right).
\end{equation}
\end{definition}

\begin{proposition}[Equivalence between gradient scale and log-DR indicators]
\label{prop:lgs-vs-ldr}
Given a Lipschitz continuous function $f>0$, let $\indfn_{LGS}$ be the gradient-scale AMR indicator defined in \eqref{eq:lgs} and $\indfn_{LDR}$ be the log-DR
indicator defined in \eqref{eq:ldr}.  Then the two indicators satisfy
\begin{equation}
\label{eq:lgs-vs-ldr}
  \indfn_{LGS}(f) = \indfn_{LDR}(f).
\end{equation}
\end{proposition}
\begin{proof}
Starting from the definition \eqref{eq:lgs}, we obtain
\begin{align*}
  \indfn_{LGS}(f)
  &=
    h \max_{x\in\overline{\Omega}_h} \Bigl( G_h[\log f](x) \Bigr)
  \\&=
    h \max_{x,y\in\overline{\Omega}_h}
      \left( \frac{\abs{\log f(x) - \log f(y)}}{h} \right)
  \\&=
    \max_{x,y\in\overline{\Omega}_h} \abs{\log\left(\frac{f(x)}{f(y)}\right)}.
\end{align*}
Utilizing the monotonicity of the logarithm and $f>0$, we continue
\begin{align*}
  \max_{x,y\in\overline{\Omega}_h} \abs{\log\left(\frac{f(x)}{f(y)}\right)}
  &=
    \max\left\{
      \max_{x,y\in\overline{\Omega}_h}  \log \left(\frac{f(x)}{f(y)}\right),
      -\min_{x,y\in\overline{\Omega}_h}  \log \left(\frac{f(x)}{f(y)}\right)
    \right\}
  \\&=
    \max\left\{
       \log \left(\max_{x,y\in\overline{\Omega}_h} \frac{f(x)}{f(y)}\right),
      -\log \left(\min_{x,y\in\overline{\Omega}_h} \frac{f(x)}{f(y)}\right)
    \right\}
  \\&=
    \max\left\{
       \log \left( \frac{\max_{x\in\overline{\Omega}_h} f(x)}
                        {\min_{y\in\overline{\Omega}_h} f(y)} \right),
      -\log \left( \frac{\min_{x\in\overline{\Omega}_h} f(x)}
                        {\max_{y\in\overline{\Omega}_h} f(y)} \right)
    \right\}.
\end{align*}
Additionally, the second argument in the outermost maximum is
\[
  -\log \left( \frac{\min_{x\in\overline{\Omega}_h} f(x)}
                    {\max_{y\in\overline{\Omega}_h} f(y)} \right)
  =
  \log \left( \frac{\min_{x\in\overline{\Omega}_h} f(x)}
                   {\max_{y\in\overline{\Omega}_h} f(y)} \right)^{-1}.
\]
Therefore we can simplify
\[
  \max_{x,y\in\overline{\Omega}_h} \abs{\log\left(\frac{f(x)}{f(y)}\right)}
  =
  \log \left( \frac{\max_{x\in\overline{\Omega}_h} f(x)}
                   {\min_{y\in\overline{\Omega}_h} f(y)} \right),
\]
which shows the claim of the proposition.
\end{proof}

\begin{remark}
The indicator $\indfn_{LGS}$ can be seen in relation to the Lipschitz constant.
Let us assume that $K_h$ is a discrete version of the Lipschitz constant such
that
  $\abs{\log f(x) - \log f(y)} \le K_h \abs{x - y}$
for all $x,y\in\overline{\Omega}_h$ with $h\le\abs{x-y}$.
Then the following estimate holds:
\[
    \max_{\substack{x,y\in\overline{\Omega}_h \\ h\le\abs{x-y}}}
    \frac{\abs{\log f(x) - \log f(y)}}{\abs{x - y}}
  \le
    \max_{x,y\in\overline{\Omega}_h}
    \frac{\abs{\log f(x) - \log f(y)}}{h}
  =
    \frac{\indfn_{LGS}(f)}{h}.
\]
Thus the indicator computes an upper bound on the discrete Lipschitz constant scaled by the characteristic cell size:
  $h K_h \le \indfn_{LGS}(f)$.
\end{remark}

\begin{remark}
In order to compute the indicators \eqref{eq:gs}, \eqref{eq:lgs}, and
\eqref{eq:ldr}, a minimum and/or maximum needs to be computed over a
mesh cell.  In practice this can be done by looping over the nodes or
degrees of freedom of the particular discretization.  Furthermore, it is
sufficient to approximate the discrete local gradient magnitude
\eqref{eq:discrete-grad} by utilizing the routines for gradient computation
that are native to the discretization.

For computing the log-DR indicator in \eqref{eq:ldr} from a discretized $f$
that has been polluted by errors due to finite precision arithmetic or solver
truncations, which are unavoidable in practice, we recommend  computing
$\indfn_{LDR}(f+\epsilon)$ with a small constant $\epsilon>0$.  This ensures
numerically that the argument of $\indfn_{LDR}$ is positive.
\end{remark}

\editTwo{}{
Proposition~\ref{prop:lgs-vs-ldr} shows that the gradient-scale and the log-DR
indicators can be used interchangeably.  Specifically, the simple-to-implement
and computationally cheaper log-DR indicator quantifies features in the solution
equally well as the (slightly) more complex gradient-scale indicator.
To decide which formulation of indicator to use depends on a particular
discretization method and the quantities that are computed in the
implementation.  If, for example, a gradient already needs to be computed, then
the gradient-scale indicator is a cheap byproduct of this computation.  For the
finite difference-based discretization that we consider in this work, the log-DR
indicator presents the computationally cheapest option.  This is why we will
utilize log-DR for our adaptivity criterion in the numerical experiments that
follow in Section~\ref{sec:results}.
}

\subsection{Dynamic AMR and indicator prediction in time}
\label{sec:amr-pred}

This section is concerned with adapting a mesh dynamically to a function
$f=f(t)$ that is evolving in time.
Dynamic mesh adaptivity is a requirement that is imposed by the PDEs we aim to
solve (see Section~\ref{sec:governing-eqns}).  These PDEs are
advection-dominated in some parts of the domain and diffusion-dominated in
other parts.

The mesh can be adapted in between time steps of the implicit time integration
scheme.  To track the variations of the solution as accurately as possible, one
would ideally adapt the mesh in between every time step.  However, this will
increase the computational cost of a parallel implicit solver, because adapting
the mesh induces additional computational costs, such as
redistributing the mesh cells across compute cores, interpolation
between coarse and fine cells, and reallocating and setting up of
solver components.  If the mesh remains unchanged for too many time steps, on
the other hand, accuracy may suffer, and fine-scale features can become
insufficiently resolved because they will migrate from fine mesh cells to
coarse ones.  Therefore dynamic AMR faces two competing constraints:
(i) how long can the time intervals be between adapting the mesh in order to
adequately track the quantity of interest and
(ii) how large is the additional computational cost that each change of the
mesh induces.

The goal in this section is to introduce a new technique to balance these two
constraints.  We want to reduce the frequency of changing the
mesh and at the same time provide sufficient resolution where it is needed
(within a certain time interval).  We also want to avoid overly fine meshes
where resolution is not needed.  It therefore is desired to ``predict''  AMR
indicators in a time evolution simulation.  We will show that such a prediction
is possible in an advection-dominated setting.

We denote $\Delta t$ as the length of the time step of the time evolution
scheme.  We define a number $n_\mathrm{adapt}\in\N_+$ (i.e., positive integer)
that determines after how many time steps $\Delta t$ adaptivity is triggered;
therefore the mesh remains fixed during the interval
  $\Delta t_\mathrm{adapt} \coloneqq n_\mathrm{adapt} \Delta t$.
This implies that within the $n_\mathrm{adapt}$ time steps the features of a
function $f(t)$, which we want to resolve,
will migrate.  This is what we want to address with AMR prediction.
The superscript notation $f^{(A)}(t)$ indicates that $f(t)$ is discretized (in
space) on a particular mesh $A$.

The diagram in Figure~\ref{fig:diagram-amr} illustrates how dynamic AMR is
carried out between time steps without AMR prediction.
For simplicity of the diagram, we set $n_\mathrm{adapt}=1$, and hence
  $\Delta t_\mathrm{adapt} = \Delta t$.
The diagram shows the sequence of mesh adaptivity, where interpolation from
mesh $A$ to mesh $B$ is performed, alternating with the time evolution scheme,
where $f^{(B)}(t)$ evolves to $f^{(B)}(t + \Delta t)$.

With AMR prediction, we aim to avoid a loss of accuracy when fine-scale
features of $f(t)$ migrate away from fine to coarser regions of the mesh.
We propose to predict the ``path of refinement'' by evolving the AMR indicators
``ahead'' of the quantity of interest, $f(t)$, by intervals of at most
$\Delta t_\mathrm{adapt}$.  In effect, we are leapfrogging the AMR
indicators forward in time relative to the (regular) simulation.
Dynamic AMR with prediction is illustrated in the diagram in
Figure~\ref{fig:diagram-amr-prediction}.  For simplicity of the diagram, we set
$n_\mathrm{adapt}=n_\mathrm{pred}=1$, and hence
  $\Delta t_\mathrm{adapt} = \Delta t_\mathrm{pred} = \Delta t$.
First, the indicator $\indfn$ is computed as in the previous case without
prediction; however, $\indfn$ is subsequently evolved forward in time by
$\Delta t_\textrm{pred}$ ahead of the simulation.  To keep the
computational costs low, we propose to use an explicit time integration scheme,
which should be significantly faster to carry out compared with the implicit
scheme used to evolve $f$.  As $\indfn$ evolves, it is simultaneously reduced
into a time-independent indicator $\indfn_\textrm{pred}$.  To adapt
the mesh, we then use
$\indfn_\textrm{pred}$, which predicts the path of adaptivity.

\begin{figure}\centering
  \resizebox{0.98\columnwidth}{!}{
    \tikzset{
  fncSty/.style={font=\footnotesize, align=center},
  meshSty/.style={font=\footnotesize, align=center, draw=darkgray, rounded corners=0.5ex},
  arrowSty/.style={->, >=latex, draw=darkgray, thick},
  arrowLabelSty/.style={font=\scriptsize, align=center},
}
\begin{tikzpicture}
  \def\dh{3.0}
  \def\dv{0.9}

  \node[fncSty]   (fA1) at (-0.5*\dh,-0.0*\dv) {$f^{(A)}(t)$};
  \node[meshSty]  (mA)  at (-0.7*\dh,-1.1*\dv) {mesh $A$};
  \draw[arrowSty] (mA) -- (fA1);

  \node[fncSty]   (iA1) at (0*\dh,-2.6*\dv) {$\indfn^{(A)}(t)$};
  \node[meshSty]  (mB)  at (0.3*\dh,-1.1*\dv) {mesh $B$};
  \draw[arrowSty] (fA1) -- (iA1);
  \draw[arrowSty] (iA1) -- (mB);

  \node[fncSty]   (fB0) at (0.5*\dh,-0*\dv) {$f^{(B)}(t)$};
  \draw[arrowSty] (fA1) -- node[arrowLabelSty,above]{interp.} (fB0);
  \draw[arrowSty] (mB)  -- (fB0);

  \node[fncSty]   (fB1) at (1.5*\dh,-0*\dv) {$f^{(B)}(t + \Delta t)$};
  \draw[arrowSty] (fB0) -- node[arrowLabelSty,above]{evolve} (fB1);

  \node[fncSty]   (iB1) at (2*\dh,-2.6*\dv) {$\indfn^{(B)}(t + \Delta t)$};
  \node[meshSty]  (mC)  at (2.3*\dh,-1.1*\dv) {mesh $C$};
  \draw[arrowSty] (fB1) -- (iB1);
  \draw[arrowSty] (iB1) -- (mC);

  \node[fncSty]   (fC0) at (2.5*\dh,-0*\dv) {$f^{(C)}(t + \Delta t)$};
  \draw[arrowSty] (fB1) -- node[arrowLabelSty,above]{interp.} (fC0);
  \draw[arrowSty] (mC)  -- (fC0);

  \node[fncSty]   (fC1) at (3.5*\dh,-0*\dv) {$f^{(C)}(t + 2\Delta t)$};
  \draw[arrowSty] (fC0) -- node[arrowLabelSty,above]{evolve} (fC1);
\end{tikzpicture}
  }
  \caption{Diagram of \textbf{dynamic AMR without prediction}.
    In the ``interp.'' (interpolation) step, the indicator $\indfn$ is computed
    from $f$, and this indicator is used to adapt the mesh (e.g., from mesh
    $A$ to mesh $B$).
    The ``evolve'' step advances $f$ forward in time by $\Delta
    t_\mathrm{adapt} = \Delta t$.}
  \label{fig:diagram-amr}
  \vskip 8ex
  \resizebox{0.98\columnwidth}{!}{
    \tikzset{
  fncSty/.style={font=\footnotesize, align=center},
  meshSty/.style={font=\footnotesize, align=center, draw=darkgray, rounded corners=0.5ex},
  arrowSty/.style={->, >=latex, draw=darkgray, thick},
  arrowLabelSty/.style={font=\scriptsize, align=center},
}
\begin{tikzpicture}
  \def\dh{3.0}
  \def\dv{0.9}

  \node[fncSty]   (fA1) at (-0.5*\dh,-0.0*\dv) {$f^{(A)}(t)$};
  \node[meshSty]  (mA)  at (-0.7*\dh,-1.1*\dv) {mesh $A$};
  \draw[arrowSty] (mA) -- (fA1);

  \node[fncSty]             (iA1) at (0*\dh,-3*\dv) {$\indfn^{(A)}(t)$};
  \node[fncSty,anchor=west] (iA2) at (1*\dh,-3*\dv) {$\indfn^{(A)}(t + \Delta t)$};
  \draw[arrowSty] (iA1) -- node[arrowLabelSty,below]{evolve (expl.)} (iA2);
  \draw[arrowSty] (fA1) -- (iA1);

  \draw[decoration={brace},decorate] (iA1.north) -- (iA2.north east);
  \node[fncSty]   (pA2) at (0.9*\dh,-2.2*\dv) {$\indfn^{(A)}_\mathrm{pred}(t)$};
  \node[meshSty]  (mB)  at (0.7*\dh,-1.1*\dv) {mesh $B$};
  \draw[arrowSty] (pA2) -- (mB);

  \node[fncSty]   (fB0) at (0.5*\dh,-0*\dv) {$f^{(B)}(t)$};
  \draw[arrowSty] (fA1) -- node[arrowLabelSty,above]{interp.} (fB0);
  \draw[arrowSty] (mB)  -- (fB0);

  \node[fncSty]   (fB1) at (1.5*\dh,-0*\dv) {$f^{(B)}(t + \Delta t)$};
  \draw[arrowSty] (fB0) -- node[arrowLabelSty,above]{evolve}
                           node[arrowLabelSty,below]{(implicit)} (fB1);

  \node[fncSty]             (iB1) at (2*\dh,-3*\dv) {$\indfn^{(B)}(t + \Delta t)$};
  \node[fncSty,anchor=west] (iB2) at (3*\dh,-3*\dv) {$\indfn^{(B)}(t + 2\Delta t)$};
  \draw[arrowSty] (iB1) -- node[arrowLabelSty,below]{evolve (expl.)} (iB2);
  \draw[arrowSty] (fB1) -- (iB1);

  \draw[decoration={brace},decorate] (iB1.north) -- (iB2.north east);
  \node[fncSty]   (pB2) at (2.9*\dh,-2.2*\dv) {$\indfn^{(B)}_\mathrm{pred}(t + \Delta t)$};
  \node[meshSty]  (mC)  at (2.7*\dh,-1.1*\dv) {mesh $C$};
  \draw[arrowSty] (pB2) -- (mC);

  \node[fncSty]   (fC0) at (2.5*\dh,-0*\dv) {$f^{(C)}(t + \Delta t)$};
  \draw[arrowSty] (fB1) -- node[arrowLabelSty,above]{interp.} (fC0);
  \draw[arrowSty] (mC)  -- (fC0);

  \node[fncSty]   (fC1) at (3.5*\dh,-0*\dv) {$f^{(C)}(t + 2\Delta t)$};
  \draw[arrowSty] (fC0) -- node[arrowLabelSty,above]{evolve}
                           node[arrowLabelSty,below]{(implicit)} (fC1);
\end{tikzpicture}
  }
  \caption{Diagram of \textbf{dynamic AMR with prediction}.
    The indicator $\indfn$ is first computed from $f$ and subsequently evolved
    forward in time by $\Delta t_\textrm{pred} = \Delta t$ ahead of the
    simulation, thus predicting the path of AMR.  The evolution of
    $\indfn$ is reduced into one time-independent indicator
    $\indfn_\textrm{pred}$.  In the ``interp.'' (interpolation) step,
    $\indfn_\textrm{pred}$ is used to adapt the mesh (e.g., from mesh $A$ to
    mesh $B$).
    The ``evolve'' step advances $f$ forward in time by $\Delta
    t_\mathrm{adapt} = \Delta t$.}
  \label{fig:diagram-amr-prediction}
\end{figure}

Algorithm~\ref{alg:amr-prediction} describes in more detail how dynamic AMR
with prediction is carried out during the time evolution of a function $f$.
Specifically, the reduction of the leapfrogged indicator $\indfn$ is performed
by taking the maximum of the evolving indicator across all time steps.

\begin{algorithm}
\caption{AMR prediction by leapfrogging AMR indicator.}
\label{alg:amr-prediction}
\begin{algorithmic}[1]
  \Statex \textbf{Given:} Mesh~$A$ and function $f^{(A)}(t)$ supported on
    mesh~$A$ at time $t$;
    time step length $\Delta t>0$,
    number of steps $n_\mathrm{adapt}\in\N_{+}$ for adaptivity intervals, and
    number of steps $n_\mathrm{pred}\in\N_{+}$ for prediction intervals.
  \While {$t$ has not reached final time}
    \State $\Delta t_\mathrm{pred} = n_\mathrm{pred} \Delta t$
    \State $\Delta t_\mathrm{adapt} = n_\mathrm{adapt} \Delta t$
    \State compute indicator $\indfn^{(A)}(t)$ from function $f^{(A)}(t)$ on mesh $A$
    \State leapfrog $\indfn^{(A)}(t)$ until $\indfn^{(A)}(t + \Delta t_\mathrm{pred})$
      \Comment \emph{(``fast'' scheme)}
    \State reduce $\indfn^{(A)}$ to predictive indicator
      $\indfn^{(A)}_\mathrm{pred}(t) =
        \max_{\tau\in[t, t+\Delta t_\mathrm{pred}]} \indfn^{(A)}(\tau)$
    \State create mesh $B$ by adapting mesh $A$ based on $\indfn^{(A)}_\mathrm{pred}(t)$
      \Comment \emph{(AMR)}
    \State interpolate $f^{(A)}(t)$ to $f^{(B)}(t)$ onto mesh $B$
    \State evolve $f^{(B)}(t)$ until $f^{(B)}(t + \Delta t_\mathrm{adapt})$
      \Comment \emph{(``accurate'' scheme)}
    \State update time step length $\Delta t$
      \Comment \emph{(adaptive time stepping)}
  \EndWhile
\end{algorithmic}

\end{algorithm}

In the general setting, as it has been considered in this section, we have not
made assumptions about the time-stepping methods used to evolve $f(t)$ and to
leap-frog $\indfn$.  The choice of the pairing of these two time-stepping
methods is important when considering the computational complexity of the
overall simulation with dynamic AMR and AMR prediction.
In the context of the governing equations we are aiming to solve
(Section~\ref{sec:governing-eqns}), we are dealing with an advection--diffusion PDE.  The PDE
will require implicit time-stepping schemes, which have a significantly higher
computational complexity when compared with explicit schemes.
The leapfrogging of the indicators $\indfn$, on the other hand, has to be
performed only locally in time: for short time intervals and with a new initial
condition at each invocation of AMR prediction.
Therefore we can relax the requirement for accuracy, and we can approximate
local short-time behavior by eliminating the diffusion term of the PDE.  Hence,
the equation for AMR prediction will be an advection-only PDE, and we will use
the same advection coefficient as for the governing equations of the physical
models (see Section~\ref{sec:governing-eqns})

In
Section~\ref{sec:pred0-vs-pred1} we discuss the increase in accuracy of the numerical solution while keeping the
computational costs low, which is achieved due to AMR prediction.

\def\secstatus{finished}

\section{Relavitistic electron drift-kinetic solver based on PETSc and p4est}
\label{sec:solver}

DMBF (Data Management for Block-structured Forest-of-trees) is a new data
management object in PETSc that we have developed and implemented along with a
new relativistic electron drift-kinetic solver.
We first outline the capabilities of the implementation that are incorporated
into PETSc;  subsequently, we give a high-level description of the new
application code for relativistic electrons.

\subsection{DMBF: New PETSc data management for p4est-based AMR solvers}
\label{sec:dmbf}

Our \editAll{}{overall} goal is to discretize a computational domain with locally adaptively
refined quadrilateral meshes in two dimensions.
\editAll{}{In particular for relativistic electrons,} extreme local refinement is
critical for resolving the rapidly changing solution, which can vary 15--20
orders of magnitude, while over regions where the solution is relatively flat,
coarser meshes can help reduce the computational cost.  Therefore, a key
requirement of our solver is adaptive mesh refinement.

The dynamic mesh adaptivity in parallel is, at its lowest level, enabled by the
p4est library \cite{BursteddeWilcoxGhattas11, IsaacBursteddeWilcoxEtAl15}.
This library implements hierarchically refined quadrilateral and hexahedral meshes
utilizing forest-of-octree algorithms and space-filling curves
\cite{sundar2008bottom}.
Mesh refinement and coarsening are efficient in parallel because p4est performs
these tasks locally on each compute core without communication.  The
representation of the mesh as a quadtree/octree topology enables efficient 2:1
mesh balancing in parallel as well as repartitioning of the mesh cells
across compute cores, for both of which communication is necessary.
Space-filling curves transform a two- or three-dimensional space that is subdivided in cells
into a (one-dimensional) sequence of cells.  This sequence is used for an efficient
partitioning of mesh cells in parallel with the desirable property of
keeping spatially neighboring cells nearby each other in the sequence;
hence exhibiting memory locality.
This property of space-filling curves is responsible for keeping the amount of
communication low.
Overall, the algorithms of p4est have demonstrated scalability up to
$O$(100,000)
of processes on distributed-memory CPU-based systems
\cite{BursteddeGhattasGurnisEtAl10, SundarBirosBursteddeEtAl12}; moreover,
scalability to $O$(1,000,000) of CPU cores have been achived for complex
implicit PDE solvers \cite{rudi2015extreme, RudiStadlerGhattas17,
RudiShihStadler20}.
Generally, tree-based AMR algorithms have shown impressive scalability
results \cite{weinzierl2019peano, FernandoNeilsenLimEtAl19}.

DMBF is a new type of data management in PETSc.
It provides interfaces for block-structured forest-of-trees meshes.
The block structure is understood in the sense that leaf elements of the trees
can be equipped with blocks of uniformly refined cells.  DMBF implements
PETSc functions that give direct access to the functionality of p4est with the
least amount of overhead compared with existing PETSc data management approaches.
In addition, DMBF provides cellwise data management for storing any kind
of data that is mapped to the cells of the p4est mesh.
Because DMBF associates cellwise data directly with the p4est mesh, the cell
data managed by DMBF is created and destroyed when p4est cells are refined
and coarsened; and DMBF-managed cell data migrates across processes in parallel
as the p4est mesh is repartitioned.  In the context of discretizations of PDEs,
it is essential that data from neighboring cells be shared, which in a
distributed-memory setting is handled via point-to-point communication between
processes in order to communicate data of a ghost layer.  DMBF supports the
communication of DMBF-managed cell data within a ghost layer.

The focus of DMBF in managing general data objects enables scientific
applications that use the DMBF interface to store any kind of cell-dependent
data, for instance, the data required for finite element, finite volume, or
finite difference discretizations.  The shape and size of the cell data are left
entirely up to the application using DMBF and are not imposed by DMBF's
interface.  This approach allows for greater flexibility, because more
applications can utilize DMBF.  At the same time, it requires the
application to implement its own discretization routines (or rely on other
libraries).
We have taken this approach of separating concerns.  We have implemented a new
relativistic electron drift-kinetic solver that is supported by the DMBF
interfaces (see Section~\ref{ref:rfp-solver}).

To perform computations with the DMBF-managed cell data, the DMBF
interface provides iterator functions that call a user-specified function for
each cell.  Additional functions are capable of iterating over the edges/faces
of cells.  The adaptation of a mesh is handled with similar user-provided
functions that act on the data of each cell individually.
The DMBF interface is completed by allowing the generation of PETSc vectors and
matrices. Therefore, calculations that utilize DMBF leverage the linear and
nonlinear solvers of PETSc as well as PETSc's time-stepping functionality.

\subsection{Relativistic electron drift-kinetic solver based on DMBF}
\label{ref:rfp-solver}

This section gives a brief overview of the relativistic drift-kinetic Fokker--Planck solver
that is
\editAll{built}{a new application code building}
on the DMBF data structure.

The spatial discretization chooses
a finite volume scheme, MUSCL (Monotonic Upstream-centered Scheme for
Conservation Law) [20,21], or a conservative finite difference scheme, QUICK (Quadratic Upstream Interpolation for Convective Kinematics)
[22],
for the advection
operators.  Both schemes are second-order accurate in space
on a uniform mesh for time-dependent advection problems.
The collision operator is discretized by a central scheme using a three-point stencil that
also gives second-order accuracy.
\editOne{}{Note that the second-order central scheme for the collision operator can be interpreted as finite difference or finite volume~\cite{LeVeque02}.
We support both finite difference and finite volume schemes through associating corresponding degrees of freedom to each cell.
}
An adaptive quadtree-based
quadrilateral mesh is used. Thus our solver supports the handling of hanging nodes and can be nonsymmetric in a single cell.
More details on the AMR and the implementation can be found
in Section~\ref{sec:dmbf}.

Numerical solutions
\editAll{are always assumed to live on cell centers.}
{throughout the current paper are supported on nodes at cell centers.}
Each cell of
the adapted mesh is further subdiveded into four uniform cells, which
constitute the degrees of freedom (DOFs) of a solution's discretization.
To distinguish between the two notions of cells, we use the term \emph{mesh
cells},
\editAll{for}{denoting}
the cells of the adapted p4est-based mesh, and we use the term
\editOne{\emph{FV cells}}{\emph{cells}}
for the four uniformly refined cells inside a mesh cell.
Additional to the
\editOne{FV cells}{cells}
inside each mesh cell, the numerical schemes need to
access DOFs from neighboring
\editOne{FV cells}{cells}.
To this end, the DMBF cell data includes
a layer---also referred to as guard layer---of two
\editOne{FV cells}{cells}
around a mesh cell.
This outer layer of DOFs \editAll{can require}{requires}
point-to-point communication%
\editAll{ to be filled with data from other processes.}
{, if neighboring cells reside on different parallel processes.}

\editOne{}{
The size of the guard layer is determined from the required DOFs of a numerical
scheme.  Here, we employ the second-order QUICK and MUSCL schemes, which need
DOFs from up to two neighboring cells.  Hence, each mesh cell has a guard layer
of two (discretization) cells in each of the horizontal $p$- and vertical
$\xi$-directions.
Since we choose the adaptive mesh to be block-structured by incorporating four
cells of uniform refinement into each mesh cell, the parallel communication of
ghost cells is feasible with standard point-to-point communication of adjacent
neighbors of mesh cells (as opposed to needing ghost layers over two mesh
cells).
Filling a guard layer, if the neighboring mesh cells have the same level of
refinement, is straightforward:  assuming a ghost layer communication has taken
place, the values of the DOFs from neighbors can be copied into the guard layer
of a mesh cell; this is depicted in Figure~\ref{fig:guard-layer-uniform}.  When
the level of refinement differs, then it will differ only by one level, because
we enforce 2:1-balancing supported by p4est.  This leaves two cases to be
discussed: when a guard layer needs to be filled with DOFs from one coarser mesh
cell or from two mesh cells of one level finer.  We will present the example of
the left edge of a mesh cell, while all other edges are treated analogously.
}

\editOne{}{
First, the interpolation from one coarse mesh cell to the guard layers of two
fine mesh cells involves the DOFs shown in Figure~\ref{fig:guard-layer-adaptive},
left, where the figure depicts the coarse--fine interface along the left edge of
two fine mesh cells.  The guard layers' (fine) DOFs have coordinates that are
horizontally aligned with the coarse DOFs; this allows to interpolate only along
the vertical $\xi$-dimension.
We employ linear projections, because they are more stable than higher-order
projections with respect to extrapolation, which affects the fine DOFs at the
top and bottom of the guard layer (see Figure~\ref{fig:guard-layer-adaptive},
left); and we enforce nonnegativity after projection.
Second, the interpolation from two fine mesh cells to the guard layer of one
coarse mesh cell is depicted in Figure~\ref{fig:guard-layer-adaptive}, right, for
the fine--coarse interface along the left edge of a mesh cell.  The guard
layer's (coarse) DOFs are horizontally aligned with the fine DOFs (i.e.,
similarly to the previously presented fine guard layers).  A linear
interpolation only along the $\xi$-direction is necessary; and no extrapolation
is performed, because coarse DOFs will always be surrounded by fine DOFs.
}

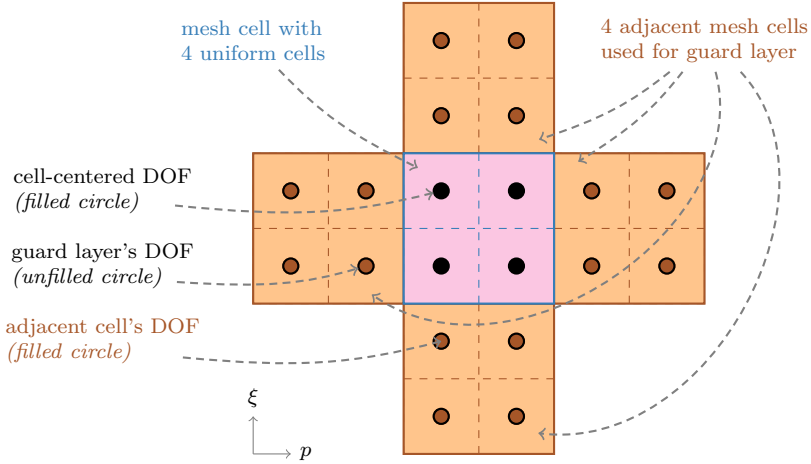
\begin{figure}\centering
    \def\colorCellBndr{jr@blue}
\def\colorCellFill{jr@pink!45}
\def\colorCellNode{black}
\def\colorGuardBndr{jr@brown}
\def\colorGuardFill{jr@orange!45}
\def\colorGuardNode{jr@brown}
\def\tikzscale{1.0}
\tikzset{
  cellSty/.style={thick},
  gridSty/.style={thin, dashed},
  nodeSty/.style={thick},
  axisSty/.style={->, gray},
  labelSty/.style={black, font=\footnotesize, align=left},
  annotateCellSty/.style={->, thick, gray, densely dashed, shorten >=-6pt},
  annotateNodeSty/.style={->, thick, gray, densely dashed, shorten <=-10pt},
  cellPic/.pic={
    \draw[fill=\colorCellFill            ] (0,0) rectangle +(2,2);
    \draw[gridSty, \colorCellBndr, step=1] (0,0) grid      +(2,2);
    \draw[cellSty, \colorCellBndr        ] (0,0) rectangle +(2,2);
    \draw[nodeSty, \colorCellNode, fill=\colorCellNode] ( 0.5,0.5) circle [radius=0.1];
    \draw[nodeSty, \colorCellNode, fill=\colorCellNode] ( 1.5,0.5) circle [radius=0.1];
    \draw[nodeSty, \colorCellNode, fill=\colorCellNode] ( 0.5,1.5) circle [radius=0.1];
    \draw[nodeSty, \colorCellNode, fill=\colorCellNode] ( 1.5,1.5) circle [radius=0.1];
  },
  guardPic/.pic={
    \draw[fill=\colorGuardFill            ] (0,0) rectangle +(2,2);
    \draw[gridSty, \colorGuardBndr, step=1] (0,0) grid      +(2,2);
    \draw[cellSty, \colorGuardBndr        ] (0,0) rectangle +(2,2);
    \draw[nodeSty, \colorGuardNode, fill=\colorGuardNode] (0.5,0.5) circle [radius=0.1];
    \draw[nodeSty, \colorGuardNode, fill=\colorGuardNode] (1.5,0.5) circle [radius=0.1];
    \draw[nodeSty, \colorGuardNode, fill=\colorGuardNode] (0.5,1.5) circle [radius=0.1];
    \draw[nodeSty, \colorGuardNode, fill=\colorGuardNode] (1.5,1.5) circle [radius=0.1];
    \draw[nodeSty, \colorCellNode,                      ] (0.5,0.5) circle [radius=0.1];
    \draw[nodeSty, \colorCellNode,                      ] (1.5,0.5) circle [radius=0.1];
    \draw[nodeSty, \colorCellNode,                      ] (0.5,1.5) circle [radius=0.1];
    \draw[nodeSty, \colorCellNode,                      ] (1.5,1.5) circle [radius=0.1];
  }
}
\begin{tikzpicture}[scale=\tikzscale]
  \begin{scope}[shift={(-2,-2)}]
    \draw[axisSty] (0,0) -- (0.5, 0) node[labelSty, anchor=west ] {$p$};
    \draw[axisSty] (0,0) -- (0, 0.5) node[labelSty, anchor=south] {$\xi$};
  \end{scope}

  \begin{scope}[shift={(-2,0)}]
    \pic[scale=\tikzscale] {guardPic};
  \end{scope}

  \begin{scope}[shift={(+2,0)}]
    \pic[scale=\tikzscale] {guardPic};
  \end{scope}

  \begin{scope}[shift={(0,-2)}]
    \pic[scale=\tikzscale] {guardPic};
  \end{scope}

  \begin{scope}[shift={(0,+2)}]
    \pic[scale=\tikzscale] {guardPic};
  \end{scope}

  \pic[scale=\tikzscale] {cellPic};

  \node[labelSty, \colorCellBndr] at (-2,+3.5) {mesh cell with\\4 uniform cells}
    edge[annotateCellSty, bend right=10] (0,1.9);
  \node[labelSty, \colorGuardBndr] at (+4,+3.5) {4 adjacent mesh cells\\used for guard layer}
    edge[annotateCellSty, bend left=10] (2,2.25)
    edge[annotateCellSty, bend left=10] (2.5,2)
    edge[annotateCellSty, bend left=70] (2,-1.75)
    edge[annotateCellSty, bend left=70] (-0.25,0);
  \node[labelSty, \colorCellNode] at (-4,+1.5) {cell-centered DOF\\\emph{(filled circle)}}
    edge[annotateNodeSty, bend right=10, shorten >=3pt] (0.5,1.5);
  \node[labelSty, \colorCellNode] at (-4,+0.5) {guard layer's DOF\\\emph{(unfilled circle)}}
    edge[annotateNodeSty, bend right=10, shorten >=3pt] (-0.5,0.5);
  \node[labelSty, \colorGuardNode] at (-4,-0.5) {adjacent cell's DOF\\\emph{(filled circle)}}
    edge[annotateNodeSty, bend right=10, shorten >=0pt] (0.5,-0.5);

\end{tikzpicture}
  \caption{\editOne{}{%
    A mesh cell \emph{(blue square)} is subdivided into four cells with one
    cell-centered DOF \emph{(back filled circle)}.  The guard layer of this mesh
    cell is comprised of four adjacent mesh cells \emph{(brown squares)}, if the
    neighboring cells have the same level of refinement.
    Filling the guard layer's DOFs \emph{(black unfilled circles)} of the
    \emph{(blue)} mesh cell is done by copying the DOFs from adjacent cells
    \emph{(brown filled circles)} in the case of uniform level of refinement.
    }}
  \label{fig:guard-layer-uniform}
\end{figure}

\begin{figure}\centering
    \def\colorCellBndr{jr@blue}
\def\colorCellFill{jr@pink!45}
\def\colorCellNode{black}
\def\colorGuardBndr{jr@brown}
\def\colorGuardFill{jr@orange!45}
\def\colorGuardNode{jr@brown}
\def\tikzscale{1.0}
\tikzset{
  cellSty/.style={thick},
  gridSty/.style={thin, dashed},
  nodeSty/.style={thick},
  axisSty/.style={->, gray},
  labelSty/.style={black, font=\footnotesize, align=left},
  titleSty/.style={black, font=\footnotesize\bfseries, align=center},
}
\begin{tikzpicture}[scale=\tikzscale]
  %
  %
  \begin{scope}[shift={(-3.5,0)}]
  \draw[axisSty] (0,0) -- (0.5, 0) node[labelSty, anchor=west ] {$p$};
  \draw[axisSty] (0,0) -- (0, 0.5) node[labelSty, anchor=south] {$\xi$};
  \end{scope}

  %
  %
  \begin{scope}
  \node[titleSty] at (-0.5,2.4) {Filling a guard layer with\\coarse-to-fine interpolation};

  \draw[fill=\colorGuardFill            ] (-2,0) rectangle +(2,2);
  \draw[gridSty, \colorGuardBndr, step=1] (-2,0) grid      +(2,2);
  \draw[cellSty, \colorGuardBndr        ] (-2,0) rectangle +(2,2);
    \draw[nodeSty, \colorGuardNode, fill=\colorGuardNode] (-1.5,0.5) circle [radius=0.1];
    \draw[nodeSty, \colorGuardNode, fill=\colorGuardNode] (-0.5,0.5) circle [radius=0.1];
    \draw[nodeSty, \colorGuardNode, fill=\colorGuardNode] (-1.5,1.5) circle [radius=0.1];
    \draw[nodeSty, \colorGuardNode, fill=\colorGuardNode] (-0.5,1.5) circle [radius=0.1];
    \draw[nodeSty, \colorCellNode,                      ] (-1.5,0.25) circle [radius=0.05];
    \draw[nodeSty, \colorCellNode,                      ] (-0.5,0.25) circle [radius=0.05];
    \draw[nodeSty, \colorCellNode,                      ] (-1.5,0.75) circle [radius=0.05];
    \draw[nodeSty, \colorCellNode,                      ] (-0.5,0.75) circle [radius=0.05];
    \draw[nodeSty, \colorCellNode,                      ] (-1.5,1.25) circle [radius=0.05];
    \draw[nodeSty, \colorCellNode,                      ] (-0.5,1.25) circle [radius=0.05];
    \draw[nodeSty, \colorCellNode,                      ] (-1.5,1.75) circle [radius=0.05];
    \draw[nodeSty, \colorCellNode,                      ] (-0.5,1.75) circle [radius=0.05];

  \draw[fill=\colorCellFill              ] (0,0) rectangle +(1,1);
  \draw[gridSty, \colorCellBndr, step=0.5] (0,0) grid      +(1,1);
  \draw[cellSty, \colorCellBndr          ] (0,0) rectangle +(1,1);
    \draw[nodeSty, \colorCellNode, fill=\colorCellNode] (0.25,0.25) circle [radius=0.05];
    \draw[nodeSty, \colorCellNode, fill=\colorCellNode] (0.75,0.25) circle [radius=0.05];
    \draw[nodeSty, \colorCellNode, fill=\colorCellNode] (0.25,0.75) circle [radius=0.05];
    \draw[nodeSty, \colorCellNode, fill=\colorCellNode] (0.75,0.75) circle [radius=0.05];
  \draw[fill=\colorCellFill              ] (0,1) rectangle +(1,1);
  \draw[gridSty, \colorCellBndr, step=0.5] (0,1) grid      +(1,1);
  \draw[cellSty, \colorCellBndr          ] (0,1) rectangle +(1,1);
    \draw[nodeSty, \colorCellNode, fill=\colorCellNode] (0.25,1.25) circle [radius=0.05];
    \draw[nodeSty, \colorCellNode, fill=\colorCellNode] (0.75,1.25) circle [radius=0.05];
    \draw[nodeSty, \colorCellNode, fill=\colorCellNode] (0.25,1.75) circle [radius=0.05];
    \draw[nodeSty, \colorCellNode, fill=\colorCellNode] (0.75,1.75) circle [radius=0.05];
  \end{scope}

  %
  %
  \begin{scope}[shift={(4,0)}]
  \node[titleSty] at (+0.5,2.4) {Filling a guard layer with\\fine-to-coarse interpolation};

  \draw[fill=\colorGuardFill              ] (-1,0) rectangle +(1,1);
  \draw[gridSty, \colorGuardBndr, step=0.5] (-1,0) grid      +(1,1);
  \draw[cellSty, \colorGuardBndr          ] (-1,0) rectangle +(1,1);
    \draw[nodeSty, \colorGuardNode, fill=\colorGuardNode] (-0.75,0.25) circle [radius=0.05];
    \draw[nodeSty, \colorGuardNode, fill=\colorGuardNode] (-0.25,0.25) circle [radius=0.05];
    \draw[nodeSty, \colorGuardNode, fill=\colorGuardNode] (-0.75,0.75) circle [radius=0.05];
    \draw[nodeSty, \colorGuardNode, fill=\colorGuardNode] (-0.25,0.75) circle [radius=0.05];
  \draw[fill=\colorGuardFill              ] (-1,1) rectangle +(1,1);
  \draw[gridSty, \colorGuardBndr, step=0.5] (-1,1) grid      +(1,1);
  \draw[cellSty, \colorGuardBndr          ] (-1,1) rectangle +(1,1);
    \draw[nodeSty, \colorGuardNode, fill=\colorGuardNode] (-0.75,1.25) circle [radius=0.05];
    \draw[nodeSty, \colorGuardNode, fill=\colorGuardNode] (-0.25,1.25) circle [radius=0.05];
    \draw[nodeSty, \colorGuardNode, fill=\colorGuardNode] (-0.75,1.75) circle [radius=0.05];
    \draw[nodeSty, \colorGuardNode, fill=\colorGuardNode] (-0.25,1.75) circle [radius=0.05];

  \draw[nodeSty, \colorCellNode,                      ] (-0.75,0.5) circle [radius=0.1];
  \draw[nodeSty, \colorCellNode,                      ] (-0.25,0.5) circle [radius=0.1];
  \draw[nodeSty, \colorCellNode,                      ] (-0.75,1.5) circle [radius=0.1];
  \draw[nodeSty, \colorCellNode,                      ] (-0.25,1.5) circle [radius=0.1];

  \draw[fill=\colorCellFill            ] (0,0) rectangle +(2,2);
  \draw[gridSty, \colorCellBndr, step=1] (0,0) grid      +(2,2);
  \draw[cellSty, \colorCellBndr        ] (0,0) rectangle +(2,2);
    \draw[nodeSty, \colorCellNode, fill=\colorCellNode] ( 0.5,0.5) circle [radius=0.1];
    \draw[nodeSty, \colorCellNode, fill=\colorCellNode] ( 1.5,0.5) circle [radius=0.1];
    \draw[nodeSty, \colorCellNode, fill=\colorCellNode] ( 0.5,1.5) circle [radius=0.1];
    \draw[nodeSty, \colorCellNode, fill=\colorCellNode] ( 1.5,1.5) circle [radius=0.1];
  \end{scope}

\end{tikzpicture}
  \caption{\editOne{}{%
    Filling a guard layer's DOFs \emph{(black unfilled circles)} in the case of
    adaptive refinement requires interpolating the DOFs of the adjacent mesh
    cell(s) \emph{(brown filled circles)} to the locations at the guard layer's
    DOFs, while the latter are positioned to resemble a stretched, uniformly
    refined grid for the DOF of the mesh cell(s) \emph{(black filled circles)}.
    }}
  \label{fig:guard-layer-adaptive}
\end{figure}
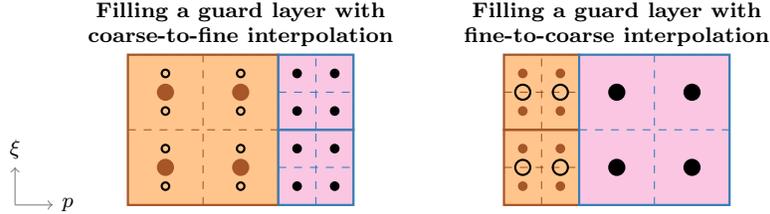

The spatial discretization is coupled with explicit or implicit time stepping.
To ease the mesh adaptivity, we choose single-step integrators.
For the explicit case a SSP Runge--Kutta is chosen, while for the implicit case
a fully implicit DIRK or ESDIRK time integrator is chosen.
Here the explicit integrators are chosen only for testing purposes.
All the production runs use implicit integrators because we are interested in
physics of the collisional time scale.

Depending on the choice of the advection scheme, the problem can be nonlinear
(in the case of MUSCL) or linear (in the case of QUICK).  For nonlinear problems, the
JFNK algorithm is the primary solver, and it is further preconditioned with an
approximate Jacobian.  The Jacobian is computed through finite difference
coloring provided by PETSc.  A GMRES iterative solver is used to invert the
linearized system (and also the linear system arising from QUICK).
A hypre \cite{FalgoutYang02} algebraic multigirid (BoomerAMG
\cite{yang2002boomeramg}) solver is used as the preconditioner.  Those choices
of the algorithm are state of the art for large-scale nonlinear simulations.

Two commonly used operators on adaptive meshes are interpolations and
integrations, both of which need to support handling of hanging nodes/edges.  For
interpolation from coarse to fine and reverse, we use tensor product
interpolation matrices because of their lower memory usage per generated Flops.
When needed, a postprocessing step is performed after interpolation to
preserve the positivity and mass conservation of the algorithm.
Computations of fluxes across hanging edges require tailored algorithms for
filling the guard cells; these are derived from the interpolation algorithms
and inherit their properties (tensor product formulation, positivity, and mass
conservation).
A line integral along $\xi$ direction is performed to compute runaway
electrons.  A simple quadrature rule is used in each cell, and the positivity
can be enforced at the price of reducing the spatial accuracy to
first-order.

\def\secstatus{finished}

\section{Experimental setup}
\label{sec:setup}

This section focuses on
  the derivation of manufactured solutions,
  the setup of numerical experiments of relativistic electron drift-kinetic simulations, and
  the experimental hardware and software environment.
The goal of this section is to ensure the reproducibility of the work and to
introduce the computational studies in the numerical results of
Section~\ref{sec:results}.

\subsection{Manufactured solution}\label{sec:solutions}

To verify the accuracy of our implementation, we rely on the method of manufactured solutions. %
Several exact solutions are constructed for the RFP equation~\eqref{eqn:rfp}.
Consider the original RFP model for the distribution function $f$,
\begin{align}
  \frac{\partial f}{\partial t} - E \xi\frac{\partial f}{\partial p}
   - E \frac{1-\xi^2}{p} \frac{\partial f}{\partial \xi} = C(f),
\end{align}
where $C(f)$ represents a general collision operator. To determine an exact solution,
we convert the RFP equation back to the coordinates of
$(p_\parallel,\mathbf{p}_\perp)$, giving
\begin{align}
  \frac{\partial f}{\partial t} - E \frac{\partial f}{\partial p_\parallel} = C(f).
  \label{eqn:rfp2}
\end{align}
Note that without the collision operator the exact solution to~\eqref{eqn:rfp2} satisfies the following identity  between $t=T$ and $t=0$:
\begin{align}
 f(T, \editOne{\textit{p}}{p_\parallel}, \mathbf{p}_\perp) = f(0, p_\parallel + ET, \mathbf{p}_\perp).
\end{align}
Therefore, exact solutions can be constructed by following characteristics. For instance, the exact solutions to the pure advection equation can be
\begin{itemize}
\item exponential solutions, such as
\begin{align*}
 f(t, p, \xi) &= \exp\Big[-p_\perp^2-(p_\parallel+Et)^2\Big] \nonumber\\
 &=\exp\Big[-p^2-2p\xi \,E t -(Et)^2\Big].
\end{align*}
 (we note that this exact solution decays properly at $p_{\max}$), or
\item sinusoidal solutions, such as
\begin{align*}
 f(t, p, \xi) =\sin(p\xi +Et), \qquad
\text{or}
\qquad
 f(t, p, \xi) = \cos(p\xi +Et)^2.
\end{align*}
\end{itemize}
Now we consider the RFP equation with a simplified collision operator:
\begin{align*}
  \frac{\partial f}{\partial t} - E \xi\frac{\partial f}{\partial p}
   - E \frac{1-\xi^2}{p} \frac{\partial f}{\partial \xi} =  \frac{\epsilon}{p^2} \frac{\partial}{\partial p}\Big[p^2 \frac{\partial f}{\partial p} \Big] + \frac{\epsilon}{p^2}\frac{\partial }{\partial \xi}\Big[(1-\xi^2)\frac{\partial f }{\partial \xi}\Big].
\end{align*}
One nontrivial exact solution can be derived as
\begin{align}
 f(t, p, \xi)=\sin(p\xi+Et) \exp(-\epsilon t)
\label{eqn:exactCollisionMMS}
\end{align}
through separation of variables~\cite{banks2017stable}.
Note that this collision operator corresponds to a special case of $C_F=0$, $C_A=C_B=\epsilon$ in the general Fokker--Planck collision operator.
These exact solutions can be used to verify numerical solutions with or without collision operators.

\paragraph{Boundary conditions associated with the manufactured solutions}
For practical problems, the RFP equation is discretized in the domain of $[p_{\min}, p_{\max}]\times[-1, 1]$ with $p_{\min}>0$ to avoid the singularity of the factor $1/p$.
We note, however, that the cited exact solutions are valid in the entire domain of $[0, +\infty)\times[-1, 1]$.
Therefore, the time-dependent Dirichlet boundary condition is used in the $p$ direction at both $p_{\min}$ and $p_{\max}$.
Because of the coefficients involving the term $1-\xi^2$, no boundary conditions are needed at $\xi=\pm 1$. Alternatively, one can enforce a Dirichlet boundary condition at $\xi=\pm 1$ for testing purposes.

\subsection{Setup of relativistic electron drift-kinetic simulations}

The RFP model considered in the current work has rich physics that are
often used to describe runaway electrons, which are one of the major sources of
tokamak disruptions.  The current work focuses on examining the capabilities
and performance of the proposed RFP solver;  further advancement of
physics studies will be carried out in follow-up work.
This section lays out details of the RFP model in order to introduce the
numerical results in the next section. It is of interest to numerically analyze
the linear and nonlinear solvers in practical physics regimes, examine the
efficiency gained by AMR, and demonstrate the parallel scalability of RFP
simulations.

The RFP model is considered in a domain of $[0.3, 60]\times[-1, 1]$ for
$[p_\mathrm{min},p_\mathrm{max}]\times[\xi_\mathrm{min},\xi_\mathrm{max}]$.
The model's dynamics are  dominated primarily by the interplay of the electric
field, the damping term, and the Fokker--Planck collision operator.
We aim to demonstrate in which regime the solvers and the preconditioners
converge robustly. It is well known that an algebraic multigrid preconditioner
deteriorates in efficacy when the advection--diffusion operator becomes
increasingly nonsymmetric~\cite{manteuffel2018nonsymmetric}.  In the RFP model,
this is possible when the electric field $E$ becomes large.  On the other hand,
there is a well-known threshold that $E$ has to surpass for a runaway
avalanche to happen~\cite{mcdevitt2018relation}.
The solver is desired to be robust at or above this threshold for $E$.

Additionally, we are concerned with the numerical performance in the presence
of the radiation damping term, $R(f)$ in~\eqref{eqn:rfp}.  This  term is close to a friction force, which is effectively another
advection term, and it mainly impacts the high-energy region (large $p$).
This operator interplays with the electric field in the moderately high energy
region, and thus it is critical for the forming of the runaway tail.
We are also interested in determining the range of $\alpha$, which is the
scaling factor of $R(f)$ in~\eqref{eqn:rfp}, where the solver exhibits a robust
performance.

Another important value of interest is the runaway electron population along the $p$-direction. This is computed through a line integration of the distribution along the $\xi$ direction with a metric that will be defined explicitly in the numerical section.
Note that computing a line integral over an adaptively refined mesh becomes a nontrivial task.
In this work we  focus only on the time evolution of this important population.
One immediate next goal would be to couple this population with the field solver so that it leads to a self-consistent model that is capable of describing the impact of the runaway electrons on the tokamak field.

Furthermore, we are interested in evolving the runaway electron distribution under the impact of the secondary knock-on source as described in Section~\ref{sec:knock_on}. Recall that we have shown
that the dominant term of the Chiu knock-on source~\eqref{eqn:chiu_full} is nonzero in a very thin region.
The major difficulty in evaluating~\eqref{eqn:chiu_full} is thus twofold:
(i) the thin region needs to be sufficiently well resolved; and
(ii) the evaluation of the source terms requires another line integral along
  $\xi$ direction, and in a parallel context the resulting integral needs to
  be communicated to each process participating in the distributed parallel
  simulation.
To address (i), we rely on AMR to resolve such a region dynamically in
time.  Therefore, this case demonstrates another aspect for the need of an
adaptive solver.
To address (ii), we design a positivity-preserving interpolation
approach for performing the line integral on an octree-based adaptively refined
mesh, and we implement MPI communication routines to distribute the integral
information across processes.

Unless otherwise noted, adaptivity criteria for mesh refinement and coarsening
are carried out with AMR indicator prediction as proposed in
Section~\ref{sec:amr-pred}.  Other details of AMR will be described in each
numerical example.

\subsection{Hardware and software environments}
\label{sec:hw-sw}

The hardware and software environments are presented in this section for the
purpose of reproducibility of numerical experiments.
The development of the solvers and the studies of the algorithmic performance
were carried out on three different HPC platforms:  Argonne
National Laboratory's Bebop cluster, Los Alamos National Laboratory's Chicoma system, and the
Cori system at the National Energy Research Scientific Computing Center.
The computations pertaining to performance, specifically to demonstrate the
parallel scalability of the solvers, were carried out on the Frontera system at
the Texas Advanced Computing Center (TACC).  Table~\ref{tab:frontera} lists the
hardware specifications of Frontera.
The software environment consists of Intel C/C++ compilers, the Intel MPI
(Message Passing Interface) library, and Intel MKL, as well as the libraries PETSc,
p4est, and hypre.  The versions of the software  utilized for parallel
scalability studies are listed at the bottom of Table~\ref{tab:frontera}.
The PETSc library has been extended with additional code that implements the
data management functionalities for adaptive meshes based on p4est. This work
relies on these extensions (currently residing in a public branch of the PETSc
repository\footnote{%
  \editTwo{}{Link to DMBF code (branch in PETSc repository):}
  \url{https://gitlab.com/petsc/petsc/-/tree/johann/jcp2023/}
}), and they will be integrated into a future release version of PETSc.

\begin{table}\centering
  \caption{Hardware specifications of the TACC Frontera system based on the CPU
    Intel Xeon Platinum 8280 (``Cascade Lake'') and the Mellanox HDR
    interconnect technology. The bottom rows list the software environment
    utilized for our parallel scalability studies.}
  \label{tab:frontera}
  \centering
  \footnotesize
\begin{tabular}[t]{lr}
  \toprule
  \textbf{Racks (total)}   &     101 \\
  \textbf{Nodes (total)}   &   8,368 \\
  \textbf{Cores (total)}   & 468,608 \\
  \midrule
  \textbf{CPUs per node}             &       2 \\
  \textbf{CPU cores per node}        &      56 \\
  \textbf{Hardware threads per core} &       1 \\
  \textbf{CPU clock rate (nominal)}  & 2.7~GHz \\
  \textbf{Memory per node}           &  192~GB \\
  \midrule
  \textbf{C/C++ compiler}     & \multicolumn{1}{l}{Intel 19.1.1}     \\
  \textbf{MPI library}        & \multicolumn{1}{l}{Intel MPI 19.0.9} \\
  \textbf{BLAS library}       & \multicolumn{1}{l}{Intel MKL 2020.1} \\
  \textbf{PETSc base version} & \multicolumn{1}{l}{3.16.4} \\
  \textbf{hypre version}      & \multicolumn{1}{l}{2.23.0} \\
  \textbf{p4est version}      & \multicolumn{1}{l}{2.0.7}  \\
  \bottomrule
\end{tabular}

\end{table}

\def\secstatus{finished}

\section{Numerical results}
\label{sec:results}

Numerical experiments were performed in order to demonstrate the convergence of the
AMR-based numerical schemes and solvers, the efficiency gained by AMR with
prediction, the robustness of the solvers for practical physics regimes, and
the algorithmic as well as parallel scalability of RFP simulations.

\subsection{Manufactured solution results: Convergence of algorithms}
\label{sec:manufactured-results}

This section employs the manufactured solution derived in
Section~\ref{sec:solutions}  to demonstrate a convergence analysis of
the overall solver, including the numerical scheme and handling of the
{finite volume discretization} %
at hanging faces of neighboring cells with different
levels of refinement.

The manufactured solution is approximated in the domain $[0.3, 60]\times[-1, 1]$ for
$[p_\mathrm{min},p_\mathrm{max}]\times[\xi_\mathrm{min},\xi_\mathrm{max}]$.
\editOne{The RFP model parameters used in}%
{The parameter $E=0.5$ is used in the model}~\eqref{eqn:exactCollisionMMS};
\editOne{are $E=0.5$, $\alpha=0.05$}{}
the simulation of the manufactured solution takes place in the
time interval $t=0,\ldots,10$; and the MUSCL scheme is used.  We  utilize
the 3rd-order explicit Runge--Kutta time integrator of PETSc with a constant
time step length \editAll{}{(varying sizes determined by spatial refinement)}.
The range of levels of mesh refinement, in between which our AMR
refinement criteria from Section~\ref{sec:amr-pred} are adapting the mesh
dynamically, is varied.  We prescribe the range of levels to be three and shift the
minimum and maximum permitted levels up by one in each simulation setup.
Consequently, the coarsest setup has refinement levels $(2,3,4)$; the next finer
setup has levels $(3,4,5)$,
\editBOne{then $(4,5,6)$, and finally $(5,6,7)$}
{etc.; and finally we reach levels $(7,8,9)$}.
Table~\ref{tab:manufactured} lists the number of
\editOne{FV cells}{cells}
that are generated by the AMR algorithm with
prediction (see Section~\ref{sec:amr-pred}) at each refinement level.

The increasingly finer mesh resolution requires a shorter time step length that
is selected to satisfy the CFL condition; see column $\Delta t$ in
Table~\ref{tab:manufactured}.  Because of the changing time step length, we
also are adjusting the frequency of time steps at which mesh refinement is
performed; the corresponding column in Table~\ref{tab:manufactured} is called
\emph{Refine freq.}  The refinement frequency determines after how many time
steps the mesh is updated to track the dynamically changing solution.
\editBOne{}{We increase the refinement frequency proportionally to the number of
time steps, hence the total number of AMR operations stays constant across all
experiments.}

\begin{table}\centering
  \caption{Numerical results pertaining to the manufactured
    solution~\eqref{eqn:exactCollisionMMS}%
    \editBOne{.  Four simulations are shown with increasing levels of refinement, with the
    total number of dynamically adapted cells increasing by factors of $\sim$3
    and the relative error to manufactured solution decreasing by a factor of
    $\sim$1/2.}{, where rows show simulations with
    increasing levels of refinement. An order of convergence of $\sim$2 is
    asymptotically reached.} }
  \label{tab:manufactured}
  \centering
  \scriptsize
  \setlength{\tabcolsep}{0.2em}  
\begin{tabular}[t]{rrrrrrrrrlrrrr}
  \toprule
      \multicolumn{8}{c}{\thead{\editOne{FV cells}{Number of cells} (per refinement level)}}
    & \thead{\#Cells}
    & \thead{$\Delta t$}
    & \thead{Time}
    & \thead{Refine}
    & \thead{Error}
    & \thead{\editBOne{}{Conv.}}
  \\[-2ex]
      level\,2 & level\,3 & level\,4 &
      level\,5 & level\,6 & level\,7 &
      level\,8 & level\,9
    & total
    &
    & \thead{steps}
    & \thead{freq.}
    &
    & \thead{\editBOne{}{order}}
  \\
  \midrule
   8&324&1,648&   --&    --&     --&     --&     --&  1,980 & 0.08      &   125&  16 & \num{6.90}{-1} &   -- \\
  --&164&1,304&4,448&    --&     --&     --&     --&  5,916 & 0.04      &   250&  32 & \num{4.55}{-1} & 0.60 \\
  --& --&1,296&3,520&14,336&     --&     --&     --& 19,152 & 0.01      & 1,000& 128 & \num{2.90}{-1} & 0.65 \\
  --& --&   --&7,268& 8,056& 48,096&     --&     --& 63,420 & 0.0025    & 4,000& 512 & \num{9.92}{-2} & 1.55 \\
  --& --&   --&   --&33,832& 20,452&163,312&     --&217,596 & 0.000625  &16,000&2048 & \num{2.80}{-2} & 1.82 \\
  --& --&   --&   --&    --&146,276& 57,632&574,784&778,692 & 0.00015625&64,000&8192 & \num{7.05}{-3} & 1.99 \\
  \bottomrule
\end{tabular}

\end{table}

\editBOne{The convergence of the numerical solution, expressed in the error to
the known exact solution~\eqref{eqn:exactCollisionMMS}, is listed in the
rightmost column of Table~\ref{tab:manufactured} denoted \emph{Error}.
The errors decrease with each additional level of refinement by a factor of 1/2
or less.  Simultaneously the number of total cells is increasing by a factor of
$\sim$3, as opposed to a factor of 4 if the refinement was performed uniformly.}
{The errors of the numerical solution relative to the known exact
solution~\eqref{eqn:exactCollisionMMS} and the order of convergence of these
errors are listed in the two rightmost columns of Table~\ref{tab:manufactured}.
The error initially decreases with an order that is around 0.6 at coarser levels
of refinement (first three rows of Table~\ref{tab:manufactured}).  As the levels
of refinement increase, the order of convergence also increases, and it reaches
a value of around two.  From this we observe an asymptotic order of convergence
of $\sim$2 for our numerical scheme with adaptive meshes.  The lower order of
error reduction that is present at coarser meshes is likely due to the specific
function that represents the manufactured solution.  A sufficiently refined mesh
may be necessary to capture certain oscillations of that sinusoidal function
before the desired asymptotic order of convergence can be observed.
}

\subsection{AMR prediction results: Higher accuracy with low computational overhead}
\label{sec:pred0-vs-pred1}

This section demonstrates numerically the benefit of adding the proposed
prediction capabilities to AMR as described in Section~\ref{sec:amr-pred}.  The
benefit is that the solution function is resolved with increased accuracy with
a comparatively small increase in computational costs.
We focus on an experimental setup where the frequency of mesh
refinement (and coarsening) varies such that AMR is performed each 32, 16, 8,
and 4 time step iterations, where the time step length is adapted based on a
local truncation error.  For brevity, we refer to these varying frequencies of
mesh refinement as RF32, RF16, RF8, and RF4, respectively.  These refinement
frequencies are used in experiments with standard mesh refinement without
prediction and AMR with prediction. The adaptive levels for mesh refinement
range from 2 to 6, hence capturing a level contrast of 4, where a uniform mesh
of level 2 would contain 192
\editOne{FV cells}{cells}
and each additional level subdivides a parent cell into 4 child cells.

The RFP model parameters for this study are $E=0.5$ and $\alpha=0.1$; the time
interval of the simulation is $t=0,\ldots,4$; and the MUSCL scheme is used.

The AMR indicator LogDR~\eqref{eq:ldr} is a metric based on features of the numerical
solution that quantifies the smoothness of the solution with nondimensional
values.  When the solution exhibits artifacts or oscillations arising from too
coarse meshes, the metric increases to values larger than one.
Therefore, we use this metric to evaluate how well the solution is resolved by
a dynamically adapted mesh, and we compare the differences as the refinement
frequency varies and for AMR without and with prediction.

Initially, we observe how the mean of the AMR indicator  evolves in time. In
Figure~\ref{fig:pred0-vs-pred1-mean} the top graphs correspond to
standard AMR (without prediction), and the bottom graphs show AMR with
prediction.  The colors of each graph indicate different refinement
frequencies.
The top of the figure
shows how more frequent mesh
adaptation is resulting in a reduction of the spatial mean of the metric and
how the metric is approaching unity.  These results stand in contrast to the bottom of
the figure, where the AMR indicators are propagated
ahead of the solution, which results in the metric being significantly lower,
by about 10--40\%
Additionally, the curves corresponding to different refinement frequencies are
more clustered together and below unity, showing that less frequent mesh
adaptation must not result in poorer-resolved solutions as in the case without
AMR prediction.
The trend as refinement frequency increases (RF32-pred \ldots RF4-pred) is
reversed compared with the top curves (RF32 \ldots RF4), because the metric
associated with RF4-pred (pink color) is above the one for RF32-pred (green
color).  The reason  is that the prediction intervals are longer for RF32-pred and
therefore more cells of the mesh are refined along the predicted path of the
solution.

Complementary to the evolution of the mean of the metric,
Figure~\ref{fig:pred0-vs-pred1-mesh} plots the evolution of the mesh size.
This figure shows how many more mesh cells are being generated due to the
prediction of AMR indicators.  The increase in the number of cells is
$<$10\%
$\sim$20\%
RF32-pred compared with RF4-pred also support the above observations about the
mean of the metric (i.e., RF32-pred's metric is below RF4-pred).
While the discussion has been highlighting the extreme cases RF32-pred and
RF4-pred, the intermediate setups RF16-pred and RF8-pred show only moderate
increases in mesh sizes (orange and purple curves in
Figure~\ref{fig:pred0-vs-pred1-mesh}) while at the same time keeping the
metric uniformly bounded below one (orange and purple curves in
Figure~\ref{fig:pred0-vs-pred1-mean}).  This shows that a balance between mesh
size, implying computational cost, and boundedness of the metric, implying a
well-resolved solution, is possible in practice.

\begin{figure}\centering
  \includegraphics[width=0.7\columnwidth]{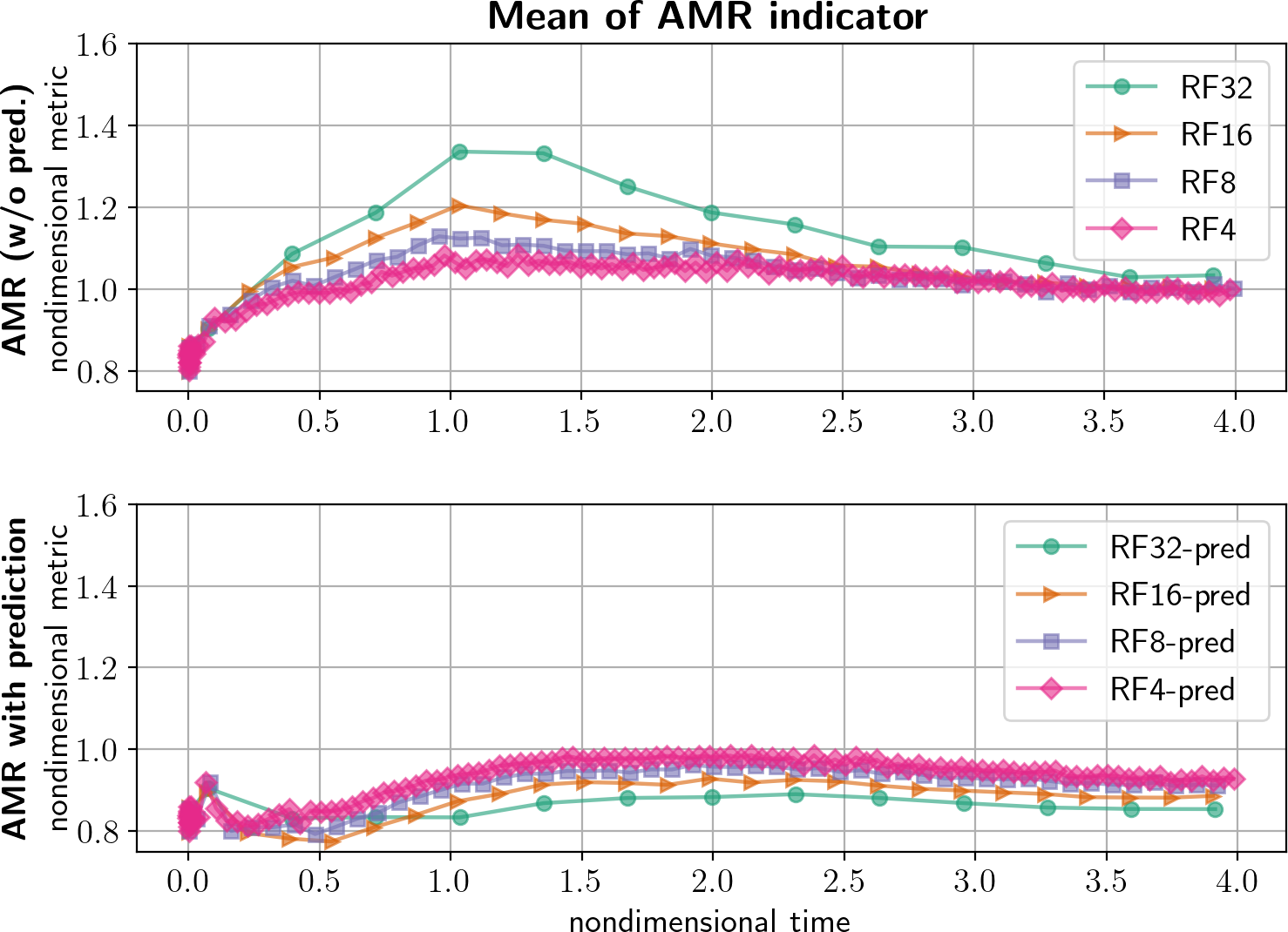}
  \vspace{-0.5\baselineskip}
  \caption{Spatial mean of the AMR indicator LogDR~\eqref{eq:ldr} evolving
    in time: AMR without prediction \emph{(top)} and AMR with prediction
    \emph{(bottom)}.
    Decreasing refinement frequencies (RF32, RF16, RF8, RF4) are lowering the
    LogDR metric, whereas with prediction the metric remains at low levels for
    any refinement frequency.}
  \label{fig:pred0-vs-pred1-mean}
  \vskip 4ex
  \includegraphics[width=0.7\columnwidth]{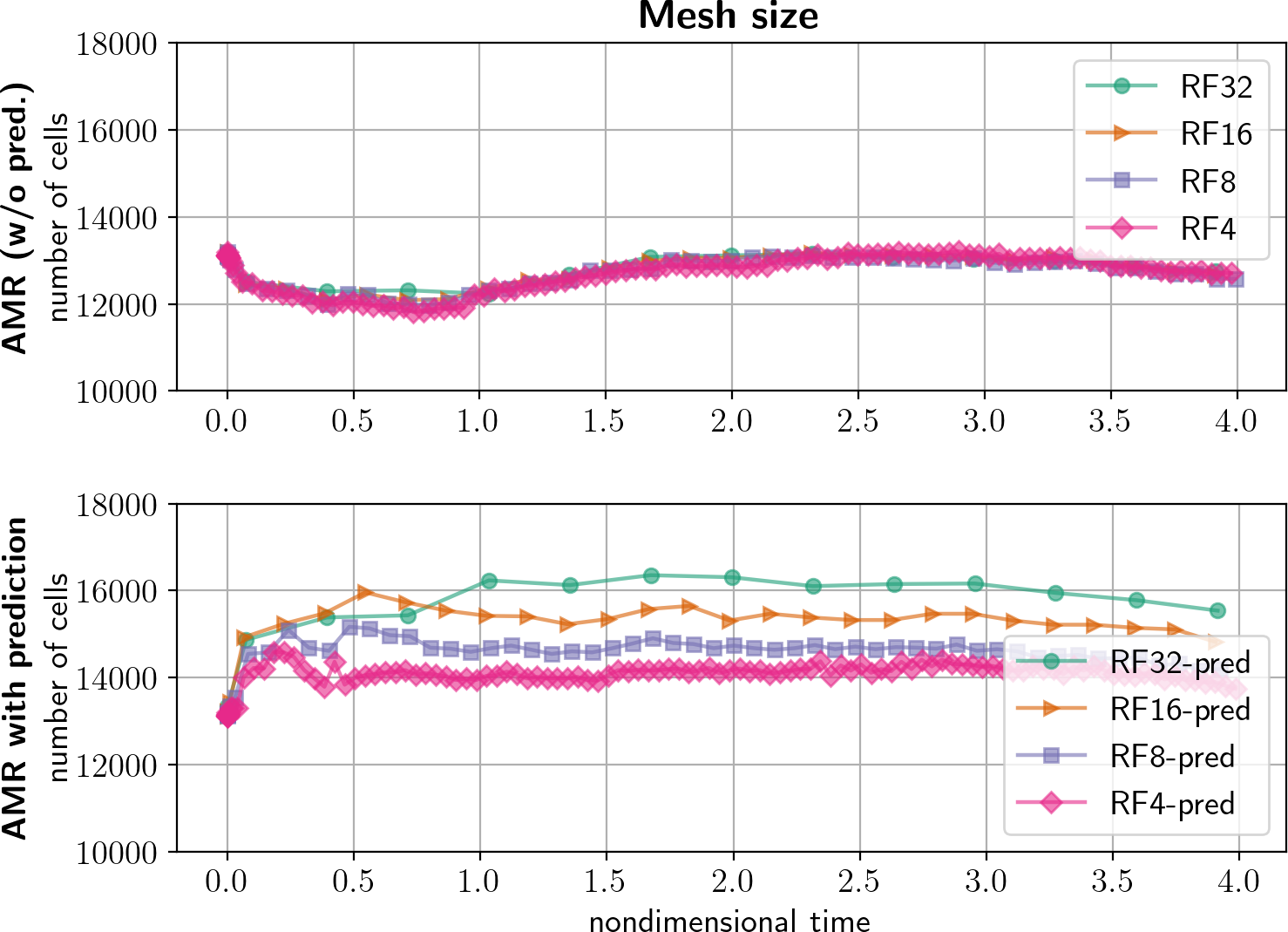}
  \vspace{-0.5\baselineskip}
  \caption{Temporal evolution of the mesh size reported as number of
    \editOne{FV cells}{cells}.
    The \emph{top} and \emph{bottom} graphs and different \emph{colors}
    correspond to the metric shown in Figure~\ref{fig:pred0-vs-pred1-mean}.}
  \label{fig:pred0-vs-pred1-mesh}
\end{figure}

Going beyond observations of the spatial mean of the metric as it evolves in
time, we also give a sense of the distribution of the metric at time instances.
Figure~\ref{fig:pred0-vs-pred1-dist} shows the envelope of one
standard deviation about the mean for each of the experiments discussed
previously.
The bottom of the figure
shows how AMR with
prediction is keeping steady control of the distribution of the AMR indicators,
because the curves are staying predominantly flat.
AMR without prediction, however, shown at the top of the figure,
demonstrates that large values of the
metric are being created for lower frequencies of refinement (see RF32, green
color).  To reduce the spread of values, one has to refine more
frequently (see RF4, pink color), which consequently comes at an increased
computational cost (see Table~\ref{tab:pred0-vs-pred1} discussed below).

The most extreme discrepancy between AMR with and without prediction is
demonstrated in Figure~\ref{fig:pred0-vs-pred1-contrast}.  The graphs in this
figure show envelopes between the spatial mean and maximum values of the metric
as it evolves in time.  That is, these graphs are showing a contrast of the
metric between mean and maximum (note that the spatial minimum of the metric is
zero for all experiments and hence cannot be used to define a contrast).  The
differences between AMR with and without prediction are most pronounced in
these figures, with a factor of $\sim$3 difference between the
top and bottom graphs of Figure~\ref{fig:pred0-vs-pred1-contrast}.
Note that in Figure~\ref{fig:pred0-vs-pred1-contrast} the initial mesh at
$t=0$ is the same for all the cases.  Thus  a few mesh adaptation steps
are needed until a more steady state of the maximum metric is reached, and it takes longer
for RF32-pred, where AMR is performed less frequently.

\begin{figure}\centering
  \includegraphics[width=0.7\columnwidth]{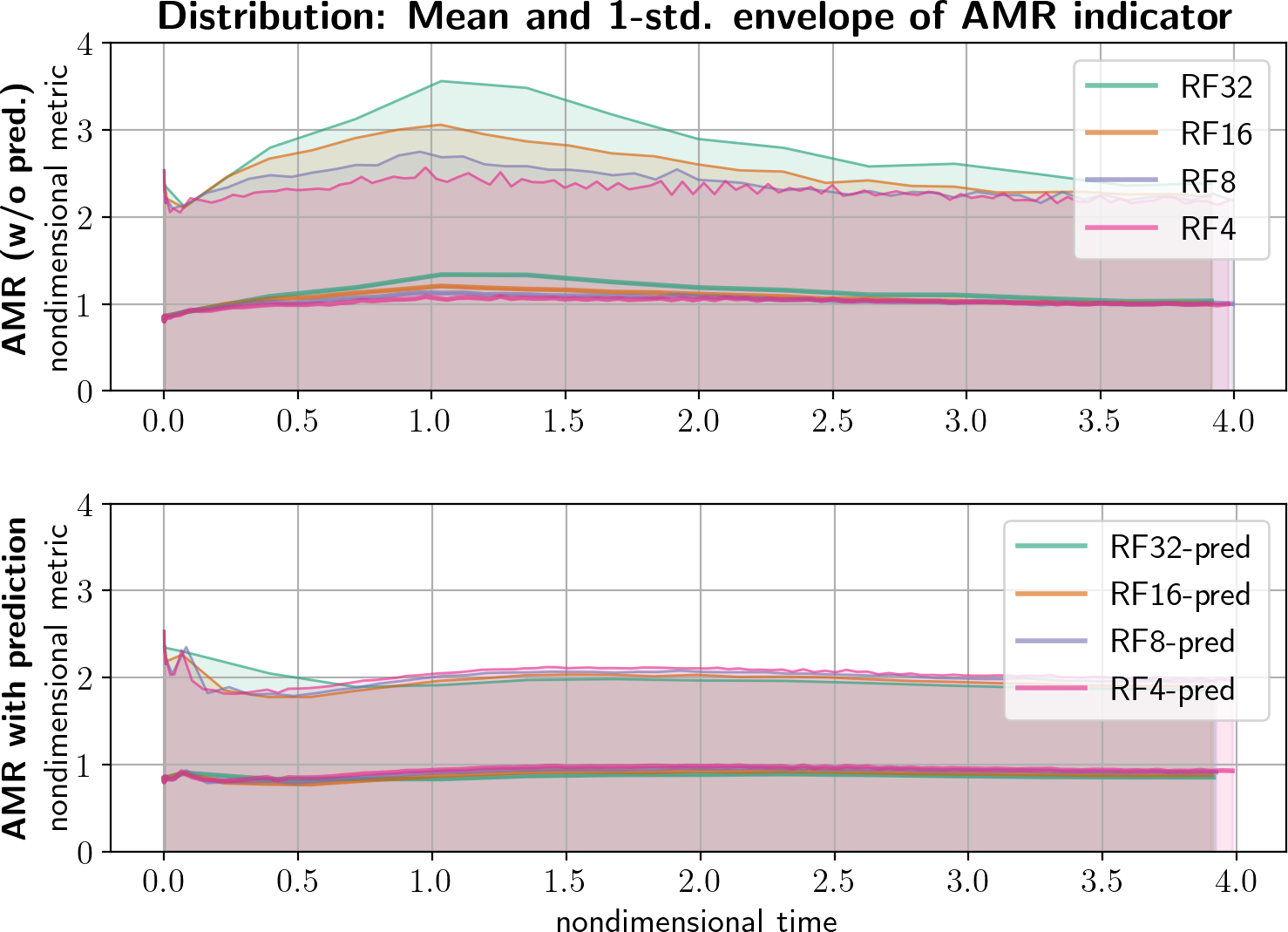}
  \vspace{-0.5\baselineskip}
  \caption{Spatial distribution of AMR indicator around the mean from
    Figure~\ref{fig:pred0-vs-pred1-mean} that is shown as an envelope of one
    standard deviation above and below the mean curve. (Note: The vertical axis
    is clipped at zero because the AMR indicator is nonnegative.)
    AMR without prediction is shown in the \emph{top} and AMR with prediction
    in the \emph{bottom} graph.}
  \label{fig:pred0-vs-pred1-dist}
\end{figure}

\begin{figure}\centering
  \includegraphics[width=0.7\columnwidth]{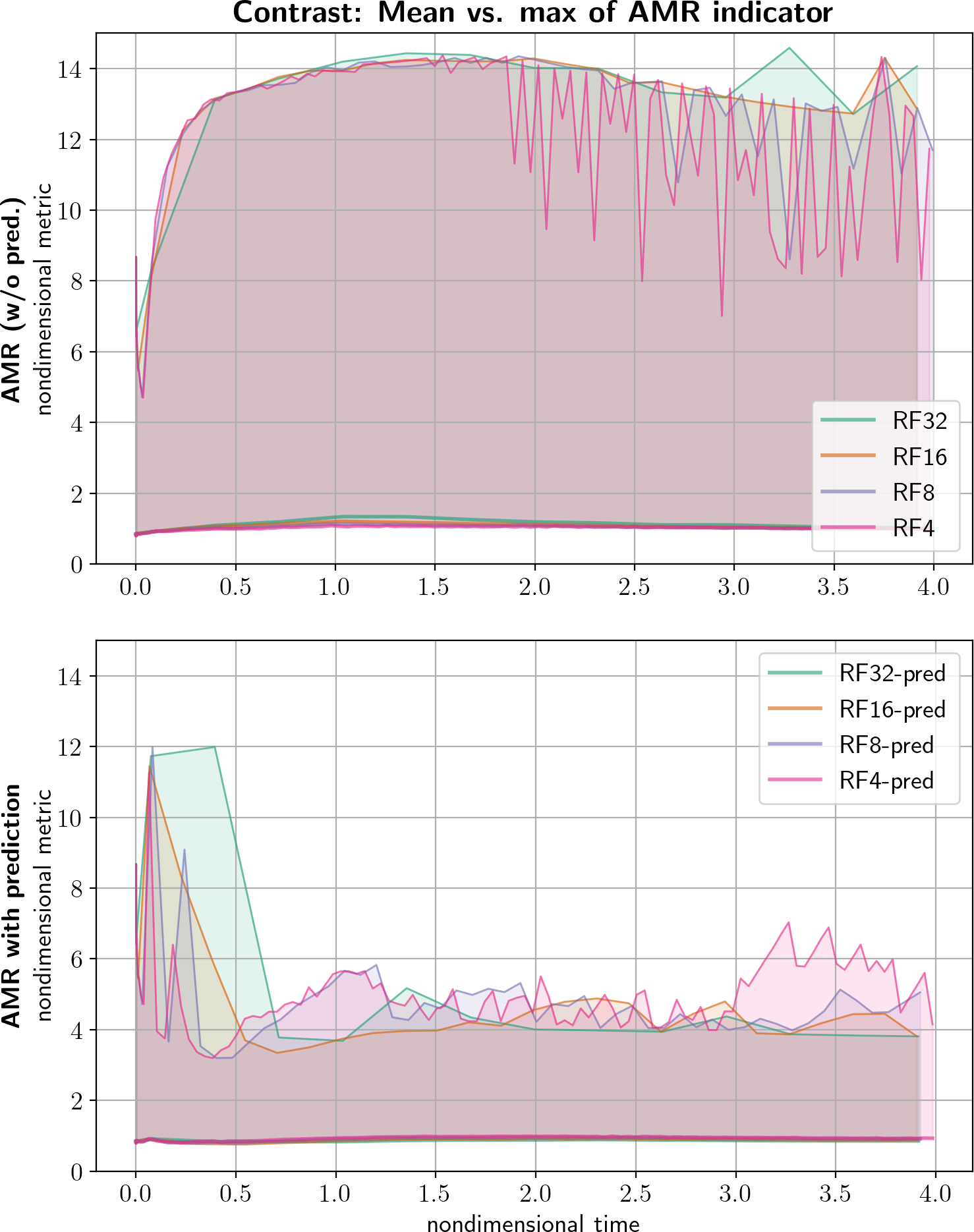}
  \vspace{-0.5\baselineskip}
  \caption{Maximum value of AMR indicator relative to the mean from
    Figure~\ref{fig:pred0-vs-pred1-mean}.
    that is shown as one standard deviation
    above and below the mean curve.
    AMR without prediction is shown in the \emph{top} and AMR with prediction
    in the \emph{bottom} graph.}
  \label{fig:pred0-vs-pred1-contrast}
\end{figure}

We summarize the preceding observations in
Table~\ref{tab:pred0-vs-pred1}.  This table lists summary statistics of the AMR
indicator by taking the time average of the spatial mean, denoted by
$\langle\text{mean}\rangle_t$, and the time averages of the standard deviation
and the maximum, denoted by $\langle\text{std}\rangle_t$ and $\langle\text{max}\rangle_t$,
respectively.  Further, the table lists a time-averaged mesh size and the
maximum mesh size over all time steps.
Complementing the earlier figures, the table presents data
about the computational costs in terms of
the time steps used until final time (these values are closely clustered), the
total number of Newton iterations across all time steps (these increase with
frequency of AMR), and the total number of GMRES iterations used (these follow
a  trend similar to the Newton iterations).
The last column of the table gives the run time of the
simulation overall, which was obtained on one node with 56 CPU cores of the
Frontera supercomputer (details in Table~\ref{tab:frontera}).
The run times show, as expected, that higher frequencies of mesh refinement
result in larger run times.
Additionally the run times for AMR with prediction are slightly larger than
for AMR without prediction, while AMR frequency is the same.  However,
we have demonstrated above that the accuracy of the solution is significantly
improved by AMR prediction.  The table shows that the run times are, for
instance, $\sim$15\%
RF8-pred).  This is a modest increase in computational cost but a dramatic
increase in accuracy as measured by the AMR indicator (shown in the second
column).

\begin{table}\centering
  \caption{Comparison of AMR without prediction \emph{(top four rows)} and AMR
    with prediction \emph{(bottom four rows)}.
    The table summarizes
    Figures~\ref{fig:pred0-vs-pred1-mean}--\ref{fig:pred0-vs-pred1-contrast} by
    reporting temporal averages of the AMR indicator and mesh sizes, where
    $\langle\cdot\rangle_t$ denotes a temporal average for the time interval
    $0.5<t\le4$.
    The computational cost is given in terms of iterations of Newton's method,
    the iterations of GMRES for the linearized systems of Newton, right-hand
    side (RHS) evaluations, and run time.}
  \label{tab:pred0-vs-pred1}
  \centering
  \scriptsize
  \setlength{\tabcolsep}{0.3em}  
\begin{tabular}[t]{lccccccc}
  \toprule
      \thead{Refinement}
    & \thead{AMR indicator}
    & \thead{Number of cells}
    & \thead{Time}
    & \thead{Newton}
    & \thead{GMRES}
    & \thead{RHS}
    & \thead{Run time}
  \vspace{-2ex}\\
      \thead{frequency}
    & $\langle\text{mean}\rangle_t ; \langle\text{std}\rangle_t ; \langle\text{max}\rangle_t$
    & average~;~max
    & \thead{steps}
    & \thead{iterations}
    & \thead{iterations}
    & \thead{eval's}
    & [seconds]
  \\
  \midrule
  RF32      & 1.16~;~1.70~;~13.88 & 12,884~;~13,596 & 456 & 1,559 & 4,896 & 59,853 & 21.02 \\
  RF16      & 1.09~;~1.49~;~13.68 & 12,836~;~13,596 & 459 & 1,711 & 5,449 & 68,163 & 23.58 \\
  RF8       & 1.05~;~1.35~;~13.32 & 12,760~;~13,596 & 458 & 1,871 & 6,006 & 75,902 & 26.42 \\
  RF4       & 1.03~;~1.28~;~12.48 & 12,740~;~13,596 & 454 & 1,934 & 6,262 & 79,222 & 28.12 \\
  \midrule
  RF32-pred & 0.86~;~1.06~;~ 4.07 & 15,400~;~16,356 & 456 & 1,586 & 4,938 & 62,916 & 26.41 \\
  RF16-pred & 0.89~;~1.07~;~ 4.15 & 15,028~;~15,960 & 457 & 1,630 & 5,158 & 66,400 & 26.24 \\
  RF8-pred  & 0.92~;~1.08~;~ 4.61 & 14,444~;~15,168 & 454 & 1,694 & 5,480 & 69,627 & 32.84 \\
  RF4-pred  & 0.95~;~1.09~;~ 4.99 & 14,004~;~14,592 & 454 & 2,034 & 6,661 & 85,296 & 32.84 \\
  \bottomrule
\end{tabular}





\end{table}

We further want to illustrate the  quantitative observations with
qualitative figures of the numerical solution.
The effects of the different AMR settings can be observed qualitatively by
visualizing the solution and the associated mesh. This is done in Figure~\ref{fig:pred0-vs-pred1-t1.99},
where the plots visualize the solution at time $t=1.99$.
Looking at the edges of the solution's distribution, we can clearly see
artifacts for nonpredictive AMR in the left column of the figure, when the
refinement lags behind the solution.  Frequency RF32 (top left) shows the most
severe artifacts.
These artifacts can be ameliorated by faster refinement frequencies, RF32,
RF16, RF8, and RF4 (top left to bottom left, respectively).
On the other hand, in the right column of Figure~\ref{fig:pred0-vs-pred1-t1.99}
when AMR prediction is active, the solution's features are resolved as well as or better
than was the case for RF4 (bottom left) before.  This result holds for all
plots: RF32-pred, RF16-pred, RF8-pred, and RF4-pred (top right to bottom right,
respectively).
The additional mesh cells that AMR prediction generates are most pronounced for
RF32-pred (top right in Figure~\ref{fig:pred0-vs-pred1-t1.99}), where the
prediction window is the longest.  This illustrates the additional
computational cost that was observed before (e.g., in
Table~\ref{tab:pred0-vs-pred1}).  Reducing the prediction window to RF16-pred
or RF8-pred results in a better balance between extra generated cells and
resolving the numerical solution.

\begin{figure}\centering
  \includegraphics[width=0.98\columnwidth]{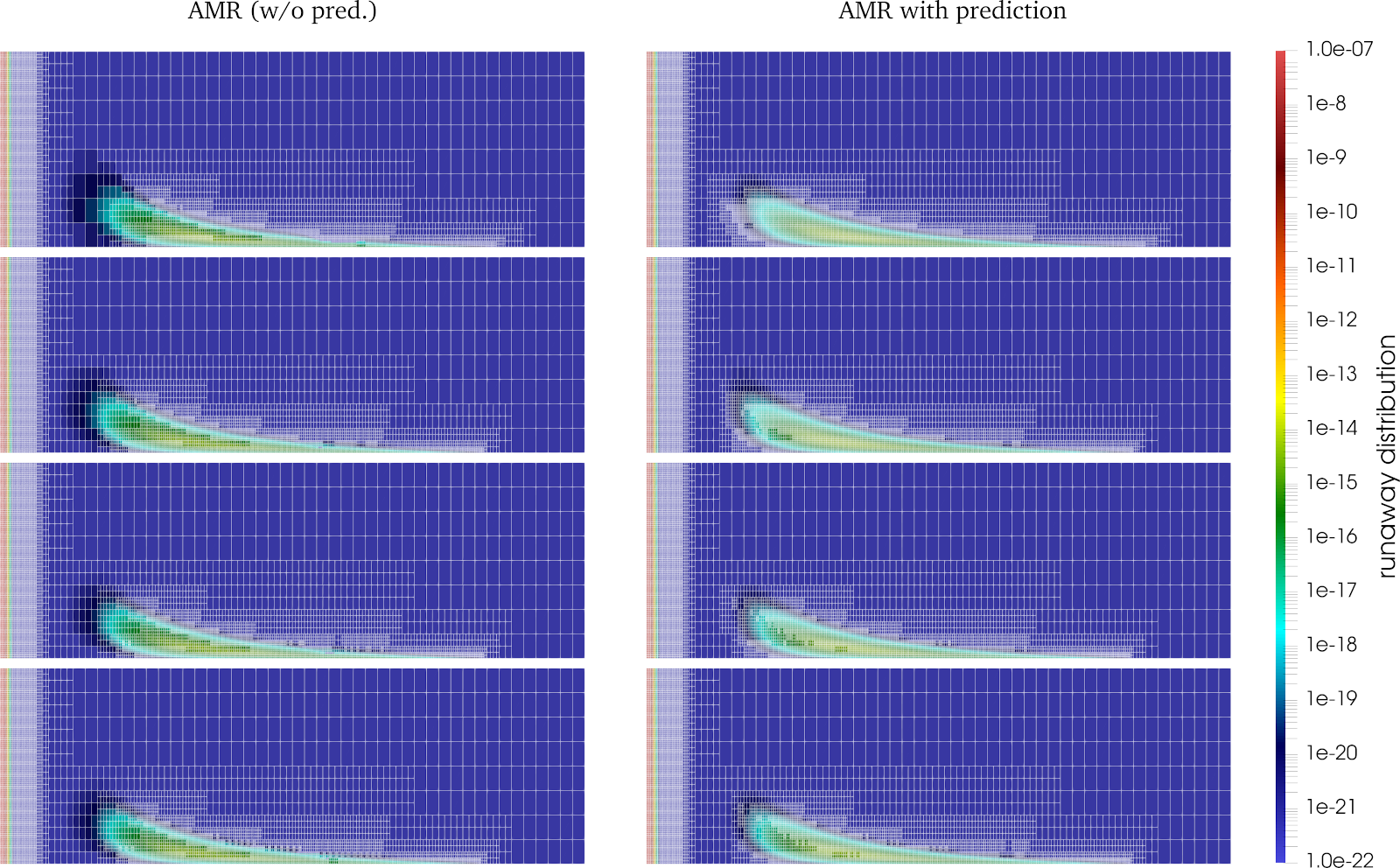}
  \caption{Visualization of the refinement levels of the dynamically adapted
    mesh \emph{(white lines)} at $t=1.99$, without prediction \emph{(left
    column)} vs.\ AMR with prediction \emph{(right column)}, while showing the
    numerical solution in \emph{colors}.  The four \emph{rows} of plots
    correspond, from top to bottom, to RF32, RF16, RF8, and RF4.}
  \label{fig:pred0-vs-pred1-t1.99}
\end{figure}

\subsection{Algorithmic robustness under different damping coefficients and electric fields}
\label{sec:robustness-results}

The first practical study focuses on the impact of two important coefficients, the damping coefficient $\alpha$, and the electric field $E$.
In this study, the Fokker--Planck collision is turned on as usual while the knock-on source is turned off.
The runs here all use the same initial condition, which is a Maxwellian with a small perturbation in the tail region.
We vary two coefficients in a range close to practice.
The AMR algorithm in this example uses a base mesh of $48\times 8$ for a domain of $[0.3, 60]\times[-1, 1]$ and a total of 7 levels of refinement.

{
\newcommand{\trimfig}[2]{\trimFig{#1}{#2}{.0}{.0}{0}{0}}
\begin{figure}[htb]
\begin{center}
\begin{tikzpicture}[scale=1]
\useasboundingbox (0,-.5) rectangle (12,2.5);
\draw(-.5,-.1) node[anchor=south west,xshift=-4pt,yshift=+0pt] {\trimfig{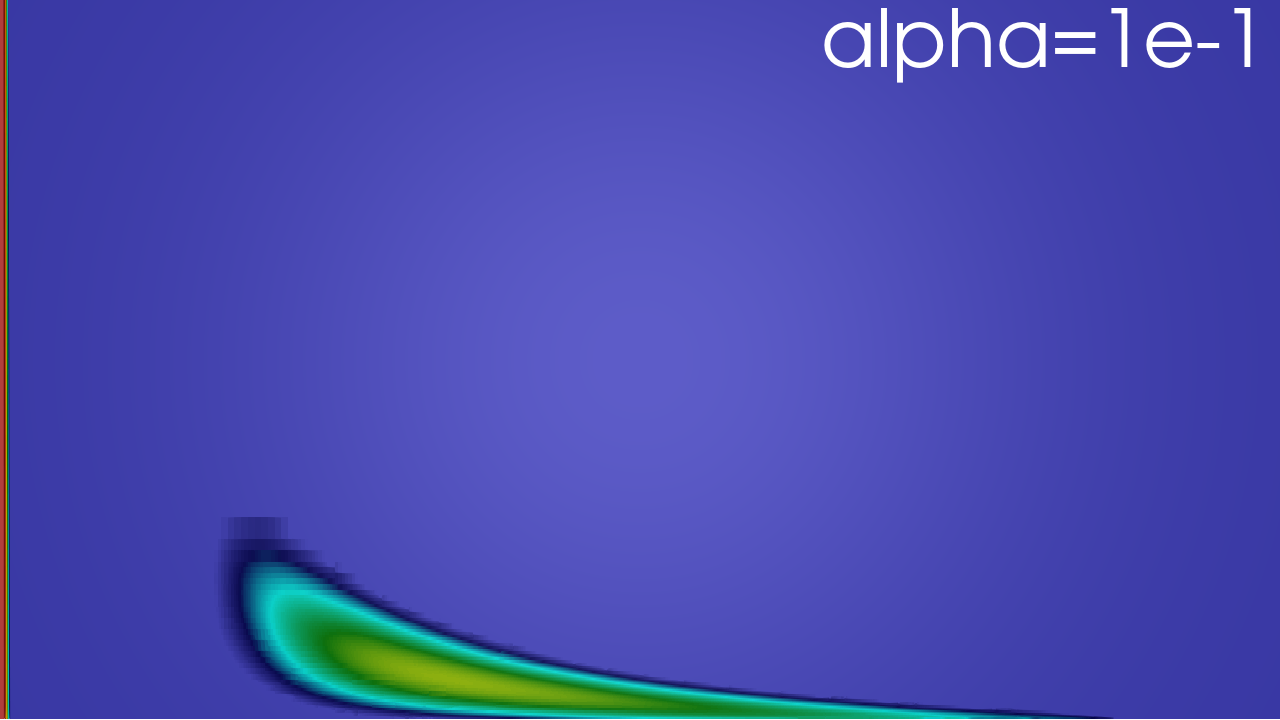}{4.35cm}};
\draw(3.9,-.1) node[anchor=south west,xshift=-4pt,yshift=+0pt] {\trimfig{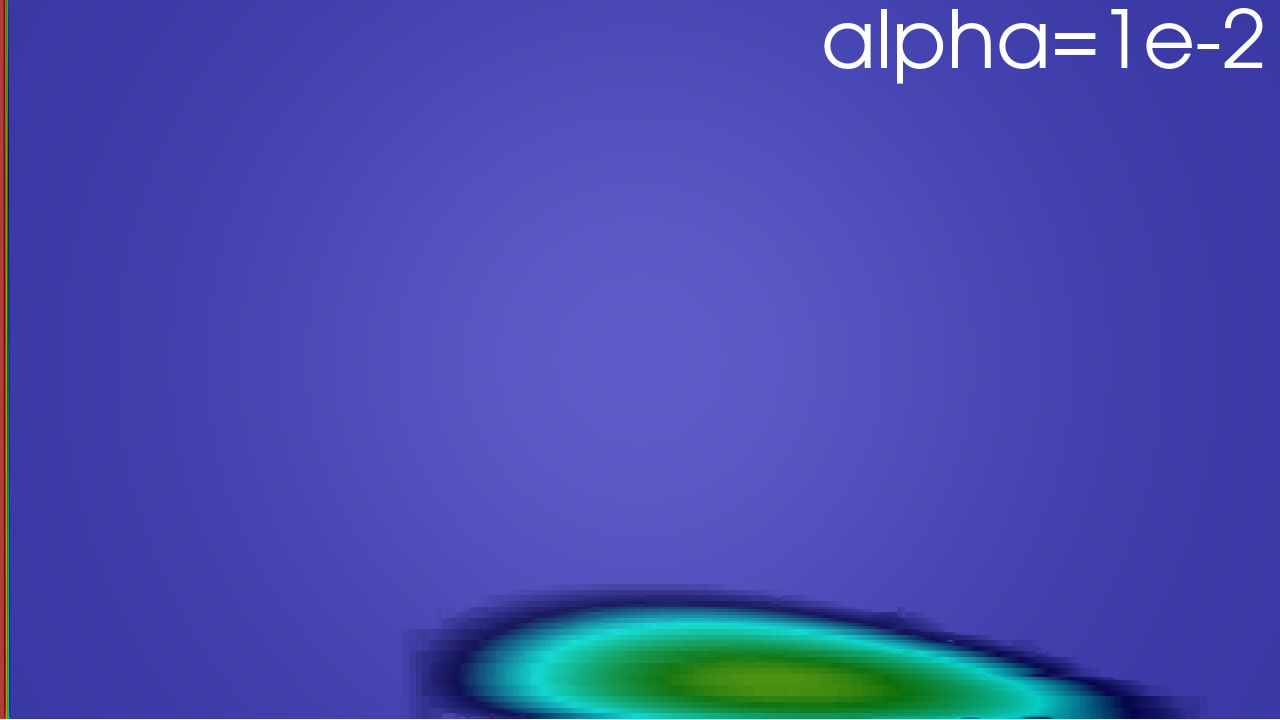}{4.35cm}};
\draw(8.3,-.1) node[anchor=south west,xshift=-4pt,yshift=+0pt] {\trimfig{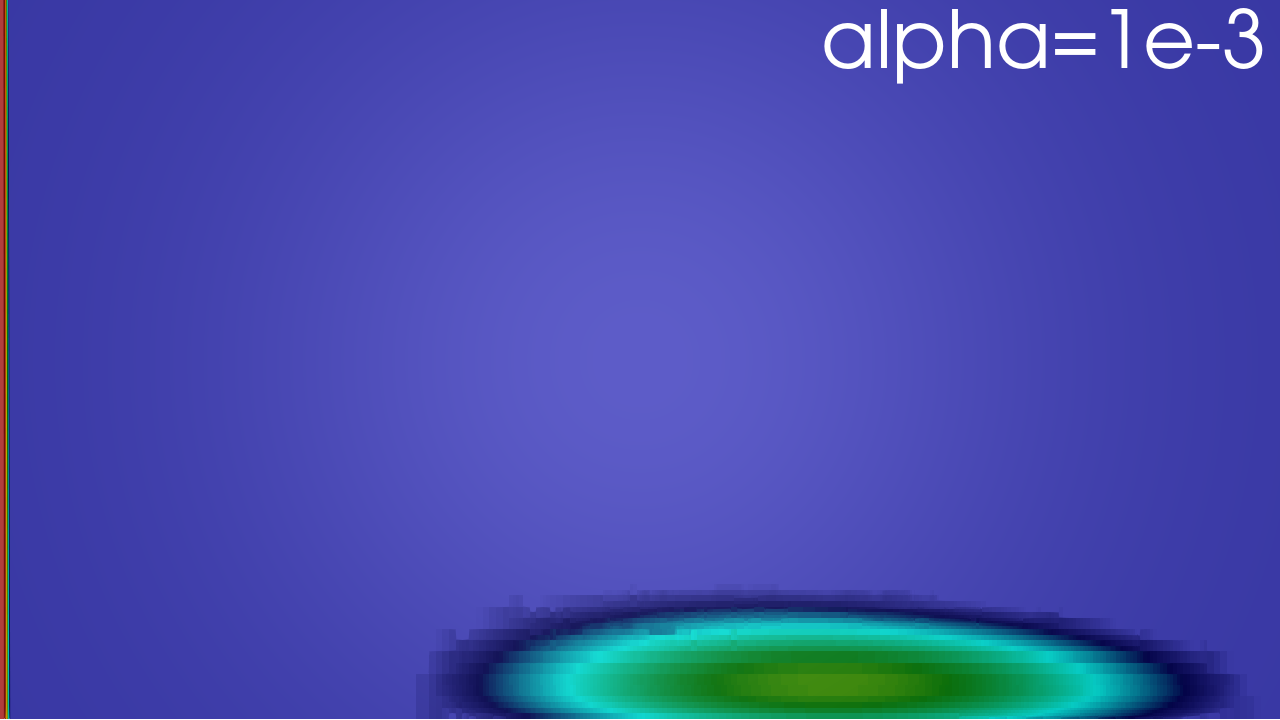}{4.35cm}};
\draw(0,-1.2) node[anchor=south west,xshift=-4pt,yshift=+0pt] {\trimfig{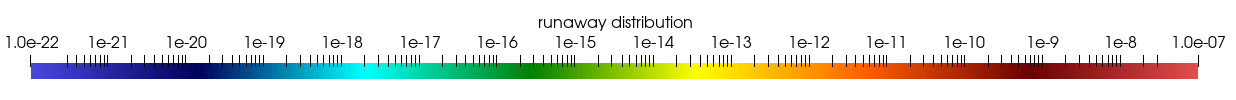}{12cm}};
\end{tikzpicture}
\end{center}
  \caption{Comparison of RFP solutions under different damping coefficients $\alpha$. The field is fixed as $E=0.5$, and the final time is $T=1$.
  The knock-on source is turned off. Note that the solutions are presented in the log scale and spread over more than 15 orders of magnitude.
  }
  \label{fig:compare_alpha}
\end{figure}
}

Distribution functions at $T=1$ with different damping coefficients are presented in Figure~\ref{fig:compare_alpha}.
One can see that a larger damping term leads to the initial high-energy perturbation moving close to the low-energy region (low $p$).
Meanwhile, Maxwellian bulks at the low-energy region are  similar under different $\alpha$, as expected.
Note that solutions in Figure~\ref{fig:compare_alpha} are presented in  log scale, which shows the distribution function spreads more than 15 orders of magnitude.

{
\newcommand{\trimfig}[2]{\trimFig{#1}{#2}{.0}{.0}{0}{0}}
\begin{figure}[htb]
\begin{center}
\begin{tikzpicture}[scale=1]
\useasboundingbox (0,-.5) rectangle (12,2.5);
\draw(-.5,-.1) node[anchor=south west,xshift=-4pt,yshift=+0pt] {\trimfig{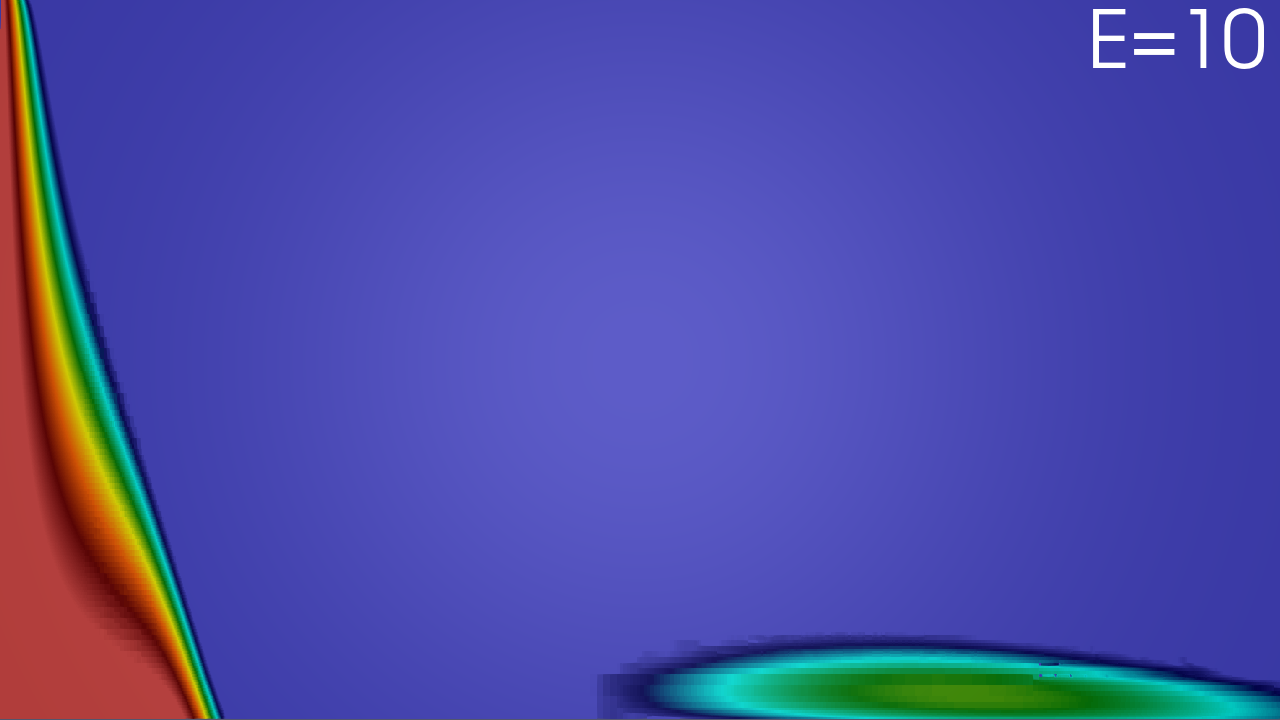}{4.35cm}};
\draw(3.9,-.1) node[anchor=south west,xshift=-4pt,yshift=+0pt] {\trimfig{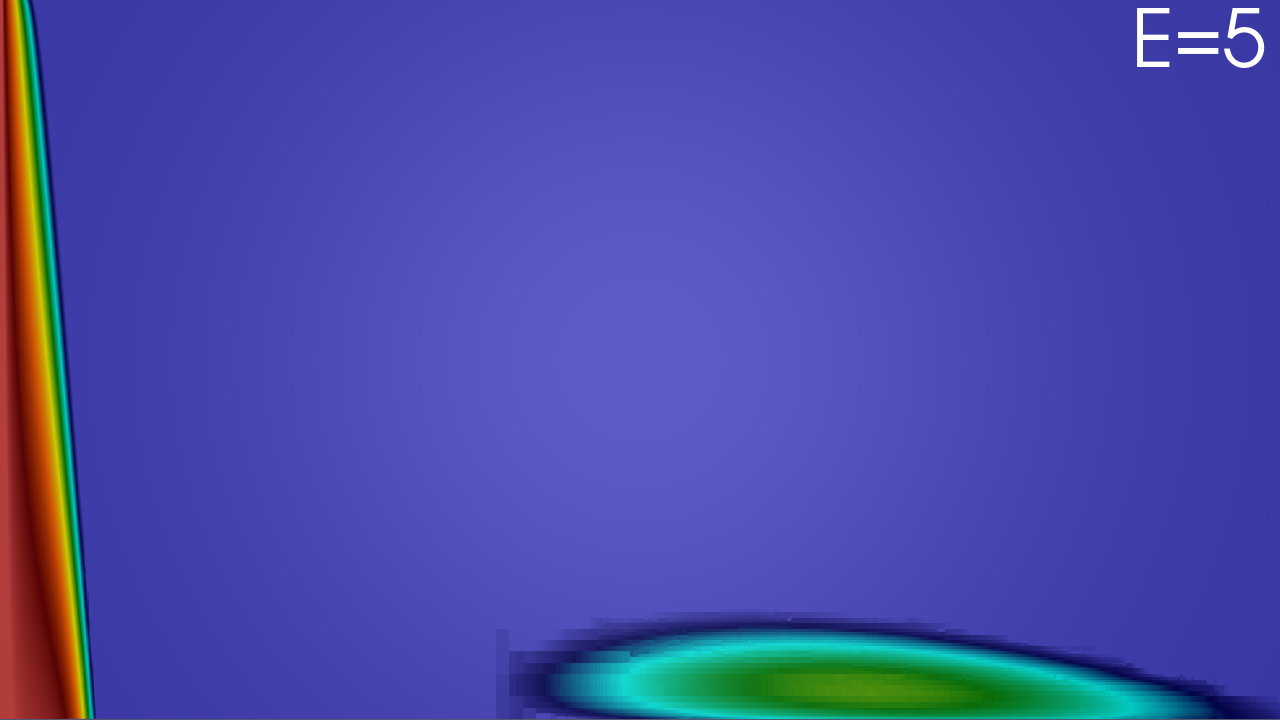}{4.35cm}};
\draw(8.3,-.1) node[anchor=south west,xshift=-4pt,yshift=+0pt] {\trimfig{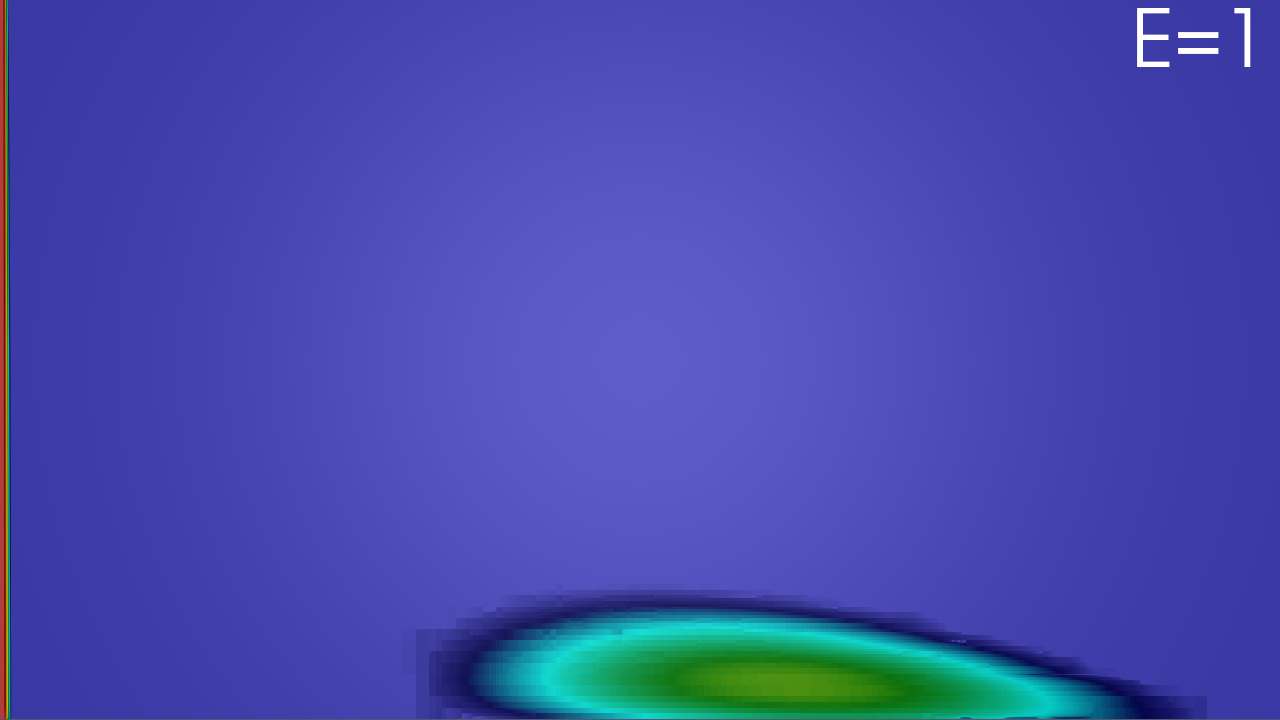}{4.35cm}};
\draw(0,-1.2) node[anchor=south west,xshift=-4pt,yshift=+0pt] {\trimfig{figs/compare/range_log}{12cm}};
\end{tikzpicture}
\end{center}
  \caption{Comparison of RFP solutions under different electric fields $E$. The damping coefficient is fixed as $\alpha=0.01$, and the final time is $T=1$.
  The knock-on source is turned off. Note that the solutions are presented in  log scale and spread over more than 15 orders of magnitude.
  }
  \label{fig:compare_e}
\end{figure}
}

Distribution functions at $T=1$ with different electric fields are presented in Figure~\ref{fig:compare_e}. We find that a larger $E$ field will push the entire distribution to the high-energy region, as expected.
The expansion of the Maxwellian is found to be faster in a larger $E$ case. In practice, however, the actual distribution will be a result of the interplay of many factors such as the damping force, the electric field, and the Fokker--Planck collision.
Therefore, the distribution will not  move all the way to the high $p$ region but instead form an interesting fat tail in a moderate high $p$ region. This effect will be studied more carefully later.
In addition, we  find that AMR can capture both the Maxwellian expansion and the movement of the tail very well (not presented here).

The second focus of this example is to study the performance of the numerical algorithms under different coefficients.
Such a study for our solver is not easy because the time steps and  details of adaptive meshes change dynamically throughout the simulations.
(Here we use ESDIRK and estimate the time step  using local truncation error estimators.) In addition, when the coefficients change, the numerical difficulty of the model also changes.
To have a fair comparison, we run all the simulations with the same final time $T=1$ and average all the important algorithm quantities over the number of time steps.

\begin{table}[htb]
\caption{Comparison of RFP solver performance under different damping coefficients $\alpha$. The field is fixed as $E=2$, and the final time is $T=1$.
  \label{table:compare_alpha}}
\centering
\vskip 1ex
\footnotesize
\begin{tabular}{ccccc}
\toprule
  \thead{$\alpha$} & \thead{Averaged $\Delta t$} & \thead{Averaged DOFs} &  \thead{RHS eval./solve} & \thead{GMRES it./solve} \\
\midrule
  \num{1}{-1} & 0.00351 & 132358 & 42.85 & 4.49 \\
  \num{5}{-2} & 0.00420 & 120592 & 43.02 & 4.52 \\
  \num{2}{-2} & 0.00441 & 114166 & 43.17 & 4.67 \\
  \num{1}{-2} & 0.00457 & 111694 & 43.22 & 4.66 \\
  \num{5}{-3} & 0.00463 & 110459 & 43.02 & 4.71 \\
\bottomrule
\end{tabular}
\end{table}

We start with testing the damping coefficients. Recall that the practical range of $\alpha$ is 0.001 to 0.3. We test several $\alpha$ and present the corresponding results in Table~\ref{table:compare_alpha}.
The electric field is fixed as $E=2$, which is above the so-called avalanche threshold (see Figure 7 in~\cite{mcdevitt2018relation}).
Here we use hypre's algebraic multigrid solver (BoomerAMG) as the preconditioner for the linearized system.
We further choose hypre's parallel ILU solver as the multigrid smoother (Euclid) and use default values  in the rest hypre options.
Table~\ref{table:compare_alpha} shows that the algorithm performs  well for different $\alpha$.
We also note that as $\alpha$ increases, the averaged time step becomes smaller, and the averaged AMR degrees of freedom become larger,
which all indicate that the problem becomes harder for large $\alpha$.
Nevertheless, the  solver still performs well. %
We also present the averaged RHS evaluation per solve. The majority of those evaluations come from evaluating the Jacobian through finite difference coloring.

\begin{table}[htb]
\caption{Comparison of RFP solver performance under different electric fields $E$. The damping coefficient is fixed as $\alpha=0.1$, and the final time is $T=1$. \label{table:compare_e}}
\centering
\vskip 1ex
\footnotesize
\begin{tabular}{ccccc}
\toprule
  \thead{$E$} & \thead{Averaged $\Delta t$} & \thead{Averaged DOFs} &   \thead{RHS eval./solve} & \thead{GMRES it./solve} \\
\midrule
 2  & 0.00351 & 132358 & 42.85 & 4.49 \\
 5  & 0.00244 & 130881 & 43.02 & 4.50 \\
 10 & 0.00192 & 131288 & 43.12 & 3.96 \\
 15 & 0.00183 & 132139 & 41.82 & 3.70 \\
 20 & 0.00179 & 132657 & 41.77 & 3.51 \\
\bottomrule
\end{tabular}
\end{table}

We then consider the impact of the electric field. A reasonable region for the normalized electric field from practice would be the interval $[0, 20]$. We test several $E$ and present the corresponding results  in Table~\ref{table:compare_e}.
We find that the linear and nonlinear solvers all perform well. We also note that the averaged time steps become smaller for larger $E$, which is expected as the problem becomes  stiffer.

\subsection{Algorithmic and parallel scalability}
\label{sec:scalability-results}

This section presents the \editOne{}{algorithmic as well as} parallel
scalability of the overall numerical simulations.
\editOne{Included are}
{One aspect of scalability is algorithmic scalability, which is the dependence
of Newton and/or Krylov iterations on the spatial and temporal resolutions.  The
second aspect is parallel scalability of the implementation, which is the
runtime measured on increasing numbers of compute cores.  Studying both aspects
is required to fully assess the performance of a solver at scale.
The runtimes measuring parallel scalability include}
the AMR algorithms, the numerical scheme, the linear solver, and the setup of
the AMG-based preconditioner.  The computations are carried out on the Frontera
system,  a CPU-based platform detailed in Section~\ref{sec:hw-sw}.

\editBOne{The setup of the scalability runs is challenging because during a
simulation the mesh is adaptively refined in a dynamic fashion, where the
refinement and coarsening depend on properties of the solution (see
Section~\ref{sec:amr}).  Therefore, it is difficult to increase the problem size
proportionally with the number of compute cores.  Because of these difficulties
we choose to present a sequence of strong scalability results only, meaning we
omit the demonstration of weak scalability.}{}
In order to obtain strong scalability results, the parameters of the problem,
and the problem size specifically, remain the same for each run while the number
of compute cores increases by a factor of 2.
\editBOne{}{The setup of weak scalability runs entails the challenge to control
the increase of the problem size proportional to the number of compute cores,
because the mesh is adaptively refined in a dynamic fashion depending on
properties of the solution.}

We generate different experiment configurations that exhibit increasing problem
sizes.  Each of these configurations provides a new baseline for showing strong
scalability, where the baseline core counts are increasing with larger problem
sizes:
(i) problem size with 1.4 million (time-averaged) \editOne{FV cells}{cells} is scaled from 112 to 7,168 cores;
(ii) problem size with 2.9 million (time-averaged) \editOne{FV cells}{cells} is scaled from 224 to 14,336 cores; and
(iii) problem size with 4.4 million (time-averaged) \editOne{FV cells}{cells} is scaled from 448 to 28,672 cores.
These three different mesh sizes feature increasingly aggressive mesh adaptivity,
which is measured by the difference between finest mesh level and coarsest mesh
level of refinement;  we call this the \emph{mesh level contrast}.
The mesh level contrast for the above described problem sizes is
(i) 8 levels contrast,
(ii) 9 levels contrast, and
(iii) 10 levels contrast.
In addition to the different mesh configurations, we alter the frequency
of mesh adaptation analogous to Section~\ref{sec:pred0-vs-pred1}.  As before,
the label RF32 stands for performing mesh adaptivity after every 32 time steps,
and additionally reported refinement frequencies are RF64, RF128, and RF256.
The AMR indicators are predicted, as proposed in Section~\ref{sec:amr-pred},
for all of these frequencies.

The RFP model parameters for this study are $E=5$ and $\alpha=0.1$, which
correspond to a realistic parameterization (see
Section~\ref{sec:robustness-results}). The Fokker--Planck collision is
turned on, while the knock-on source is turned off.
To demonstrate scalability, we prescribe 512 time
steps of the implicit second-order ESDIRK
time integrator of PETSc with a constant time step length that decreases along
with finer mesh resolutions. We use the QUICK scheme.

\editOne{}{
Table~\ref{tab:alg-scalability} demonstrates the algorithmic scalability of the
overall solver.  It shows along the columns the three increasing problem sizes
from 1.4 million to 4.4 million cells; and it also lists the varying refinement
frequencies, RF512 to RF32, in each row.  The number of GMRES iterations per
time step is given either as a range, to present its variation with increasing
core counts, or as a single number if there are no variations.  Recall that the
number of time steps is prescribed and constant. Moreover, the stopping
criterion for GMRES is the tolerance $10^{-6}$ for relative residual reduction.
We can see from the table that the number of GMRES iterations per time step is
around 14 for the smallest problem size (1.4~M cells), and it is six for the
larger problem sizes (2.9 and 4.4~M cells).  This demonstrates optimal
algorithmic scalability, because the number of iterations is not increasing as
the mesh is refined.  The larger GMRES iteration count of $\sim$14 for 1.4~M
cells can be explained by slightly worse conditioning of the disretized PDE
operator, which can be caused by large variations in the coefficients relative
to the resolution of the mesh.
}

\begin{table}\centering
  \caption{\editOne{}{Algorithmic scalability results show the number of GMRES
  iterations per time step either as a range (\emph{min \ldots\ max} value) over
  runs with increasing core counts or as a single number (when min is equal to
  max).  Columns show increasing problem sizes as number of cells, and rows show
  decreasing refinement frequencies.
  The number of GMRES iterations needed per solve is not increasing as the mesh
  becomes finer, which means the solver achieves optimal algorithmic scalability.
  }}
  \label{tab:alg-scalability}
  \centering
  \footnotesize
\begin{tabular}[t]{lccc}
  \toprule
      \thead{Refinement}
    & \thead{1.4 M cells}
    & \thead{2.9 M cells}
    & \thead{4.4 M cells}
  \vspace{-2ex}\\
      \thead{frequency}
    & (112 to  7,168 cores)
    & (224 to 14,336 cores)
    & (448 to 28,672 cores)
  \\
  \midrule
  RF256 & 12.7 \ldots\ 17.1 & 6.0 & 6.0 \\
  RF128 & 11.2 \ldots\ 14.5 & 6.0 & 6.0 \\
  RF64  & 13.6 \ldots\ 14.1 & 6.0 & 6.0 \\
  RF32  & 13.6 \ldots\ 14.6 & 6.0 & 6.0 \\
  \bottomrule
\end{tabular}

\end{table}

Implicit solvers with variable coefficients and discretized on adaptive meshes
are well known to be challenging to scale strongly in parallel requiring
significant efforts (e.g., see \cite{rudi2015extreme} for variable coefficient
Poisson and Stokes solvers).
\editBOne{Each of the three figures shows strong scalability for one of the
defined problem configurations: Figures~\ref{fig:strong-scalability-1}
and~\ref{fig:strong-scalability-3}
correspond to the strong scalability results for the problem sizes
(i) 1.4 million,
(ii) 2.9 million, and
(iii) 4.4 million \editOne{FV cells}{cells}, respectively.}
{The following two figures show strong scalability for the previously defined
problem configurations:
Figures~\ref{fig:strong-scalability-1} corresponds to problem size (i) 1.4 million and
Figures~\ref{fig:strong-scalability-3} corresponds to problem size (iii) 4.4 million
(we omit the presentation for problem size (ii) 2.9 million for brevity).}
Each figure also shows four differently colored curves, which we use to
distinguish between refinement frequencies.
Every figure shows the number of cores on the horizontal axis.  Below the core
count, the average number \editOne{FV cells}{cells} per core is reported in
parentheses, which tells how large the portion of the distributed problem is for
each core.
Mainly because of the reduction in problem size per core, the
communication starts to dominate the simulation's run time, because the fewer
degrees of freedom that reside at a compute unit's local memory, the more
communication needs to be performed via MPI.
The vertical axis shows the speedup in run time, which is calculated as the
quotient of the baseline run time over the run time with associated core count.
We report normalized speedup, where the run time is normalized with respect to
the number of GMRES iterations per time step, because we document parallel
scalability separately from algorithmic scalability.
The speedup depicted by the gray dashed line is the idealized speedup,
which, for for implicit solvers, is well known to be out of reach.

\begin{figure}\centering
  \includegraphics[width=0.9\columnwidth]{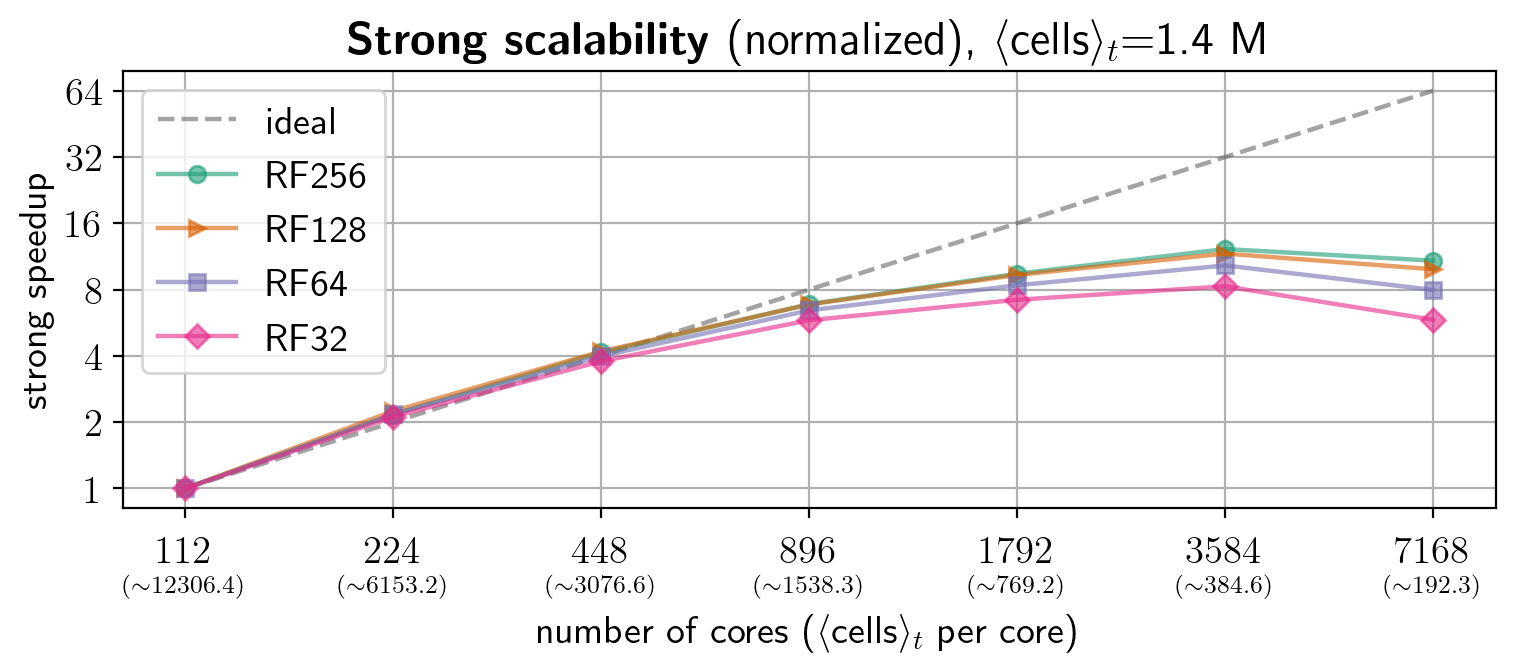}
  \vspace{-0.5\baselineskip}
  \caption{Strong scalability results on Frontera from 2 to 128 nodes, where
    the baseline problem has 1.4 million
    \editOne{FV cells}{cells}
    and a range of 8 levels of refinement; $\Delta t=10^{-5}$.}
  \label{fig:strong-scalability-1}
  \vskip 4ex
  \vskip 4ex
  \includegraphics[width=0.9\columnwidth]{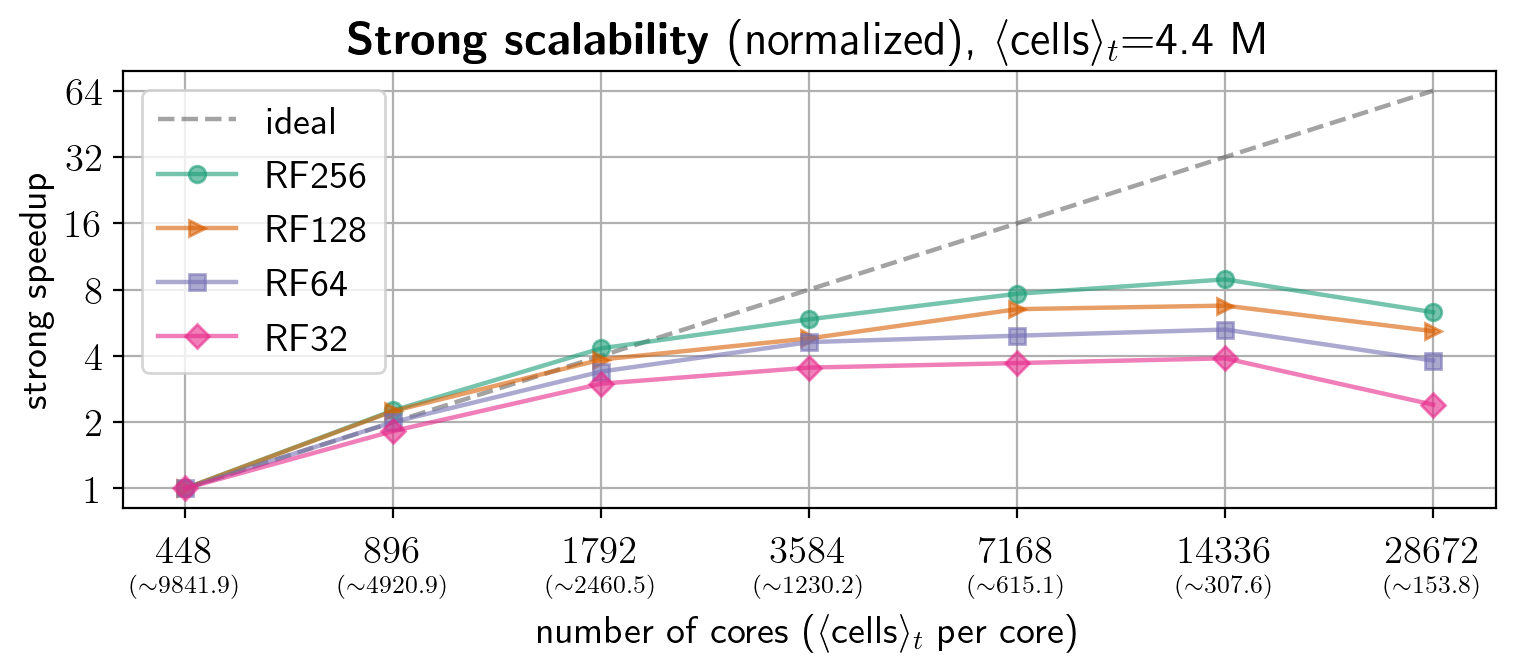}
  \vspace{-0.5\baselineskip}
  \caption{Strong scalability results on Frontera from 8 to 512 nodes, where
    the baseline problem has 4.4 million
    \editOne{FV cells}{cells}
    and a range of 10 levels of refinement; $\Delta t=10^{-6}$.}
  \label{fig:strong-scalability-3}
  \vskip 4ex
  \includegraphics[width=0.9\columnwidth]{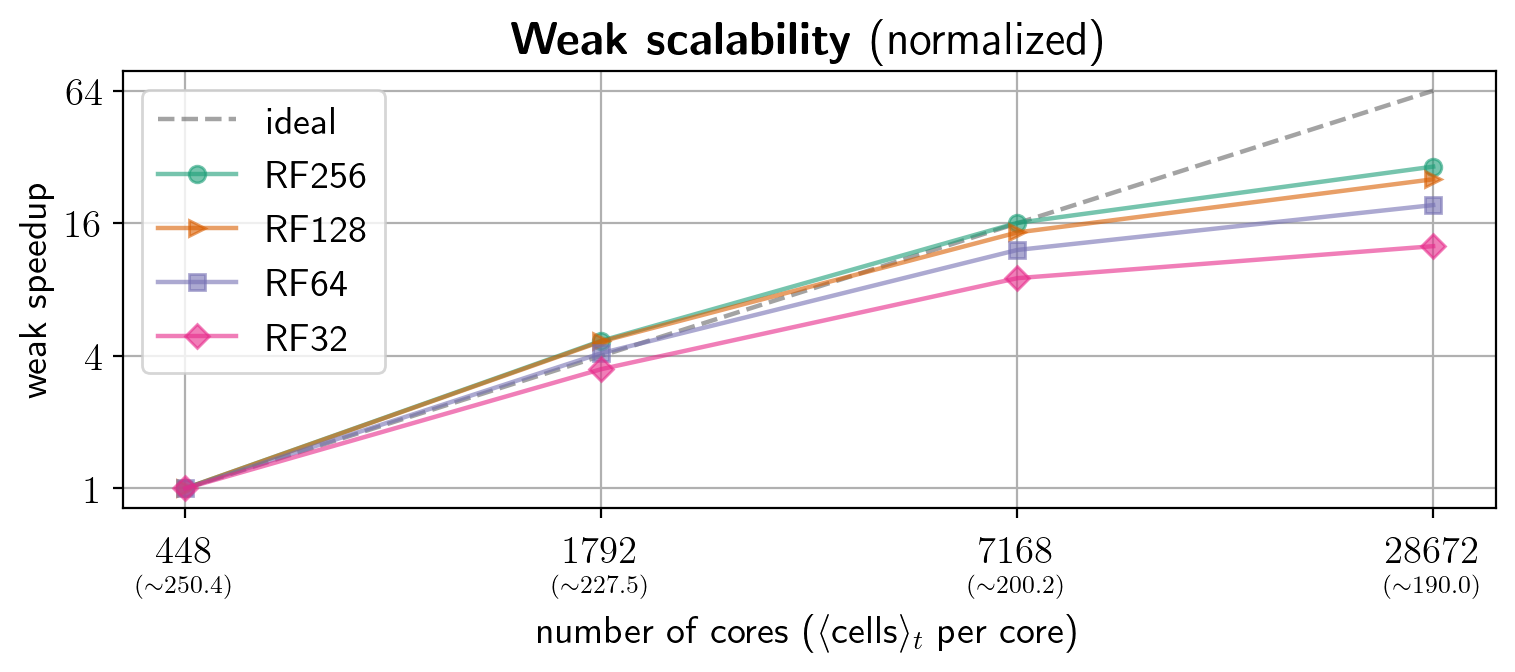}
  \vspace{-0.5\baselineskip}
  \caption{Weak scalability results on Frontera from 8 to 512 nodes, where
    the problem sizes range from 112 thousand (at 448 cores) to 5.4 million
    cells (at 28,672 cores).}
  \label{fig:weak-scalability}
\end{figure}

Each of the Figures~\ref{fig:strong-scalability-1} and~\ref{fig:strong-scalability-3}
shows a general trend of speedup that is unavoidable when strongly scaling an
implicit method: At first, the speedup stays close to the ideal speedup with
initial increases in cores; this is where the run time is dominated by
computations and/or communication is overlapped with computations.  Speedup
then gradually deviates from ideal and starts to flatten up to a certain
increase in cores.  After the flattening, the speedup decreases because the
time used for communication dominates compared with the time used for
computations.
Our strong scalability results follow this general trend. Specifically,
they demonstrate that the highest achievable speedup is 12.2 (green curve at
3,584 cores in Figure~\ref{fig:strong-scalability-1})
The highest speedup is slightly lower in
Figure~\ref{fig:strong-scalability-3} (green curve reaches speedup of 8).
We show the scalability of different refinement frequencies because this work
focuses on the AMR algorithms as well as AMR-induced overheads.
The more often the mesh is refined dynamically (i.e., the lower the refinement
frequency), the wider the gap is between the case with only two refinements during
the whole simulation (green curve for RF256) and the frequencies RF128
(orange), RF64 (purple), and RF32 (pink).
This implies that to achieve better scalability in our solver, mesh adaptation would
need to happen less often.
Note that as discussed in Section~\ref{sec:dmbf}, the AMR algorithm based on p4est
is highly scalable. The inefficiency comes from the overhead in rebuilding various
operators when the mesh is updated.
 However, less frequent adaptation of the mesh is
detrimental for the accuracy of the solution, as demonstrated in
Section~\ref{sec:pred0-vs-pred1}.
These conflicting objectives need to be balanced, which was the main motivation
in proposing indicator prediction for AMR, which achieves control over higher
accuracies while keeping the impact on computational overheads, and hence
on scalability, at a minimum.
With the demonstrated improvements in accuracy from
Section~\ref{sec:pred0-vs-pred1} due to AMR prediction, we can expect that each
of the refinement frequencies in Figures~\ref{fig:strong-scalability-1}
and~\ref{fig:strong-scalability-3}
resolves the solution similarly well.  Therefore, one can choose the
frequency that is most advantageous for scalability without sacrificing
accuracy.

\editBOne{}{The strong scalability results of
Figure~\ref{fig:strong-scalability-3} are complemented by weak scalability
results in Figure~\ref{fig:weak-scalability}.
We observe a nearly ideal speedup up to 7,168 cores and see a slight deviation
from the ideal for the maximum number of cores. Because of the dynamic
adaptivity, the problem sizes are not perfectly proportionally increasing with
core counts; see the fewer cells per core shown below the numbers of cores in
brackets.  This is a contributing reason that the weak scalability at 28
thousand cores performs slightly worse.
The trend of fewer refinements with RF256 (green curve) performing better than
the other refinement frequencies is similarly observed for weak scalability as
it was for strong scalability.
Note that we also observe better than ideal speedups for 1,792 cores, which is
a known phenomena for time dependent solvers.
}

\subsection{Benchmark with nonlinear PDE}
\label{sec:nlpde}

\editBOne{}{%
One issue of interest is how our scalable implicit solver and
dynamic AMR methods would perform for nonlinear PDEs as opposed to
the linear relativistic kinetic equation shown in this paper.  In the
context of runaway electron dynamics in a magnetized plasma, the
nonlinear coupling is through the inductive parallel electric field,
which is described by a modified Ohm's law in which the plasma current
density is replaced by the difference between total current density
and the runaway current density.  The leading order physics has a
dynamically changing parallel inductive electric field that is
independent of spatial position, so the nonlinear coupling via the
inductive electric field is numerically straightforward, and the
implementation of which is trivially simple compared with the AMR
kinetic solver in momentum space. To give a
  more reasonable test of the AMR machinery for nonlinear problems, we
  design new numerical experiments and
  implement them within our solver codes.
To this end we take
\editCOne{the two-dimensional viscous Burgers' equation}
{a two-dimensional nonlinear convection--diffusion equation}
as a test problem because of its common use for numerical benchmarks.

The domain supporting the PDE is a channel of extensions $x \in [0,6]$
horizontally and $y \in [-1,1]$ vertically.  We seek the solution $f(t,x,y)$ of
\editCOne{Burgers' equation}{the following PDE}
written in conservative form
\begin{equation}
  \label{eq:burgers}
  \frac{\partial f}{\partial t} +
  n_x \frac{\partial}{\partial x} \left( \frac12 f^2 \right) +
  n_y \frac{\partial}{\partial y} \left( \frac12 f^2 \right)
  =
  \nu \frac{\partial^2 f}{\partial x^2} +
  \nu \frac{\partial^2 f}{\partial y^2},
\end{equation}
for $t>0, x \in (0,6), y \in (-1,1)$, where $\mathbf{n}=(n_x,n_y)$ is a
given unit vector, and $\nu>0$ is a given diffusion coefficient.
Equation~\ref{eq:burgers} is complemented with Neumann boundary conditions and we
propagate an initial Gaussian distribution
$f(0,x,y) = \exp\left( -(x - 5)^2/0.16 \right) + \exp\left( -y^2/0.16 \right)$
with the direction of advection $\mathbf{n} = (-1,0)$ and viscosity $\nu =
10^{-4}$ until final time $T_\mathrm{final}=20$.
\editCOne{}{A one-dimensional version of \eqref{eq:burgers} was discussed in
\cite{Burgers1948}, and in the case $\nu=0$, the PDE is commonly referred to as
Burgers' equation.}
Due to the nonlinear nature of the PDE \eqref{eq:burgers}, the changing
advective speed results in the formation of discontinuous solutions.  This
behavior is similar to shock waves, and the shock dissipates as the solution
travels because of the viscosity.

Dynamic AMR is performed every 100th time step with prediction of adaptive
refinement using a linear approximation to \eqref{eq:burgers} and $\nu=0$, such
that a computationally cheap explicit scheme can be employed for propagation of
AMR indicators.  The overall levels of mesh refinement ranges from two to
eight, resulting in a mesh level contrast of six.
Figure~\ref{fig:burgers} presents the snapshots of the numerical solution at
times $t=0,5,10,20$.  We observe the formation of a shock discontinuity at the
left tip of the solution.  AMR is facilitating a higher resolution of the
solution's sharp features as is depicted with a mesh wireframe overlaying the
solution function at final time $t=20$.
}

\begin{figure}\centering
  \includegraphics[width=0.98\columnwidth]{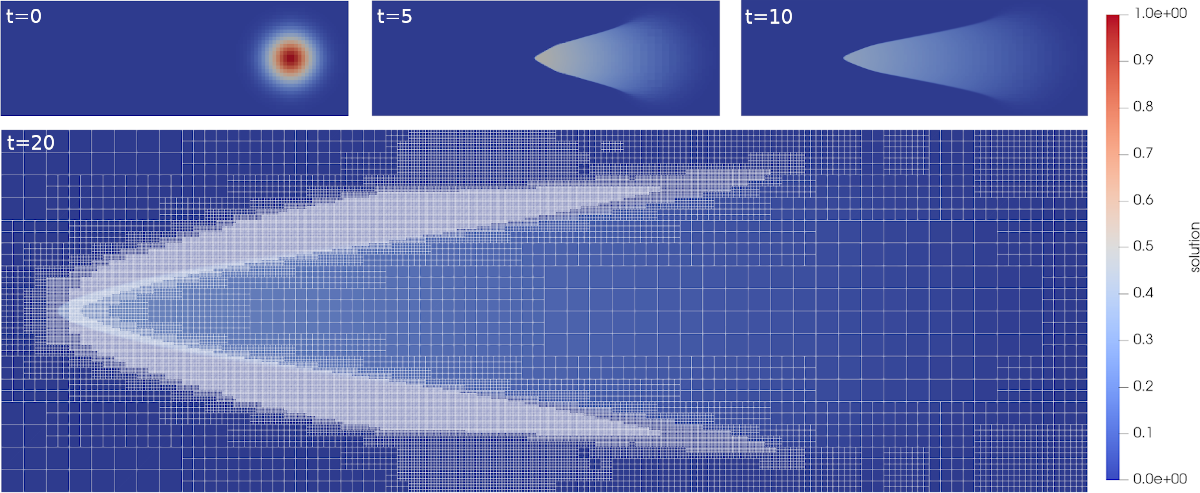}
  \caption{\editBOne{}{%
    The numerical solution of the \editCOne{viscous Burgers' equation}{nonlinear
    PDE \eqref{eq:burgers}} shown as
    snapshots over time. At final time $t=20$, the solution is overlaid with the
    mesh wireframe to show the adaptive refinement resolving sharp features of
    the solution.}}
  \label{fig:burgers}
\end{figure}

\subsection{Compatible boundary condition at $p=0$ with AMR}
\label{sec:p-eq-zero}

{
\newcommand{\figWidth}{2.4cm}
\newcommand{\trimfig}[2]{\trimFig{#1}{#2}{.0}{1}{0}{0}}
\newcommand{\trimfigb}[2]{\trimFig{#1}{#2}{.0}{0}{0}{0}}
\begin{figure}[tb]
\begin{center}
\begin{tikzpicture}[scale=1]
\useasboundingbox (0.0,0) rectangle (12,5);  %
\draw(0,   -.5) node[anchor=south west,xshift=-4pt,yshift=+0pt] {\trimfig{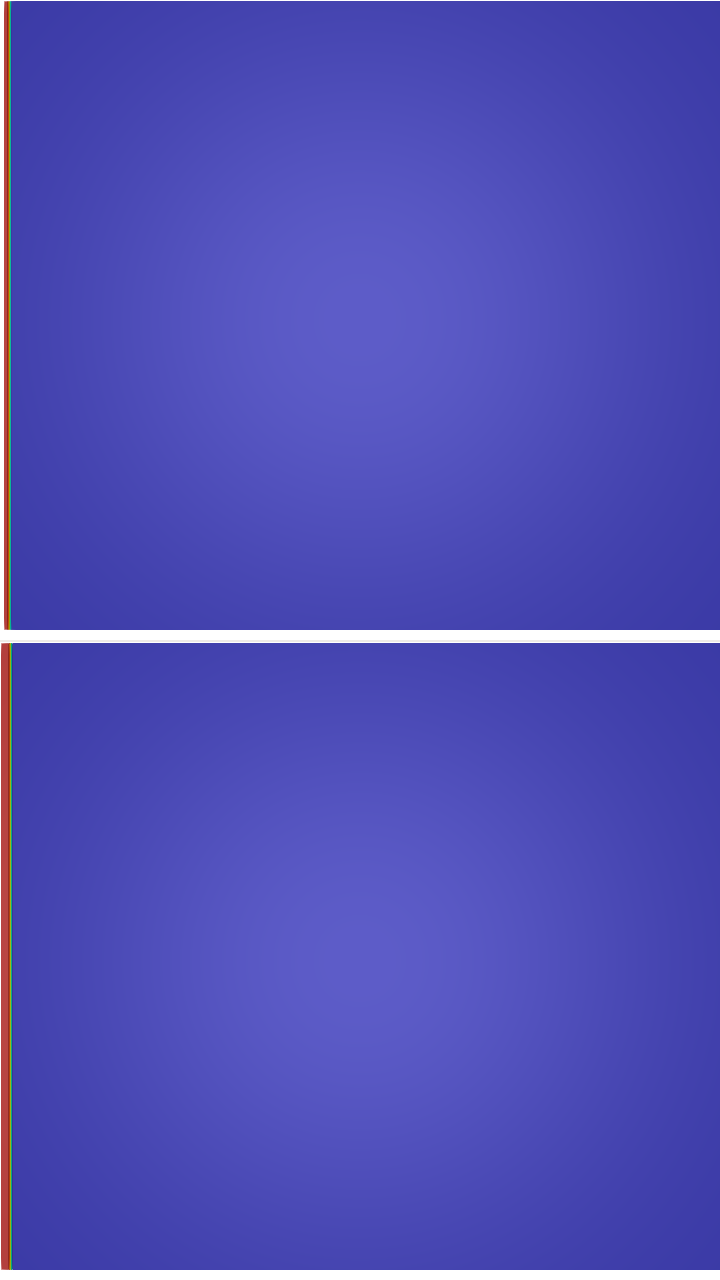}{\figWidth}};
\draw(2.5,-.5) node[anchor=south west,xshift=-4pt,yshift=+0pt] {\trimfig{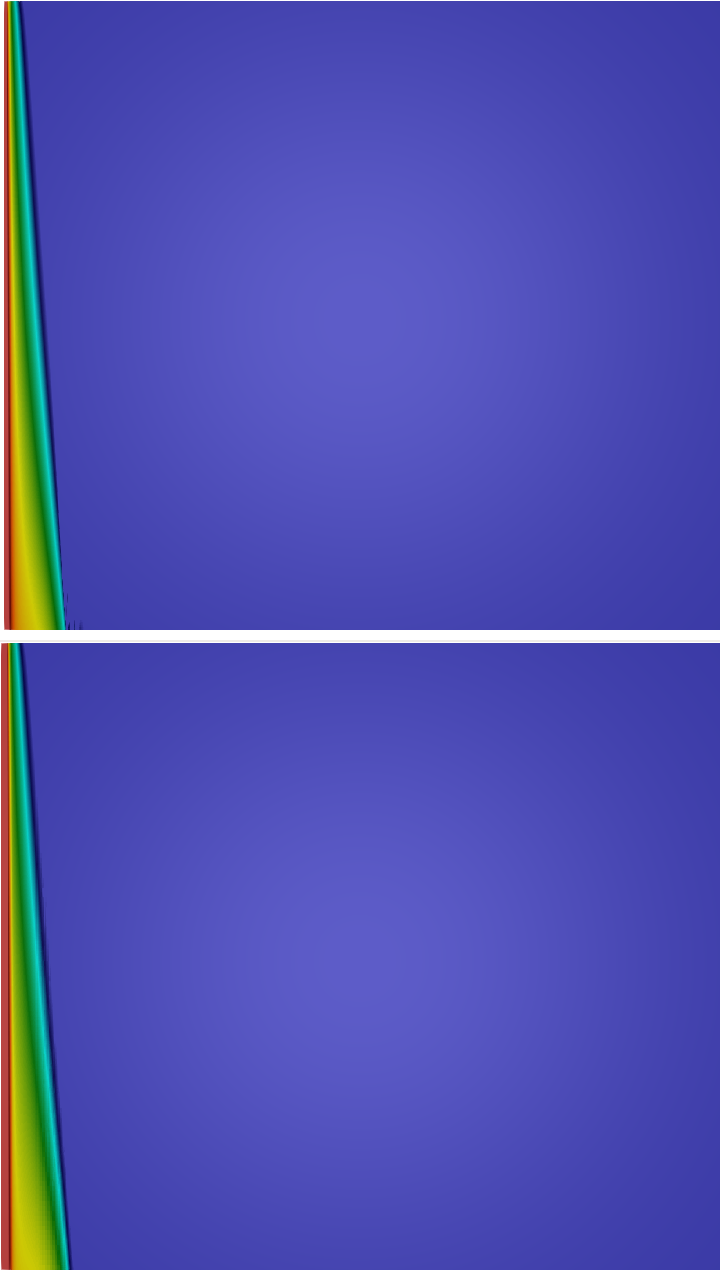}{\figWidth}};
\draw(5,   -.5) node[anchor=south west,xshift=-4pt,yshift=+0pt] {\trimfig{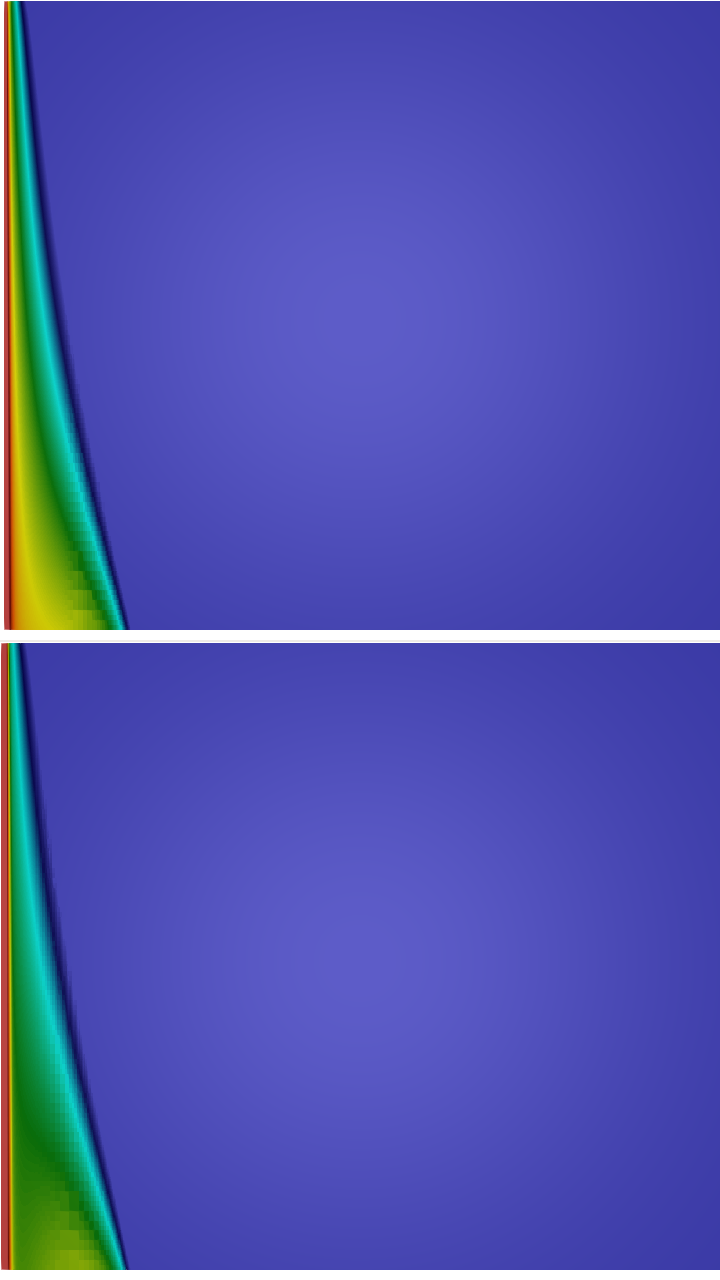}{\figWidth}};
\draw(7.45,-.5) node[anchor=south west,xshift=-4pt,yshift=+0pt] {\trimfigb{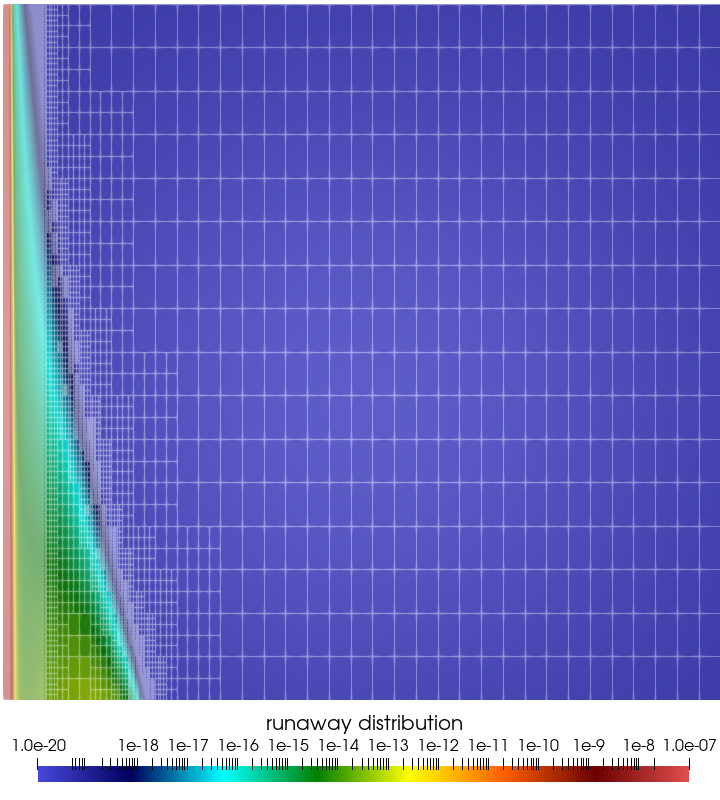}{4.7cm}};
\end{tikzpicture}
\end{center}
  \caption{\editBOne{}{%
    Left: Distribution functions at $t=0$, 2 and 4. The top row uses the
    Dirichlet boundary condition at $p=3\hat{v}_t$ and the bottom row uses the
    compatible boundary condition at $p=0$.  Right: The adaptive mesh at the
    final time.}}
  \label{fig:pzero}
\end{figure}
}

\editBOne{}{
In this test, we provide an extension of the boundary up to $p=0$. This study uses a primary runaway electron test, in which the initial condition is given as
\[
f_0(p,\xi) = \frac{1}{\hat{v}_t^3 \pi^{\frac{3}{2}} } \exp{\left(\frac{1-\sqrt{1+p^2}}{\hat{v}_t^2/2}\right)},
\]
where the normalized thermal velocity is taken as $\hat{v}_t = 0.1$.
The runaway electron in the tail region will grow exponentially due to the driving electric field. We select an electric field of $E=3$, which is well above the runaway electron threshold.
We extended the idea given in~\cite{mohseni2000numerical} to handle the boundary
condition at $p=0$ with AMR.
Using the expansion of $f(p, \xi)$ with azimuthal symmetry, it can be derived that the compatible boundary condition for distribution at $p=0$ should be
\begin{equation}
  \label{eq:p-zero-bc}
  f(-p,-\xi) = f(p,\xi).
\end{equation}
This condition is used to fill in the ghost point at the negative $p$ location.
We utilize the connectivity of trees of the p4est library to connect two coarse
quadrants in $\xi$-direction along the $p=0$ quadrant face; and we also flip the
orientation of this connectivity to account for the alternating sign of the
$\xi$ argument in \eqref{eq:p-zero-bc}.  Therefore, the mesh adaptivity
along the $p=0$ boundary can be treated consistently with all other cell
boundaries with varying levels of refinement.
The results of two different types of boundary conditions are presented in Figure~\ref{fig:pzero}.
The physical locations are perfectly aligned in the figure for easy comparison.
It is clear to see that the boundary layer of $p=0$ is thicker than that of $p=3\hat{v}_t$, but they match well starting from $p=3\hat{v}_t $,
despite different types of boundary conditions being used.
It is also observed that the shapes of runaway tails match well throughout the entire run.
At the final time, it is observed that the details of runaway tails are slightly different,
which is due to the impact of different boundary conditions.
In our forthcoming physics study, we will document the specifics of runaway growth rates in the two different setups.
}

\subsection{Interaction between a Maxwellian and a runaway tail}

We next present a physics-motivated  example to demonstrate the AMR capability for resolving both  of a bulk Maxwellian and a runaway distribution tail.
Our aim is to study the algorithm performance in a practical setting, and thus we design a test that is still relatively simple.
The problem is initialized with a Maxwellian uniformly in $\xi$  under a small perturbation in the tail, which is centered around $(p, \xi) = (40, -0.9)$,
\[
f_0(p,\xi) = \frac{1}{\hat{v}_t^3 \pi^{\frac{3}{2}} } \exp\editBOne{}{\left(\frac{1-\sqrt{1+p^2}}{\hat{v}_t^2/2}\right)} + 10^{-15} \exp\left(\frac{(p-40)^2}{25} \right) \exp\left(\frac{-(\xi+0.9)^2}{0.0025} \right),
\]
where the normalized thermal velocity is taken as $\hat{v}_t = 0.1$. A Neumann boundary condition is applied at the right boundary, while the left boundary uses a Dirichlet boundary condition from
the initial Maxwellian. The computational domain is $[0.3, 60]\times[-1, 1]$ as we let $p_{\min} = 3 \, \hat{v}_t$ for which the Dirichlet is still a good approximation.  We choose a numerical scheme based on the QUICK finite difference scheme and an  implicit second-order ESDIRK integrator through PETSc's TS interface.

The proposed AMR algorithm and indicators are tested using this problem.
The coarsest mesh in the AMR algorithm is chosen to be $48\times 8$, and we use a maximum of 6 more levels of refinement.
Initially, we enforce the AMR algorithm to refine around the left and bottom boundaries for better resolving the Maxwellian bulk and the large portion of the perturbation (see the top row in Figure~\ref{fig:chiu1}).
After the initial indicator is computed from the given solution, the proposed AMR indicator prediction is used to evolve those indicator values over time.
The mesh is checked and updated every 6 time steps. The averaged number of DOFs
in the whole simulation is only \editOne{287,684}{71,156, which amounts to
a reduction in DOFs of $\sim$4.5\%
with the same finest level as the adaptive mesh.}
The previous work~\cite{GuoMcDevittTang2017} needed \editOne{several}{0.62}
million DOFs through a structured stretched grid to have a comparable
resolution, \editOne{}{which again highlights the computational savings obtained
with AMR.}

{
\newcommand{\figWidth}{12cm}
\newcommand{\trimfig}[2]{\trimFig{#1}{#2}{.0}{.0}{0}{0}}
\begin{figure}[tb]
\begin{center}
\begin{tikzpicture}[scale=1]
\useasboundingbox (0.0,0) rectangle (12,11.7);  %
\draw(0,9.2) node[anchor=south west,xshift=-4pt,yshift=+0pt] {\trimfig{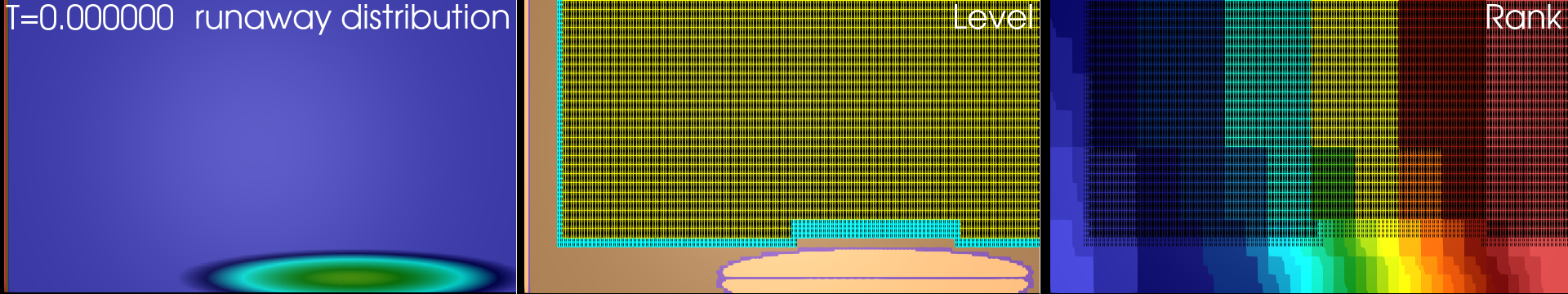}{\figWidth}};
\draw(0,6.9) node[anchor=south west,xshift=-4pt,yshift=+0pt] {\trimfig{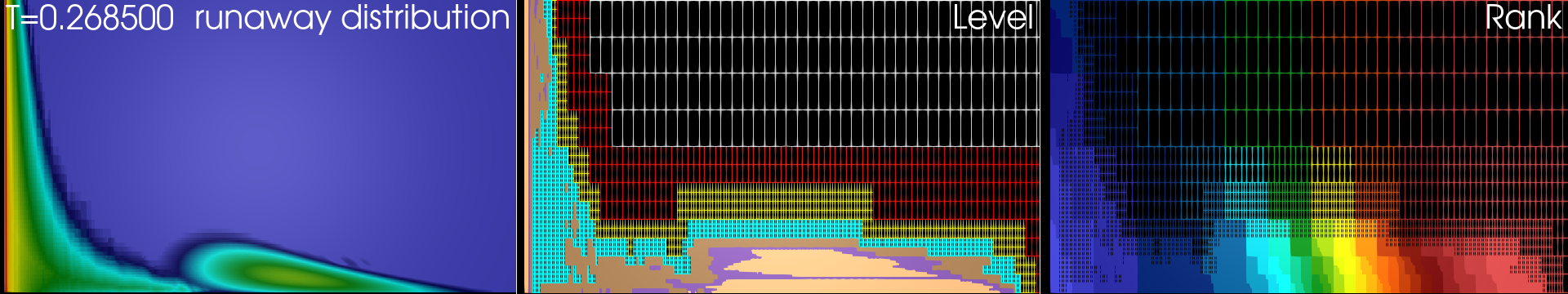}{\figWidth}};
\draw(0,4.6) node[anchor=south west,xshift=-4pt,yshift=+0pt] {\trimfig{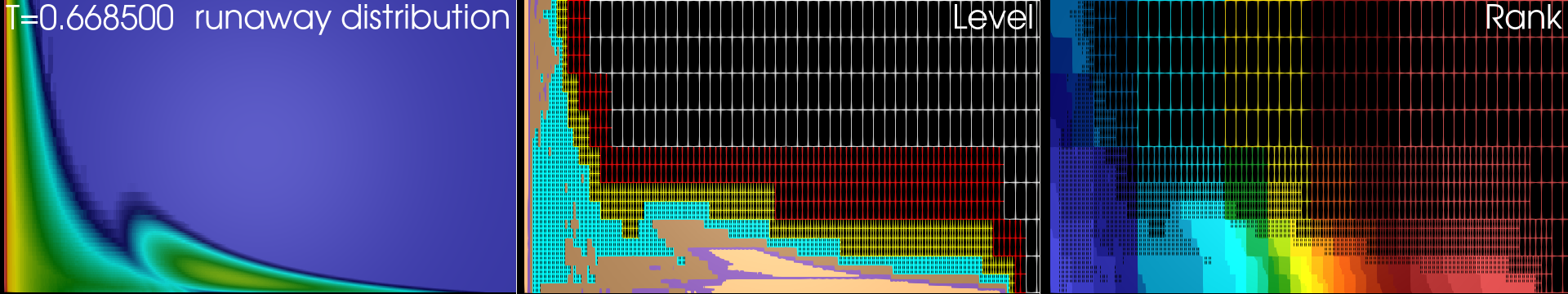}{\figWidth}};
\draw(0,2.3) node[anchor=south west,xshift=-4pt,yshift=+0pt] {\trimfig{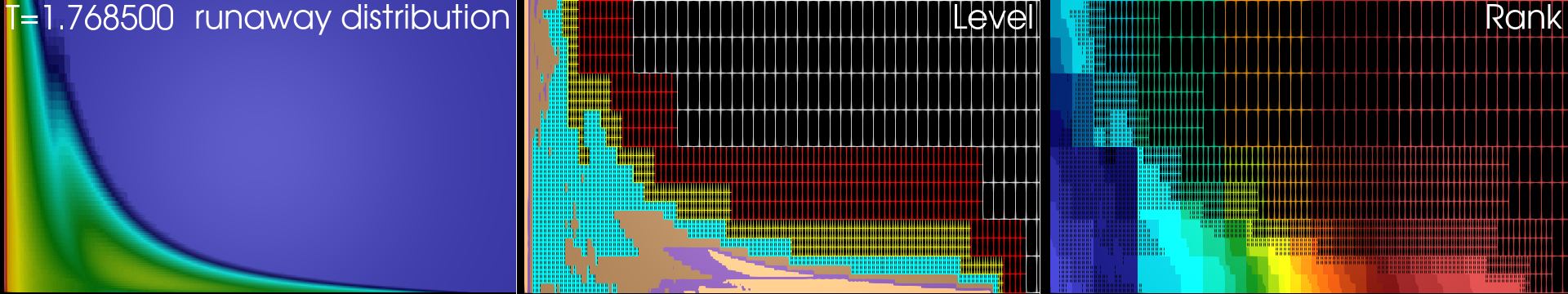}{\figWidth}};
\draw(0,-.62) node[anchor=south west,xshift=-4pt,yshift=+0pt] {\trimfig{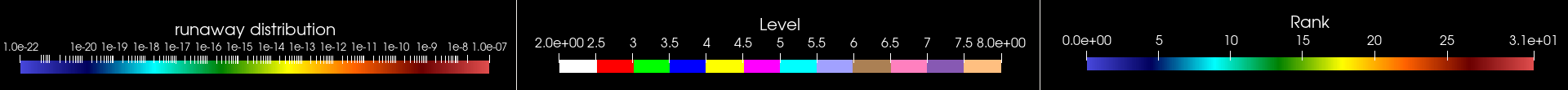}{\figWidth}};
\draw(0,0) node[anchor=south west,xshift=-4pt,yshift=+0pt] {\trimfig{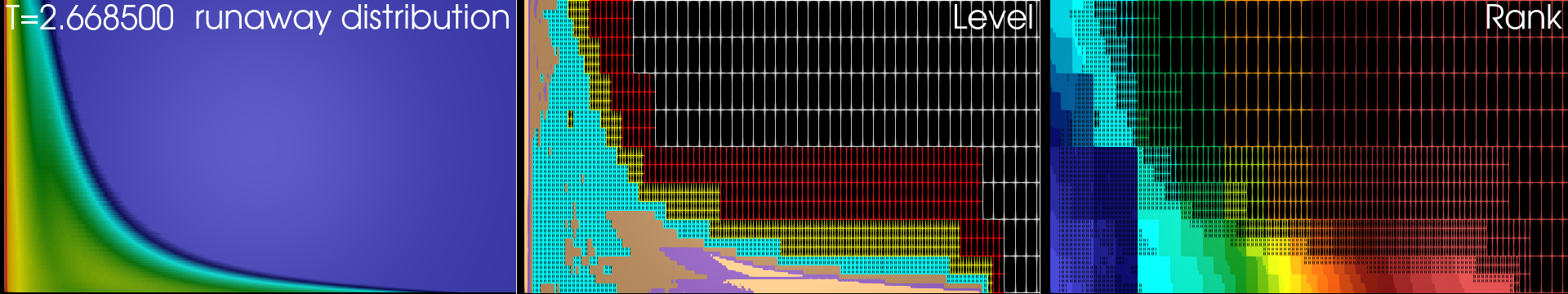}{\figWidth}};
\end{tikzpicture}
\end{center}
  \caption{Distribution functions (left), mesh and AMR levels (middle), and MPI ranks (right) over time.}
  \label{fig:chiu1}
\end{figure}
}

The numerical results are presented in Figure~\ref{fig:chiu1}. The distribution solutions along with the adaptive meshes are shown.
We note that the Maxwellian bulk and the tail perturbation are slowly merging
into each other due to the interaction between the driven electric field, the
relativistic Fokker--Planck collision, synchrotron radiation,
and the secondary knock-on source. Here we let the damping strength be $\alpha=0.1$ and turn on the partial screening effects. We note that our AMR algorithm captures the interesting features in the solution very well.
The finest mesh is around the region where it has a large gradient, which is either near the Maxwellian boundary layer or near the runaway tail.
It is also critical for this test that the solver be capable of resolving a large variation of the solution. The solution varies from $10^{-2}$ (the bulk Maxwellian boundary layer) to $10^{-20}$ (the runaway tail).
In order to demonstrate the large variation, the distribution function in Figure~\ref{fig:chiu1} is presented in the log scale, which shows that our adaptive solver is capable of resolving both regions.
To demonstrate the dynamic load-balancing, we include the MPI ranks  in the figure.
Note that this test is designed to be relatively simple so that 32 processors are sufficient.

{
\newcommand{\drawPlot}[4]{%
\begin{scope}[#1]
\draw(0,0) node[anchor=south west,xshift=-4pt,yshift=+0pt] {\trimfig{figs/chiu2/int#2}{\figWidth}};
\draw(3.5,.7) node[draw,fill=white,anchor=east,xshift=2pt,yshift=0pt,scale=0.8]{\scriptsize #3};
\end{scope}
}

\newcommand{\figWidth}{4.cm}
\newcommand{\trimfig}[2]{\trimFig{#1}{#2}{.0}{.0}{0}{0}}
\newcommand{\trimfigb}[2]{\trimFig{#1}{#2}{.0}{.0}{0}{0.1}}
\begin{figure}[htb]
\begin{center}
\begin{tikzpicture}[scale=1]
\useasboundingbox (0.0,-.5) rectangle (12,5.4);  %
\drawPlot{xshift= -.1cm,yshift=2.2cm}{0}{$t=0$}{};
\drawPlot{xshift= 4cm,yshift=2.2cm}{1}{$t=0.269$ }{};
\drawPlot{xshift= 8.1cm,yshift=2.2cm}{2}{ $t=0.669$}{};
\drawPlot{xshift= 1.95cm,yshift=-1cm}{3}{ $t=1.769$}{};
\drawPlot{xshift= 6.05cm,yshift=-1cm}{4}{ $t=2.669$}{};
\end{tikzpicture}
\end{center}
  \caption{Runaway electron distributions in $p$ over time.}
  \label{fig:chiu_runaway}
\end{figure}
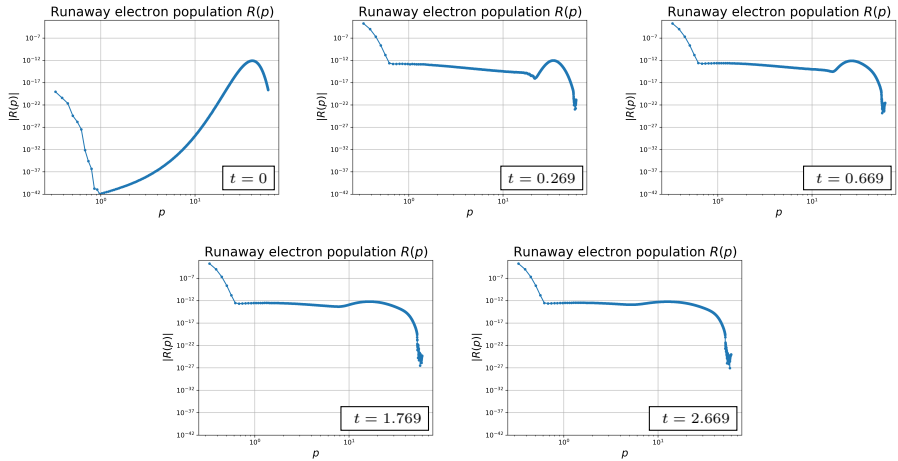
}

In addition to the distribution, we evaluate the so-called runaway electron population $R(p, t)$,
\begin{align*}
R(p, t) :=  \int_{\xi=-1}^{\xi=1} f(p,\xi, t)  v_\parallel 2\pi p^2  \, d\xi,
\end{align*}
where the parallel velocity is defined as  $v_\parallel = {p \, \xi}/{\gamma \, m_e}$.
This is a distribution in the $p$ direction to measure the runaway electron strength, which is a good indicator of the number of high-energy runaway electrons in the system.
Figure~\ref{fig:chiu_runaway} shows that initially there is a large portion of runaway electrons due to the perturbation and that they are slowly reduced and merged into the bulk.

{
\newcommand{\trimfig}[2]{\trimFig{#1}{#2}{.0}{.0}{0}{0}}
\begin{figure}[htb]
\begin{center}
\begin{tikzpicture}[scale=1]
\useasboundingbox (0,-.5) rectangle (12,3.8);
\draw(0,-.9) node[anchor=south west,xshift=-4pt,yshift=+0pt] {\trimfig{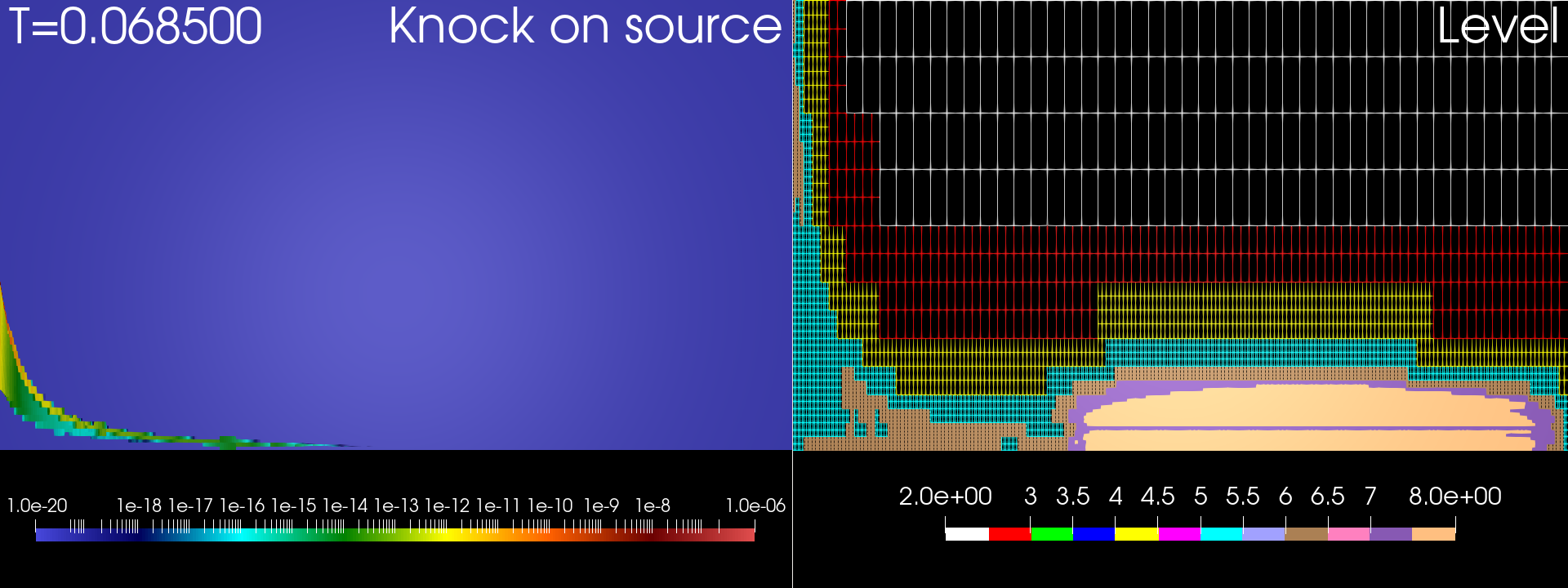}{12cm}};
\end{tikzpicture}
\end{center}
  \caption{Knock-on electron source at the first checkpoint.}
  \label{fig:chiu2}
\end{figure}
}

We further present the knock-on electron source term in Figure~\ref{fig:chiu2}. We note that the knock-on source is exactly in the range of $\xi \in [-\sqrt{\gamma/(\gamma+1)}, -p/(\gamma+1)]$ as we indicated in Section~\ref{sec:knock_on}.
The mesh plot of Figure~\ref{fig:chiu2} shows that the AMR algorithm is able to capture the source term in this narrow region through adjusting the mesh on the fly.

\section{Conclusion}
\label{sec:conclusion}

In this work we designed and developed adaptive, scalable, fully implicit solvers for the relativistic Fokker--Planck equation in phase space.
One key goal of the work is to develop a scalable and efficient dynamic AMR.
In practice, one needs to determine the requirements on scalability and then
adjust the parameters of the AMR algorithms in terms of refinement frequencies
and the AMR prediction horizon.  The present work proposes a new prediction strategy for refinement indicators,
giving a potential tool to exert control
over accuracy at computational overheads that were demonstrated to be low.
Numerical experiments quantify the predictive approach behaving better than conventional indicators, which lack prediction, in terms of resolving features in the solutions.
When the need for scalability is of higher importance to the application, one can use the
proposed AMR prediction at the expense of further increases of computational
overheads from finer meshes.

We contribute a new scalable implementation of the proposed algorithm using the p4est and PETSc frameworks.
Although the focus of the current work is to develop a solver aiming specifically at runaway electrons during tokamak disruptions,
a large portion of our proposed algorithm is general and thus can be potentially applied to solving many other time-dependent PDEs with dynamic mesh adaptivity.
Several numerical examples are presented in order to verify the solver's accuracy, scalability, and efficiency in adaptivity.
These numerical results confirm the parallel and algorithmic scalabilities and accuracy of the schemes.
A significant improvement of computational cost owing to the AMR algorithm was demonstrated using a practical disruption study.

Future work includes generalizing the fixed external field in the current model to a self-consistent model that evolves the electric field through involving the runaway current.
\editCOne{Another direction to pursue is to extend the computational domain to
the low-energy region, which needs a careful treatment of the boundary condition
at $p=0$.}{}

\section*{Acknowledgment}
This research used resources provided by the Los Alamos National Laboratory Institutional Computing Program, which is supported by the U.S. Department of Energy National Nuclear Security Administration under Contract No.~89233218CNA000001,
and the National Energy Research Scientific Computing Center
(NERSC), a U.S. Department of Energy Office of Science User Facility located at Lawrence Berkeley National Laboratory, operated under Contract No. DE-AC02-05CH11231 using NERSC award FES-ERCAP0021219.

This research was funded in part and used resources of the Argonne Leadership Computing Facility, which is a DOE Office of Science User Facility supported under Contract No. DE-AC02–06CH11357.

The Texas Advanced Computing Center (TACC) at The University of Texas at Austin
provided HPC resources that have contributed to the research results reported
within this paper.

\clearpage
\appendix

\section{Guiding center coordinate transformation}
\label{sec:coordinate}

The coordinate transformation from
$(\vec{X}, p_\parallel, \vec{p}_\perp)$ to $(\vec{X}, p, \xi)$ is described in this section. The discussion focuses on the phase space.
The differential element for momentum space integration is
\begin{align}
d^3{\vec{p}} = dp_\parallel d^2{\vec{p}_\perp} = dp_\parallel (p_\perp dp_\perp d\theta).
\end{align}
The guiding center model assumes azimuthal symmetry. Thus one can integrate over $\theta\in (0,2\pi)$ (here $p_\perp$ is positive), giving
\begin{align}
\int_\theta d^3\vec{p} = 2\pi p_\perp d p_\parallel d p_\perp.
\end{align}
The transformation from $(p_\parallel, p_\perp)$ to $(p,\xi)$ coordinate is given by
\begin{align}
p & = \sqrt{p_\parallel^2 + p_\perp^2} \\
\xi & = \frac{p_\parallel}{p} = \frac{p_\parallel}{\sqrt{p_\parallel^2+p_\perp^2}}.
\end{align}
The Jacobian of this transformation is
\begin{align}
\frac{\partial (p_\parallel, p_\perp)}{\partial (p,\xi)}
=
\left|\frac{\partial p_\parallel}{\partial p} \frac{\partial p_\perp}{\partial\xi}
- \frac{\partial p_\parallel}{\partial\xi} \frac{\partial p_\perp}{\partial p}\right|
= \left|- \xi \frac{p\xi}{\sqrt{1-\xi^2}} - p\sqrt{1-\xi^2}\right|
= \frac{p}{\sqrt{1-\xi^2}}.
\end{align}
One then has
\begin{align}
\int_\theta d^3\vec{p} = 2\pi p_\perp d p_\parallel d p_\perp
= 2\pi p\sqrt{1-\xi^2} \frac{p}{\sqrt{1-\xi^2}} dp d\xi
= 2\pi p^2 dp d\xi .
\end{align}
Noting that $\xi\in (-1,1),$ the integration over the entire momentum space is simply
$(4\pi/3) p^3,$ as expected.
This also indicates  that the overall Jacobian of $(p,\xi,\theta)$ coordinates is $p^2,$
for gyro-angle independent problems, and the overall Jacobian for the $(p,\xi)$ coordinate
is $2\pi p^2.$ For convenience, we will be using $J=p^2$ for the $(p,\xi)$ coordinates.

\section{Summary of physical quantities in RFP}
\label{sec:quantities}
Table~\ref{table:notation} summarizes some physical quantities and the corresponding notations  in this paper.

\begin{table}
\caption{Physical quantities in the RFP model and the notations used in the current work.\label{table:notation}}
\centering
\vskip 1ex
\footnotesize
\begin{tabular}{lc|lc}
\toprule
  \thead{Description} & \thead{Symbol} &  \thead{Description} & \thead{Symbol} \\
\midrule
  time     & $t$   & electron  charge        & $e$\\
  configuration space    & $\xv$ & electron mass & $m_e$  \\
  phase space & $\pv$  &  speed of light         & $c$ \\
  phase space parallel to $\Bv$      & $p_\parallel$     &  Lorentz factor        & $\gamma$      \\
  phase space perpendicular to $\Bv$ & $p_\bot$      &  electrical permittivity & $\epsilon_0$           \\
  magnitude of the momentum       & $p=\norm{\pv}$       & Coulomb logarithm & $\ln \Lambda$      \\
  pitch angle                     & $\xi=p_\parallel/p$       &   thermal velocity  & $v_t$       \\
  runaway election density        & $f$        & effective charge number & $Z$               \\
  energy flux                     & $\Gamma_p$         & collision time scale                     & $\tau_c$      \\
  pitch angle flux                & $\Gamma_\xi$     & synchrotron  time scale                 & $\tau_s$        \\
  magnetic field                  & $\Bv$       &      intensity of damping    & $\alpha=\tau_c/\tau_s$             \\
  induced electric field  parallel to $\Bv$        & $E_\parallel$      &     M{\o}ller cross section & ${d\sigma}/{dp}$    \\
\bottomrule
\end{tabular}
\end{table}

\clearpage
\bibliography{references}

\begin{thebibliography}{10}
\expandafter\ifx\csname url\endcsname\relax
  \def\url#1{\texttt{#1}}\fi
\expandafter\ifx\csname urlprefix\endcsname\relax\def\urlprefix{URL }\fi
\expandafter\ifx\csname href\endcsname\relax
  \def\href#1#2{#2} \def\path#1{#1}\fi

\bibitem{boozer-pop-2015}
A.~H. Boozer, Theory of runaway electrons in {ITER}: Equations, important
  parameters, and implications for mitigation, Physics of Plasmas 22~(3) (2015)
  032504.

\bibitem{breizman2019physics}
B.~N. Breizman, P.~Aleynikov, E.~M. Hollmann, M.~Lehnen, Physics of runaway
  electrons in tokamaks, Nuclear Fusion 59~(8) (2019) 083001.

\bibitem{GuoMcDevittTang2017}
Z.~Guo, C.~J. McDevitt, X.-Z. Tang, Phase-space dynamics of runaway electrons
  in magnetic fields, Plasma Physics and Controlled Fusion 59~(4) (2017)
  044003.

\bibitem{stahl2017norse}
A.~Stahl, M.~Landreman, O.~Embr{\'e}us, T.~F{\"u}l{\"o}p, {NORSE}: A solver for
  the relativistic non-linear {F}okker--{P}lanck equation for electrons in a
  homogeneous plasma, Computer Physics Communications 212 (2017) 269--279.

\bibitem{hesslow2018effect}
L.~Hesslow, O.~Embr{\'e}us, G.~J. Wilkie, G.~Papp, T.~F{\"u}l{\"o}p, Effect of
  partially ionized impurities and radiation on the effective critical electric
  field for runaway generation, Plasma Physics and Controlled Fusion 60~(7)
  (2018) 074010.

\bibitem{daniel2020fully}
D.~Daniel, W.~T. Taitano, L.~Chac{\'o}n, A fully implicit, scalable,
  conservative nonlinear relativistic {F}okker--{Planck} {0D-2P} solver for
  runaway electrons, Computer Physics Communications 254 (2020) 107361.

\bibitem{strauss1998adaptive}
H.~Strauss, D.~Longcope, An adaptive finite element method for
  magnetohydrodynamics, Journal of Computational Physics 147~(2) (1998)
  318--336.

\bibitem{philip2008implicit}
B.~Philip, L.~Chac{\'o}n, M.~Pernice, Implicit adaptive mesh refinement for
  {2D} reduced resistive magnetohydrodynamics, Journal of Computational Physics
  227~(20) (2008) 8855--8874.

\bibitem{baty2019finmhd}
H.~Baty, {FINMHD}: An adaptive finite-element code for magnetic reconnection
  and formation of plasmoid chains in magnetohydrodynamics, The Astrophysical
  Journal Supplement Series 243~(2) (2019) 23.

\bibitem{peng2020adaptive}
Z.~Peng, Q.~Tang, X.-Z. Tang, An adaptive discontinuous {Petrov--Galerkin
  method for the Grad--Shafranov equation}, SIAM Journal on Scientific
  Computing 42~(5) (2020) B1227--B1249.

\bibitem{tang2022adaptive}
Q.~Tang, L.~Chac{\'o}n, T.~V. Kolev, J.~N. Shadid, X.-Z. Tang, An adaptive
  scalable fully implicit algorithm based on stabilized finite element for
  reduced visco-resistive {MHD}, Journal of Computational Physics 454 (2022)
  110967.

\bibitem{hittinger2013block}
J.~A. Hittinger, J.~W. Banks, Block-structured adaptive mesh refinement
  algorithms for {Vlasov} simulation, Journal of Computational Physics 241
  (2013) 118--140.

\bibitem{adams2017landau}
M.~F. Adams, E.~Hirvijoki, M.~G. Knepley, J.~Brown, T.~Isaac, R.~Mills, Landau
  collision integral solver with adaptive mesh refinement on emerging
  architectures, SIAM Journal on Scientific Computing 39~(6) (2017) C452--C465.

\bibitem{wettervik2017relativistic}
B.~S. Wettervik, T.~C. DuBois, E.~Siminos, T.~F{\"u}l{\"o}p, Relativistic
  {Vlasov--Maxwell} modelling using finite volumes and adaptive mesh
  refinement, The European Physical Journal D 71~(6) (2017) 1--14.

\bibitem{arslanbekov2013kinetic}
R.~R. Arslanbekov, V.~I. Kolobov, A.~A. Frolova, Kinetic solvers with adaptive
  mesh in phase space, Physical Review E 88~(6) (2013) 063301.

\bibitem{kolobov2019boltzmann}
V.~Kolobov, R.~Arslanbekov, D.~Levko, Boltzmann-fokker-planck kinetic solver
  with adaptive mesh in phase space, in: AIP Conference Proceedings, Vol. 2132,
  AIP Publishing, 2019.

\bibitem{berger1989local}
M.~J. Berger, P.~Colella, Local adaptive mesh refinement for shock
  hydrodynamics, Journal of Computational Physics 82~(1) (1989) 64--84.

\bibitem{Constantinescu_A2008a}
E.~Constantinescu, A.~Sandu, G.~Carmichael, Modeling atmospheric chemistry and
  transport with dynamic adaptive resolution, Computational Geosciences 12~(2)
  (2008) 133--151.

\bibitem{offermans2022error}
N.~Offermans, D.~Massaro, A.~Peplinski, P.~Schlatter, Error-driven adaptive
  mesh refinement for unsteady turbulent flows in spectral-element simulations,
  Computers \& Fluids (2022) 105736.

\bibitem{GuoCheng17}
W.~Guo, Y.~Cheng, An adaptive multiresolution discontinuous galerkin method for
  time-dependent transport equations in multidimensions, SIAM Journal on
  Scientific Computing 39~(6) (2017) A2962--A2992.

\bibitem{BursteddeWilcoxGhattas11}
C.~Burstedde, L.~C. Wilcox, O.~Ghattas, {\texttt{p4est}}: {S}calable algorithms
  for parallel adaptive mesh refinement on forests of octrees, SIAM Journal on
  Scientific Computing 33~(3) (2011) 1103--1133.

\bibitem{LeVeque02}
R.~J. LeVeque, Finite Volume Methods for Hyperbolic Problems, Cambridge Texts
  in Applied Mathematics, Cambridge University Press, 2002.

\bibitem{Toro09}
E.~F. Toro, Riemann Solvers and Numerical Methods for Fluid Dynamics, Springer
  Berlin, Heidelberg, 2009.

\bibitem{Leonard79}
B.~P. Leonard, A stable and accurate convective modelling procedure based on
  quadratic upstream interpolation, Computer Methods in Applied Mechanics and
  Engineering 19~(1) (1979) 59--98.

\bibitem{petsc-user-ref}
S.~Balay, S.~Abhyankar, M.~F. Adams, S.~Benson, J.~Brown, P.~Brune,
  K.~Buschelman, E.~Constantinescu, L.~Dalcin, A.~Dener, V.~Eijkhout,
  J.~Faibussowitsch, W.~D. Gropp, V.~Hapla, T.~Isaac, P.~Jolivet, D.~Karpeev,
  D.~Kaushik, M.~G. Knepley, F.~Kong, S.~Kruger, D.~A. May, L.~C. McInnes,
  R.~T. Mills, L.~Mitchell, T.~Munson, J.~E. Roman, K.~Rupp, P.~Sanan,
  J.~Sarich, B.~F. Smith, S.~Zampini, H.~Zhang, H.~Zhang, J.~Zhang, {PETSc/TAO}
  users manual, Tech. Rep. ANL-21/39 -- Revision 3.18, Argonne National
  Laboratory (2022).

\bibitem{alexander1977diagonally}
R.~Alexander, Diagonally implicit {R}unge--{K}utta methods for stiff {ODE's},
  SIAM Journal on Numerical Analysis 14~(6) (1977) 1006--1021.

\bibitem{giraldo2013implicit}
F.~X. Giraldo, J.~F. Kelly, E.~M. Constantinescu, Implicit-explicit
  formulations of a three-dimensional nonhydrostatic unified model of the
  atmosphere (numa), SIAM Journal on Scientific Computing 35~(5) (2013)
  B1162--B1194.

\bibitem{knoll2004jacobian}
D.~A. Knoll, D.~E. Keyes, Jacobian-free {Newton--Krylov} methods: a survey of
  approaches and applications, Journal of Computational Physics 193~(2) (2004)
  357--397.

\bibitem{hypre}
{\sl hypre}: High performance preconditioners,
  \url{https://llnl.gov/casc/hypre},
  \url{https://github.com/hypre-space/hypre}.

\bibitem{FalgoutYang02}
R.~D. Falgout, U.~M. Yang, hypre: A library of high performance
  preconditioners, in: P.~M.~A. Sloot, A.~G. Hoekstra, C.~J.~K. Tan, J.~J.
  Dongarra (Eds.), Computational Science --- ICCS 2002, Springer Berlin
  Heidelberg, Berlin, Heidelberg, 2002, pp. 632--641.

\bibitem{brizard1999nonlinear}
A.~J. Brizard, A.~A. Chan, Nonlinear relativistic gyrokinetic
  {V}lasov--{M}axwell equations, Physics of Plasmas 6~(12) (1999) 4548--4558.

\bibitem{connor1975relativistic}
J.~Connor, R.~Hastie, Relativistic limitations on runaway electrons, Nuclear
  Fusion 15~(3) (1975) 415.

\bibitem{Guo-etal-PoP-2019}
Z.~Guo, C.~Mcdevitt, X.~Tang, Toroidal effect on runaway vortex and avalanche
  growth rate, Physics of Plasmas 26~(8) (2019) 082503.

\bibitem{papp2011runaway}
G.~Papp, M.~Drevlak, T.~F{\"u}l{\"o}p, P.~Helander, Runaway electron drift
  orbits in magnetostatic perturbed fields, Nuclear Fusion 51~(4) (2011)
  043004.

\bibitem{hesslow2017effect}
L.~Hesslow, O.~Embr{\'e}us, A.~Stahl, T.~C. DuBois, G.~Papp, S.~L. Newton,
  T.~F{\"u}l{\"o}p, Effect of partially screened nuclei on fast-electron
  dynamics, Physical Review Letters 118~(25) (2017) 255001.

\bibitem{Rosenbluth_1997}
M.~Rosenbluth, S.~Putvinski, Theory for avalanche of runaway electrons in
  tokamaks, Nuclear Fusion 37~(10) (1997) 1355.

\bibitem{chiu1998fokker}
S.~Chiu, M.~Rosenbluth, R.~Harvey, V.~Chan, {F}okker--{P}lanck simulations mylb
  of knock-on electron runaway avalanche and bursts in tokamaks, Nuclear Fusion
  38~(11) (1998) 1711.

\bibitem{mcdevitt2019avalanche}
C.~J. McDevitt, Z.~Guo, X.-Z. Tang, Avalanche mechanism for runaway electron
  amplification in a tokamak plasma, Plasma Physics and Controlled Fusion
  61~(5) (2019) 054008.

\bibitem{IsaacBursteddeWilcoxEtAl15}
T.~Isaac, C.~Burstedde, L.~C. Wilcox, O.~Ghattas, Recursive algorithms for
  distributed forests of octrees, SIAM Journal on Scientific Computing 37~(5)
  (2015) C497--C531.

\bibitem{sundar2008bottom}
H.~Sundar, R.~S. Sampath, G.~Biros, Bottom-up construction and 2:1 balance
  refinement of linear octrees in parallel, SIAM Journal on Scientific
  Computing 30~(5) (2008) 2675--2708.

\bibitem{BursteddeGhattasGurnisEtAl10}
C.~Burstedde, O.~Ghattas, M.~Gurnis, T.~Isaac, G.~Stadler, T.~Warburton, L.~C.
  Wilcox, Extreme-scale {AMR}, in: SC10: Proceedings of the International
  Conference for High Performance Computing, Networking, Storage and Analysis,
  ACM/IEEE, 2010.

\bibitem{SundarBirosBursteddeEtAl12}
H.~Sundar, G.~Biros, C.~Burstedde, J.~Rudi, O.~Ghattas, G.~Stadler, Parallel
  geometric-algebraic multigrid on unstructured forests of octrees, in: SC12:
  Proceedings of the International Conference for High Performance Computing,
  Networking, Storage and Analysis, ACM/IEEE, 2012.

\bibitem{rudi2015extreme}
J.~Rudi, A.~C.~I. Malossi, T.~Isaac, G.~Stadler, M.~Gurnis, P.~W. Staar,
  Y.~Ineichen, C.~Bekas, A.~Curioni, O.~Ghattas, An extreme-scale implicit
  solver for complex {PDEs}: highly heterogeneous flow in earth's mantle, in:
  SC15: Proceedings of the International Conference for High Performance
  Computing, Networking, Storage and Analysis, ACM, 2015, pp. 1--12.

\bibitem{RudiStadlerGhattas17}
J.~Rudi, G.~Stadler, O.~Ghattas, Weighted {BFBT} preconditioner for {S}tokes
  flow problems with highly heterogeneous viscosity, SIAM Journal on Scientific
  Computing 39~(5) (2017) S272--S297.

\bibitem{RudiShihStadler20}
J.~Rudi, Y.-H. Shih, G.~Stadler, Advanced {N}ewton methods for geodynamical
  models of {S}tokes flow with viscoplastic rheologies, Geochemistry,
  Geophysics, Geosystems 21~(9) (2020).

\bibitem{weinzierl2019peano}
T.~Weinzierl, The {Peano} software---parallel, automaton-based, dynamically
  adaptive grid traversals, ACM Transactions on Mathematical Software (TOMS)
  45~(2) (2019) 1--41.

\bibitem{FernandoNeilsenLimEtAl19}
M.~Fernando, D.~Neilsen, H.~Lim, E.~Hirschmann, H.~Sundar, Massively parallel
  simulations of binary black hole intermediate-mass-ratio inspirals, SIAM
  Journal on Scientific Computing 41~(2) (2019) C97--C138.

\bibitem{yang2002boomeramg}
V.~E. Henson, U.~M. Yang, {BoomerAMG}: A parallel algebraic multigrid solver
  and preconditioner, Applied Numerical Mathematics 41~(1) (2002) 155--177.

\bibitem{banks2017stable}
J.~W. Banks, W.~D. Henshaw, D.~W. Schwendeman, Q.~Tang, A stable partitioned
  {FSI} algorithm for rigid bodies and incompressible flow. {Part I}: Model
  problem analysis, Journal of Computational Physics 343 (2017) 432--468.

\bibitem{manteuffel2018nonsymmetric}
T.~A. Manteuffel, J.~Ruge, B.~S. Southworth, Nonsymmetric algebraic multigrid
  based on local approximate ideal restriction {($\ell$AIR)}, SIAM Journal on
  Scientific Computing 40~(6) (2018) A4105--A4130.

\bibitem{mcdevitt2018relation}
C.~J. McDevitt, Z.~Guo, X.-Z. Tang, Relation of the runaway avalanche threshold
  to momentum space topology, Plasma Physics and Controlled Fusion 60~(2)
  (2018) 024004.

\bibitem{Burgers1948}
J.~Burgers, A mathematical model illustrating the theory of turbulence, Vol.~1
  of Advances in Applied Mechanics, Elsevier, 1948, pp. 171--199.

\bibitem{mohseni2000numerical}
K.~Mohseni, T.~Colonius, Numerical treatment of polar coordinate singularities,
  Journal of Computational Physics 157~(2) (2000) 787--795.

\end{thebibliography}

 \begin{center}
	\scriptsize \framebox{\parbox{4in}{Government License (will be removed at publication):
			The submitted manuscript has been created by UChicago Argonne, LLC,
			Operator of Argonne National Laboratory (``Argonne").  Argonne, a
			U.S. Department of Energy Office of Science laboratory, is operated
			under Contract No. DE-AC02-06CH11357.  The U.S. Government retains for
			itself, and others acting on its behalf, a paid-up nonexclusive,
			irrevocable worldwide license in said article to reproduce, prepare
			derivative works, distribute copies to the public, and perform
			publicly and display publicly, by or on behalf of the Government. The Department of Energy will provide public access to these results of federally sponsored research in accordance with the DOE Public Access Plan. http://energy.gov/downloads/doe-public-access-plan.
}}
	\normalsize
\end{center}

\end{document}